\documentclass[12pt]{article}

\usepackage{setspace,enumerate}
\usepackage[margin=1in]{geometry}
\usepackage{parskip}
\usepackage[sc]{mathpazo}
\usepackage[T1]{eulervm}
\usepackage[scaled]{helvet}

\usepackage{amsmath,amsthm, amssymb, amsfonts}
\usepackage{mathtools}
\usepackage{graphicx}
\usepackage[hidelinks]{hyperref}
\usepackage{xcolor}

\newcommand{\arxiv}[1]{\href{https://arxiv.org/abs/#1}{\texttt{ArXiv:#1}}}
\newcommand{\arxivmath}[1]{\href{https://arxiv.org/abs/math/#1}{\texttt{ArXiv:#1}}}
\newcommand{\arxivcond}[1]{\href{https://arxiv.org/abs/cond-mat/#1}{\texttt{ArXiv:#1}}}
\newcommand{\arxivph}[1]{\href{https://arxiv.org/abs/math-ph/#1}{\texttt{ArXiv:#1}}}

\theoremstyle{plain}
\newtheorem{thm}{Theorem}[section]
\newtheorem{cor}[thm]{Corollary}
\newtheorem{lem}[thm]{Lemma}
\newtheorem{prop}[thm]{Proposition}

\theoremstyle{definition}
\newtheorem{defn}[thm]{Definition}
\newtheorem{rem}[thm]{Remark}
\newtheorem{exm}[thm]{Example}

\numberwithin{equation}{section}

\newcommand{\Hp}{\mathbb{H}}
\newcommand{\R}{\mathbb{R}}
\newcommand{\Z}{\mathbb{Z}}
\newcommand{\Ze}{\mathbb{Z}_{\rm{even}}}

\newcommand{\La}{\mathcal{L}}
\newcommand{\Ai}{\mathcal{A}}
\newcommand{\hh}{\mathbf{h}}
\newcommand{\eps}{\varepsilon}

\newcommand{\bemph}[1]{{\bf #1}}
\newcommand{\cm}{\operatorname{cmaj}}

\title{Infinite geodesics, competition interfaces and the second class particle in the scaling limit}
\author{Mustazee Rahman \thanks{\textsc{Department of Mathematical Sciences, Durham University}. \textit{Email}: \texttt{mustazee@gmail.com}}
	\and B\'alint Vir\'ag\thanks{\textsc{Department of Mathematics, University of Toronto}. \textit{Email}: \texttt{balint@math.toronto.edu}}}
\date{}

\begin{document}
	\maketitle
		\setcounter{tocdepth}{1}
 
	\tableofcontents
	\begin{abstract} 
	    We establish fundamental properties of infinite geodesics and competition interfaces in the directed landscape.
	    We construct infinite geodesics in the directed landscape, establish their uniqueness and coalescence, and define Busemann functions.
	    We then define competition interfaces in the directed landscape. We prove the second class particle in tasep
	    converges under KPZ scaling to a competition interface. Under suitable conditions, we show the competition interface
	    has an asymptotic direction, analogous to the speed of a second class particle, and determine its law. Moreover, we
	    prove the competition interface has an absolutely continuous law on compact sets with respect to infinite geodesics.
	\end{abstract}
	\medskip
	{\footnotesize
	}
	
	\newpage
	
	\section{Introduction} \label{sec:intro}
	Imagine a random metric over the plane. Given two points, consider the locus of all points that are equidistant from them.
	In Euclidean geometry this is of course a straight line, the perpendicular bisector between the points, and also an infinite geodesic
	of the Euclidean metric. What is the shape of this curve in the random metric? Of interest are its local behaviour, its relation to
	geodesics, and its global nature.
	
	We study geodesics and bisectors (called competition interfaces) in the  directed landscape, which is expected to be the scaling limit of the most natural planar discrete random metric: first passage percolation. This connection is conjectural, but it has been shown that the directed landscape is the scaling limit of most classical last passage percolation models, and
	carries the geometric structure inherent in models belonging to the KPZ universality class. It therefore describes random, universal KPZ geometry.

    Let us introduce the directed landscape as the scaling limit of exponential last passage percolation.
    Assign independent exponential random variables $\omega_p$ of mean 1 to points $p \in \Ze^2 =\{(x,y)\in \mathbb Z^2,x+y \mbox{ even}\}$.
	A finite sequence $\pi_0,\ldots, \pi_n$ of points in $\Ze^2$  is an (upwards) directed path if $\pi_{i+1}-\pi_{i}=(1,1)$ or $(-1,1)$ for all $i$.
	For $a,b\in \Ze^2$, let 
	\begin{equation} \label{eqn:LPPdef}
		L(a;b)=\max_{\pi} \sum_{u\in \pi} \omega_u 
	\end{equation} 
	where the maximum is taken over directed paths beginning at $a$ and ending in $b$; let $L(a;b)= -\infty$ if there is no such path, and extend $L$ to $(\mathbb R^2)^2$ by rounding arguments to the nearest point in $\Ze^2$. 
	
	In \cite{DV} it is shown that as $\eps \to 0$, jointly as functions of $(y,s;x,t) \in \R^4$,
	\begin{equation} \label{eqn:KPZscaling2}
		 \frac{\eps}{2} \,  L\Big(2y/\eps^2,s/\eps^3\;;\;2x/\eps^2,t/\eps^3\Big) - (t-s)/\eps^2 \;\stackrel{law}{\Rightarrow}\; \La(y,s;x,t),
	\end{equation}
	where  the random function $\La :(\R^2)^2\to \R\cup\{-\infty\}$ is the {\bf directed landscape}.
	This function is a directed metric on $\mathbb R^2$: it satisfies $\mathcal L(p,p)=0$ and the reverse triangle inequality
	$\La(p,q)\ge \La(p,r)+\La(r,q)$. For distinct points $p = (y,s), q = (x,t)$ with $s \geq t$, $\La(p,q)=-\infty$.
	Otherwise $\La(p,q)+(x-y)^2/(t-s)$ has  a GUE Tracy-Widom law scaled by $(t-s)^{1/3}$.
	
	As an application of the results in this paper, we find the following formula for the asymptotic direction of bisectors in the directed landscape. Let $\Hp = \{(x,t): x \in \R, t > 0\}$ denote the open upper half plane.
	
	\begin{thm}[Direction of bisectors] \label{thm:bisdir}
		Consider the points $(-1,0)$ and $(1,0)$ and the locus of all points $(x,t) \in \Hp$
		that are $\La$-equidistant from them, namely $\La(-1,0; x,t) = \La(1,0; x,t)$. This set is the graph $\{(I(t), t): t> 0\}$ of a continuous function $I(t)$. Moreover,  
		$$\lim_{t \to \infty} \frac{I(t)}{t} = D$$
  	exists almost surely and is normally distributed with mean 0 and variance $1/4$.
		
		Next, consider the line segment $[-1,1]\times \{0\}$. A point $(x,t) \in \Hp$ is $\La$-equidistant from its left half and its right half if $\sup_{y \in [-1,0]}  \La(y,0;x,t) = \sup_{y \in [0,1]}  \La(y,0;x,t)$.
		The $\La$-equidistant points form the graph $\{(I(t), t): t > 0\}$ of a continuous function and, almost surely,
		$$\lim_{t \to \infty} \frac{I(t)}{t} = D \stackrel{law}{=} \frac{(\beta_2 - \beta_1)\chi_5 - N}{2}.$$
		The pair $(\beta_1, \beta_2)$ are the first two components from a Dirichlet(1,1,2) random vector.
		The random variable $\chi_5$ has the chi-five distribution and $N$ has a standard normal distribution.
		All three of $(\beta_1, \beta_2), \chi_5$ and $N$ are independent. 
	\end{thm}

    Bisectors and geodesics of the directed landscape are intimately connected, and the paper explores their interplay. Along the way we also draw on the connection between bisectors and the second class particle in the totally asymmetric simple exclusion process. The rest of the Introduction discusses these concepts and our results in three parts: competition interfaces, second class particles and geodesics.

    \subsection{Competition interfaces of the directed landscape} \label{sec:intcomp}
	
	Competition interfaces in last passage percolation were introduced by Ferrari, Martin and Pimentel \cite{FP, FMP}. They 
	are related to the trajectory of a 2nd class particle in tasep and also identify directions with non-unique geodesics
	in last passage percolation. Their asymptotic direction has been a subject of intense research.
        In  \cite{FP} Ferrari and Pimentel derived the direction of the competition interface in exponential last passage percolation with the step
        initial condition. This was generalized by Ferrari, Martin and Pimentel \cite{FMP, FMP2} to initial conditions having different inclinations.
        Georgiou, Rassoul-Agha and Sepp\"{a}l\"{a}inen \cite{GRS} showed existence of the direction for step initial condition in the last passage percolation model with continuous i.i.d.~weights and certain assumptions. Cator and Pimentel \cite{CP} considered exponential last passage percolation and the Hammersley process with rather general initial conditions, and described the limiting direction of competition interfaces in terms of suprema of certain random walks. 
 
        Bal\'{a}zs, Cator, and Sepp\"{a}l\"{a}inen \cite{BCS} proved an interesting duality relating competition interfaces to geodesics. Ferrai and Nejjar \cite{FN} found single-point limiting fluctuations for the competition interface in the presence of shocks.

        In this paper we introduce competition interfaces of the directed landscape. This is the topic of Section \ref{sec:Interfaces}.
	Competition interfaces are defined in $\S$\ref{sec:defofint}, and basic properties such as continuity and distributional invariances are established in  $\S$\ref{sec:propertiesofint} and $\S$\ref{sec:invarianceofint}.
 
        We believe the next theorem is the first {\it geometric} representation of the direction of competition interfaces. Here is the setup. An {\bf initial condition} $h_0: \R \to \R\cup\{-\infty\}$ is an upper semicontinuous function that is finite at some point and satisfies $h_0(x) \leq c(1+ |x|)$
        for all $x$ and some constant $c$. For $(x,t) \in \Hp$, the distance from $h_0$ to $(x,t)$ is
	\begin{equation} \label{eqn:Ldist} 
		\La(h_0; x,t) = \sup_{y \in \R} \, \{ h_0(y) + \La(y,0;x,t)\}.
	\end{equation}
	
	We say that $y$ is an {\bf interior reference point} for an initial condition $h_0$ if  there are $x < y < z$ such that $h_0(x)$ and $h_0(z)$ are finite. We say $y$ is a {\bf non-polar point} if for a two-sided Brownian motion $B$ with diffusivity constant $\sqrt{2}$, and every $a>0$, the function $h_0(x) + B(x)$ does not have a maximum over $[-a,a]$ at $x = 0$ almost surely. Let $h_0^-$ and $h_0^+$ be the restrictions of $h_0$ to the non-positive  and non-negative axes, taking value $-\infty$ elsewhere. In Proposition \ref{prop:intunique} we show that if $0$ is a non-polar interior reference point for $h_0$, the set of points $(x,t) \in \Hp$ that are $\La$-equidistant from $h_0^-$ and $h_0^+$ forms the graph
	\begin{equation}  \label{eqn:interface}
		\{(I(t),t), t>0\} 
	\end{equation}
	of a continuous function $I(t)$. The curve $I(t)$ is the bisector of $h_0$ from the origin $x=0$,  also called a competition interface. 
	
	Let $\cm(f)$ denote the concave majorant of the function $f$: the infimum of all concave functions that are at least $f$.
	Let $B$ be a two-sided Brownian motion with constant diffusivity $\sqrt{2}$. The following theorem gives the direction of $I(t)$.
	
	\begin{thm}[Direction of interfaces] \label{thm:directionlaw}
	Suppose that $0$ is a non-polar interior reference point for an initial condition $h_0$.
        Assume that there are $a,b > 0$  so that $h_0(x) \leq -a|x| + b$ for every $x$.
	Then there is an almost sure limit:
	 $$
	 \lim_{t\to\infty} \frac{I(t)}{t} = D \stackrel{law}{=} -\frac{1}{2} \cm(h_0 + B)'(0).
	 $$
	\end{thm}
	In words, the vector $(2D,1)$ is perpendicular to the graph of $\cm(h_0+B)$ at zero. This is a randomized analogue of  Euclidean plane geometry, where the direction is  perpendicular to $2\cm(h_0)$. Theorem \ref{thm:directionlaw} allows us to compute the direction in several settings, in particular, recent work of Ouaki and Pitman \cite{OP} about $\cm(B|_{[0,\infty]})$ allows us to deduce Theorem \ref{thm:bisdir}. 
	
    The asymptotic direction of competition interfaces is studied in $\S$\ref{sec:directionofint}.
	Theorem \ref{thm:direction} proves that when $h_0$ is asymptotically flat, its interfaces are asymptotically vertical.
	Theorem \ref{thm:randdirection} is then a generalization of Theorem \ref{thm:directionlaw}.

 Next, we connect competition interfaces to geodesics by studying the local nature of the curve $I(t)$ from \eqref{eqn:interface}. In Euclidean geometry, bisectors are generally lines, which are themselves Euclidean geodesics.
 In this random geometry, the curve $I(t)$ is locally indistinguishable in law from a geodesic of $\mathcal L$.
	
	Specifically, there is almost surely a unique infinite geodesic of $\La$ from $(0,0)$ in direction zero,
	see Theorems \ref{thm:geoexist} and \ref{thm:longgeo}. This is the unique random continuous function $(g(t), t \geq 0)$ such that
	(i) $g(0) = 0$ and $g(t)/t \to 0$ as $t \to \infty$, and (ii) for every $t_1 < t_2 < t_3$ and $z_i = (g(t_i), t_i)$, $\La(z_1,z_3) = \La(z_1,z_2) + \La(z_2,z_3)$. 	The following theorem about absolute continuity of competition interfaces with respect to geodesics is proved in $\S$\ref{sec:intabscont}.

	\begin{thm}[Geodesic-like interfaces] \label{thm:abscont}
	For any compact interval $[a,b] \subset (0,\infty)$, the law of the restricted competition interface  $I|_{[a,b]}$ is absolutely continuous, as a continuous process, with respect to the law of  the restricted geodesic $g|_{[a,b]}$. 
	\end{thm}

	The implications of absolute continuity are very strong: all local almost sure properties of $g(t)$ are inherited by $I(t)$.
	For example, $I(t)$ has three-halves variation and is H\"{o}lder continuous with any exponent less than 2/3, see \cite{DSV}.
	
        Moreover, from a single sample of $I(t)$ one can recover the law of the entire stationary geodesic $\gamma_*$ defined in \cite{DSV}. This $\gamma_*$ is the scaling limit of a geodesic near a time away from the boundary, see \cite{DSV}. 
        It can also be thought of as a geodesic conditioned to be doubly infinite. (Such conditioning changes the law of the directed landscape environment; doubly infinite geodesics do not exist in the directed landscape as recently shown in \cite{Bh}.) To see how to recover the law of $\gamma_*$, let $A$ be an event defined in terms of $\gamma_*$ restricted to some compact time-interval $[a,b]$. Let $I_{n,k}(t)= n^{2/3}(I(1+t/n+k/n)-I(1+k/n))$ for $t \in [a,b]$ and $1 \leq k \leq n$. Then along a deterministic sequence of natural numbers, 
	$$
	\lim \frac{1}{n}\sum_{k=1}^n\mathbf{1}(I_{n,k}\in A) =\textbf{Pr}(\gamma_{*}\in A) \qquad a.s.
	$$

	The law of the competition interface $I(t)$ at a fixed time is discussed in $\S$\ref{sec:interfacelaw}, where it
	is expressed in terms of a concave majorant associated to the Airy process.

	Section \ref{sec:portrait} looks at the family of all competition interfaces associated to an initial condition $h_0$.
	We call this the interface portrait of $h_0$. It is complementary to geodesics from $h_0$ and
	is a forest lying in the upper half plane, see Theorem \ref{thm:portraitgeom}.
	The portrait has curious topological and geometric properties. For instance, it is quite ``thin" -- its horizontal sections are always
	a discrete set as shown in Proposition \ref{prop:portraittopo}. There should be an interesting duality between interface portraits
	and geodesics that remains to be understood.

  \subsection{The 2nd class particle in tasep} \label{sec:int2nd}

	We use interfaces to derive the scaling limit of a 2nd class particle in the totally asymmetric simple exclusion process (tasep).
	A fundamental result of Ferrari, Martin and Pimentel is that competition interfaces are related to the trajectory of a 2nd class particle
	in tasep via a time-change \cite{FP, FMP}. We show that, under suitable conditions, the scaling limit of a competition interface in exponential last passage percolation is a corresponding interface $I(t)$. This leads to the scaling limit of a 2nd class particle in tasep. There are no previous functional limit theorems about the 2nd class particle, but the theory built so far allows us to prove such a result. Let us elaborate on it now.
	
	In the continuous-time tasep, particles initially occupy sites along the integers, at most one per site,
	and those sites which are not occupied are called holes. Each particle tries to jump randomly at rate 1 to its
	neighbouring site to the right. However, a jump is allowed only if a particle is displacing a hole, as per the exclusion rule
	each site can have at most one particle.
	
	A \bemph{2nd class particle} in tasep initially starts at site 0, surrounded by other particles and holes. It, too, tries
	to jump right like a regular particle, displacing holes but not particles. But when a regular particle wants to jump
	to the site of the 2nd class particle, it is allowed to do so, moving the 2nd class particle leftward. See Liggett's book \cite{Lig} for both a rigorous construction of tasep and of tasep with 2nd class particles.
	
	Consider tasep with a single 2nd class particle at the origin and initial condition $X_0(\cdot)$, where $X_0(n)$ is the initial location
	of regular particle number $n$, with particles labelled right to left and particle number $1$ being the first one below site $0$.
	Let $X(t)$ be the position of the 2nd class particle at time $t$.
	
	Suppose there is a sequence of tasep initial conditions $X^{n}_0$ for $n \geq 1$, and it converges under the re-scaling
	\begin{equation} \label{eqn:diffscaling}
		n^{-1/2}( X^{n}_0(\lceil n x \rceil ) + 2 nx) \to - h_0(-x) \quad \text{as}\; n \to \infty,
	\end{equation}
	where $h_0$ is an initial condition in the continuum as introduced earlier. Suppose the convergence is uniform on compact subsets of $x \in \R$. Suppose the initial conditions also satisfy the following technical growth condition: $n^{-1/2}( X^{n}_0(\lceil nx \rceil ) + 2 nx) \leq c (1+|x|)$ for a constant $c$.
	Assume that $0$ is a non-polar interior reference point for $h_0$. Then there is a competition interface $I(t)$ associated to $h_0$ as in \eqref{eqn:interface}.
	
	\begin{thm}[Scaling limit of the 2nd class particle] \label{thm:scalinglimit}
		Let $X^{n}(t)$ be the position of the 2nd class particle at time $t$ with initial condition $X^{n}_0$.
		Under the assumptions above the re-scaled trajectory of the 2nd class particle converges to the interface:
		$$ (n^{-2/3}/2) X^{n}(2 n t) \to I(t)$$
		in law, under the topology of uniform convergence over compact subsets of $t \in (0,\infty)$.
	\end{thm}
	The mode of convergence \eqref{eqn:diffscaling} of initial conditions can be relaxed to include a broader class.
	We state and prove the more general theorem in $\S$\ref{sec:2ndclassparticle} as Theorem \ref{thm:2ndclass}.
	
	Theorem \ref{thm:scalinglimit} is new even for the classical flat and stationary initial conditions.
	The flat initial condition consists of particles at every even integer site, and under the re-scaling \eqref{eqn:diffscaling}
	it converges to $h_0(x) \equiv 0$. The stationary initial condition has a particle at each site independently with probability $1/2$; under diffusive scaling \eqref{eqn:diffscaling} it converges in law to a two-sided Brownian motion. For the stationary initial condition, the convergence of the 2nd class particle holds jointly over the randomness of the initial condition and the independent evolution of tasep (one needs an additional argument to justify the technical growth condition above, which is included in Proposition \ref{prop:statintlimit}).
	
	The flat interface $I_{\rm{flat}}(t)$ has an invariant law under re-scaling $\sigma^{-2} I_{\rm{flat}}(\sigma^3 t)$.
	Its law at time $t=1$ is
	$$I_{\rm{flat}}(1) \stackrel{law}{=} - \frac{1}{2} \cm( \mathfrak A)' (0)$$
	where $\mathfrak A(x) = \La(0,0; x,1)$ is the parabolic Airy process.
	Indeed, in Proposition \ref{p:interface-distr} we provide a geometric characterization of the single-time law of interfaces.
	
	The stationary interface $I_{\rm{stat}}(t)$ has the same law as the aforementioned infinite geodesic of $\La$ from $(0,0)$ in direction zero.
	This equality is based on a duality between geodesics and competition interfaces in last passage percolation \cite{Sep2}.
	The law of the infinite geodesic is in turn described by Corollary \ref{cor:geolaw}, in terms of geodesics from the stationary (Brownian) initial condition.
	The law of $I_{\rm{stat}}(t)$ is scale invariant like $I_{\rm{flat}}(t)$, and
	$$I_{\rm{stat}}(1) \stackrel{law}{=} \mathrm{argmax} \, \{ \mathfrak A(x) + B(x)\}$$
	where $B$ is a two-sided Brownian motion with diffusivity constant $\sqrt{2}$ and independent of $\mathfrak A$.
	Ferrari and Spohn \cite{FS} found the distribution function of $I_{\rm{stat}}(1)$ in terms of the KPZ scaling function.
	Bal\'{a}zs, Cator, and Sepp\"{a}l\"{a}inen \cite{BCS} gave optimal bounds on the fluctuation of the 2nd class particle and competition interface in stationarity.
	
The motivation behind the scaling in Theorem \ref{thm:scalinglimit} and its connection to the Burgers' equation is discussed in Section \ref{s:burgers}.

	\subsection{Geodesics of the directed landscape} \label{sec:intgeo}
	A {\bf geodesic} of the directed landscape is a continuous function $g$ from a closed interval $J$ so that for all  $t_1<t_2<t_3$ in $J$ and
	$z_i=(g(t_i),t_i)$, the reverse triangle inequality is sharp: $\La(z_1,z_3)=\La(z_1,z_2)+\La(z_2,z_3)$. Section \ref{sec:geodesic} studies geodesics, which are in fact the building blocks for competition interfaces and therefore discussed first.
	
	We consider geodesics from an initial condition $h_0$.
	Given $(x,t) \in \Hp$, a geodesic from $h_0$ to $(x,t)$ is a geodesic $g$ of the directed landscape defined on $[0,t]$ with $g(t)=x$
	and so that $$ y \mapsto h_0(y)+\mathcal L(y,0;x,t)$$ is maximized at $y=g(0)$.
	In $\S$\ref{sec:geotoinit}, regularity properties of geodesics from an initial condition are established, and the following unique geodesic condition is proved in $\S$\ref{sec:UGC}, see Theorem \ref{thm:ugc}. For every initial condition $h_0$, almost surely for every $t>0$, for all but countably many $x$, there exists a unique geodesic from $h_0$ to $(x,t)$. This fact about the web of geodesics emanating from $h_0$ allows us to study competition interfaces in detail.
	
	We also study infinite geodesics of the directed landscape in $\S$\ref{sec:longgeo}.
	These are geodesics defined on time intervals of the form $[t, \infty)$.
	Such geodesics are parametrized by their starting point and direction.
	Existence of an infinite geodesic from a fixed point in a fixed direction is shown in Theorem \ref{thm:geoexist},
	along with an effective bound on its deviation.
	Next, the almost sure uniqueness of an infinite geodesic from a fixed point in a fixed direction is proven, and it is shown that geodesic uniqueness starting from a fixed point fails only for a (random) countable number of directions, see Theorem \ref{thm:longgeo}. Finally, the almost sure coalescence of all infinite geodesics in a fixed direction is established in Theorem \ref{thm:geotree} and Corollary \ref{cor:geotree}.
	
	Infinite geodesics allow us to study Busemann functions of the directed landscape in $\S$\ref{sec:Busemann}.
	The law of the Busemann function in a fixed direction is described by Corollary \ref{c:Busemann-Brownian} and Theorem \ref{thm:Busemann-KPZ}. This is used to identify the law of an infinite geodesic in Corollary \ref{cor:geolaw}.
	
	The results on infinite geodesics and Busemann functions are inspired by, and complement, what has been proven about infinite geodesics in first and last passage percolation. The results here are nevertheless self-contained, although our methods bear a resemblance to the pioneering work on discrete models, particularly by Newman and co-authors. A key difference for the directed landscape is that one has to deal with both small and large time scales, and there are additional symmetries.
 
        The study of geodesics in first and last passage percolation models has a rich history, beginning with the work of Newman and co-authors; see for instance Newman \cite{Ne}, Licea and Newman \cite{LN}, and Howard and Newman \cite{HN}.
        In the last passage percolation literature, Ferrari and Pimentel \cite{FP} studied infinite geodesics in exponential last passage percolation
	and proved several key results, which were extended by Coupier \cite{Co}. 
	Sepp\"{a}l\"{a}inen \cite{Sep2} gives a  recent exposition about geodesics in last passage percolation, 
	including proofs of the main results for the exponential case. Busemann functions in first and last passage
	percolation are intimately connected to its geometry. Hoffman \cite{Hoff} was one of the first to study them in first passage percolation.
        Sepp\"{a}l\"{a}inen and co-authors have a large body of work in this direction for last passage percolation,
	establishing fine geometric properties, see \cite{FaSe, GRS2, JRS, SeSo}. Cator and Pimentel \cite{CP2, CP} have also studied similar questions.
 
        Busani, Sepp\"{a}l\"{a}inen and Sorensen \cite{BSS} use our results to study Busemann functions and geodesics in the directed landscape. They characterize the Busemann process simultaneously across all directions and establish geometric properties of infinite geodesics along all directions and starting points, for instance, characterizing the countable dense set of directions with non-unique geodesics. Ganguly and Zhang \cite{GZ} recently and independently study infinite geodesics and Busemann functions of the directed landscape as well. Bhatia \cite{Bh} has recently studied the fractal nature of infinite geodesics in the directed landscape and the duality with interfaces.
	
	\section{Impressions from a landscape} \label{sec:impressions}
	We gather here some preliminaries and basic notions used throughout the article.
	
	\subsection{The directed landscape} \label{sec:undrand}
	The objects in our discussion are governed by the underlying randomness of the directed landscape $\La(y,s;x,t)$,
	defined for spatial coordinates $x,y \in \R$ and temporal coordinates $s, t \in \R$. It was introduced
	by Dauvergne, Ortmann and Vir\'{a}g \cite{DOV}.
	
	The function $\La$ acts like a directed metric on $\mathbb R^2$: it satisfies $\mathcal L(p;p)=0$
	and the reverse triangle inequality $\mathcal L(p;q)\ge \La(p;r)+\La(r;q)$. In fact, it has the the metric composition property
	\begin{equation} \label{eqn:metcomp}
		\La(y,s; x,t) = \max_{z \in \R} \left \{ \La(y,s;z,u) + \La(z,u; x,t) \right \}
	\end{equation}
	for every $x,y$ and $s < u < t$. For distinct points $(y,s)$ and $(x,t)$, $\mathcal L(y,s;x,t)=-\infty$ if and only if $s \geq t$.
	Otherwise, for $s < t$, $\La(y,s;x,t)+(x-y)^2/(t-s)$ has  GUE Tracy-Widom law scaled by $(t-s)^{1/3}$.
    
    $\La$ is a continuous function, stationary separately in the space and time coordinates,
    and has independent increments in time in the metric composition semigroup. We use the notation 
	$$ \Ai(y,s; x,t)=\La(y,s; x,t) + \frac{(x-y)^2}{t-s} $$
	for the ``centered'' version of $\mathcal L$.
	Several properties of $\La$ and $\Ai$ will be used throughout.  A bound that will be used a lot, see \cite[Corollary 10.7]{DOV}, is that
	for all $x,y \in \R$ and $t > 0$,
	\begin{equation} \label{eqn:DOVAbound}
		|\Ai(y,0;x,t)| \leq  C t^{1/3} \log^{4/3}\left (\frac{2(||u||+2)}{t} \right )\log^{2/3}(||u||+2)
	\end{equation}
	for a finite random variable $C$, where $u = (y,0;x,t)$ and $||u||$ is the Euclidean norm. In particular, this bound implies that
	\begin{equation} \label{eqn:Abound}
		|\Ai(y,0;x,t)| \leq C (t^{1/2} + |x|^{1/2} + |y|^{1/2} + 1)
	\end{equation}
	for a random variable $C$.
	
	Another property of $\La$ is shear invariance \cite[Lemma 10.2]{DOV}, which states that for any $d \in \R$,
	as a process in $(y,s;x,t)$,
	\begin{equation} \label{eqn:shear}
		\La(y + ds ,s; x + dt,t) \overset{law}{=} \La(y,s;x,t) + \frac{(y-x)^2 - (y-x-d(t-s))^2}{t-s}.
	\end{equation}
	
	Other properties of $\La$ will be discussed as the time comes.
	The properties mentioned here are proved in \cite{DOV}, see also \cite{CQR}.
	
	\subsection{The KPZ fixed point} \label{sec:KPZfixedpt}
	Denote the open upper half plane by $\Hp = \{(x,t): x \in \R, t > 0 \}$ and by $\overline{\Hp}$ its closure.
	An {\bf initial condition} is a function
	$$h_0: \R \to \R \cup \{-\infty \}$$
	which is upper semicontinuous, finite on at least one point, and satisfies $h_0(x) \leq c(1+|x|)$ for some constant $c$.
	Define the height function with initial condition $h_0$ as
	\begin{equation} \label{eqn:KPZfixedpt}
		\La(h_0; x,t) = \sup_{y \in \R} \; \{ \, h_0(y) + \La(y,0; x,t) \, \},\qquad  (x,t) \in \Hp.
	\end{equation}
	The law of  $\La(h_0;\cdot)$ is called the {\bf KPZ fixed point}, first constructed by Matetski, Quastel and Remenik \cite{MQR}
	as a Markov process on function space.
	The height function is continuous in $(x,t)$, and for every $t$, $\La(h_0; \cdot,t)$ also satisfies
	the definition of an initial condition \cite{MQR}.
	
	The KPZ fixed point is  also a stochastic integrable system in that its finite
	dimensional laws can be expressed in terms of Fredholm determinants;
	see \cite{JR, JR2, Liu, MQR} for such results on the spacetime law of the KPZ fixed point,
	and \cite{BoGo, CoKPZ, QuKPZ} for surveys on the KPZ universality class.
	
	The directed landscape gives a simultaneous coupling of the KPZ fixed point for all initial conditions through the variational formula \eqref{eqn:KPZfixedpt}.
	This allows to compare different height functions through a common underlying randomness.
	We will thus call $\La(h_0;x,t)$ the directed landscape height function when it is realized as \eqref{eqn:KPZfixedpt}.
	
	\subsection{Geodesics and their properties} \label{sec:Lgeodesic}
	The function $\La(y,s; x,t)$ can be thought of as giving a random geodesic length from $(y,s)$ to $(x,t)$,
	and in this sense the equation \eqref{eqn:KPZfixedpt} measures the length of a geodesic weighted by an initial condition.
	These concepts can be formalized by discussing geodesics of $\La$.
	
	A path $p$ from $(y,s)$ to $(x,t)$ for $s < t$ is a continuous function $p: [s,t] \to \R$ with $p(s) = y$ and $p(t) =x$.
	Paths are considered to be moving upwards, so $p$ begins at $(y,s)$ and ends at $(x,t)$.
	A path may be identified with its graph
	$$ \{ (p(u), u), \; u \in [s,t] \}$$
	lying in the spacetime plane $\R^2$. The Hausdorff metric on closed subsets of the plane then induces the Hausdorff topology on the space of paths in terms of their graphs. In the following, the topology on geodesics will be induced by the Hausdorff topology on their graphs.
	
	Here are some basic notions about paths.
	\begin{description}
		\item [Ordering] A path $p$ from $(y,s) \to (x,t)$ is to the left of a path $p'$ from $(y',s) \to (x',t)$ if $p(u) \leq p'(u)$ for every $u \in [s,t]$.
		\item [Crossing] Paths $p$ and $p'$ cross if there are times $s \neq t$ such that $p(s) < p'(s)$ and $p(t) > p'(t)$.
		\item [Coalescing] Paths $p$ and $p'$ coalesce downward from time $t$ if $p(s) = p'(s)$ for every $s \leq t$, and coalesce upward
		from time $t$ if $p(s) = p'(s)$ for every $s \geq t$.
	\end{description}
	
	The length of a path $p: [s,t] \to \R$ with respect to $\La$ is
	\begin{equation} \label{eqn:pathlength}
		\ell(p) = \inf_{\text{all partitions of}\; [s,t]} \sum_{i=1}^n \La(p(t_{i-1}), t_{i-1}; p(t_i), t_i)
	\end{equation}
	where $s = t_0 < t_1 < \cdots < t_n = t$ is any partition of $[s,t]$.
	
	A geodesic of $\La$ from $(y,s)$ to $(x,t)$ is a path of maximal length. Recall $\La$ satisfies the reverse triangle inequality.
	The reverse triangle inequality becomes an equality along points of a geodesic,
	and geodesics are the paths for which that holds: for a geodesic $g$ from $(y,s)$ to $(x,t)$,
	$$\ell(g) = \sum_{i=1}^n \La(g(t_{i-1}), t_{i-1}; g(t_i), t_i)$$
	for any partition $\{t_i\}$ of $[s,t]$.  The geodesic length from $(y,s)$ to $(x,t)$ is $\La(y,s;x,t)$.

        We observe the following elementary fact about geodesics of $\La$. Suppose $g$ is a geodesic from $p=(y,s)$ to $q = (x,t)$.
        Let $a = (g(u),u)$ and $b = (g(v),v)$ be two intermediate points on $g$ with $u < v$. Suppose $g'$ is another geodesic from $a$ to $b$.
        The path $\gamma$ that is equal to $g$ on $[s,u]$, $g'$ on $[u,v]$ and then $g$ on $[v,t]$ is also a geodesic from $p$ to $q$.
        We will say that $\gamma$ is obtained by swapping the segment of $g$ from $a$ to $b$ with $g'$.
        
	Geodesics of $\La$ have properties that hold almost surely in samples of $\La$.
	Some of these are listed below and will be used throughout our discussion. They are derived from \cite{BatGH, DOV, DSV}.
	\begin{description}
		\item[Existence] There is a geodesic between every pair of points $(y,s)$ and $(x,t)$ for $s < t$.
		There are also \textbf{leftmost} and \textbf{rightmost} geodesics from $(y,s)$ to $(x,t)$.
            A leftmost geodesic $g$ from $(y,s)$ to $(x,t)$ is a geodesic such that if $g'$ is any other geodesic from $(y,s)$ to $(x,t)$
            then $g(u) \leq g'(u)$ for every $u \in [s,t]$; analogously for the rightmost geodesic. See \cite[Lemma 13.2]{DOV}.
		\item [Closure] If a sequence of geodesics converges in terms of the Hausdorff metric on their graphs, then the limit is also a geodesic.
		The length functional \eqref{eqn:pathlength} is continuous with respect to convergence of geodesics in the Hausdorff metric. See \cite[Lemma 3.1]{DSV}.
		\item [Compactness] The collection of geodesics whose graphs have endpoints lying in a given compact subset of $\{(y,s;x,t)\in \R^4: s<t\}$
		is itself compact in the Hausdorff topology. See \cite[Lemma 3.1]{DSV}.
		\item [Nearby geodesics meet] For every compact subset $K \subset \R^2$, there is an $\epsilon > 0$ such that if $g$ and $g'$ are two geodesics whose graphs lie in $K$ and are within distance $\epsilon$ in the Hausdorff metric, then there is a common time $t$ such that $g(t) = g'(t)$. See \cite[Theorem 1.18]{BatGH}.
	\end{description}
	
	\section{Geodesics} \label{sec:geodesic}
	Geodesics form the building blocks behind much of our results, so we begin our discussion with them.
 
	\subsection{Geodesics from an initial condition} \label{sec:geotoinit}
	Fix an initial condition $h_0$ as in $\S$\ref{sec:KPZfixedpt}.
	A path from $h_0$ to $(x,t) \in \Hp$ is a path $p: [0, t] \to \R$ with $p(t)=x$.
	The starting point of $p$ is $p(0)$. The length of the path relative to $h_0$ is
	\begin{equation} \label{eqn:relativelength}
		\ell(h_0;p) = h_0(p(0)) + \ell(p),
	\end{equation}
	where $\ell(p)$ is defined by \eqref{eqn:pathlength}.
	
	A geodesic from $h_0$ to $(x,t)$ is a path from $h_0$ to $(x,t)$ of maximal relative length, that is,
	$p \mapsto \ell(h_0;p)$ is maximal among all paths $p$ from $h_0$ with endpoint $(x,t)$. This maximal value is 
	the geodesic length from $h_0$ to $(x,t)$; it equals $\La(h_0;x,t)$ from \eqref{eqn:KPZfixedpt}. In turn, this equals
	equals $h_0(g(0))+\La(0,g(0); x,t)$ for any geodesic $g$ from $h_0$ to $(x,t)$.
	
	Geodesics from $h_0$ are a sub-collection of point-to-point geodesics of $\La$.
	A geodesic from $h_0$ to $(x,t)$ is obtained by maximizing over all geodesics of
	$\La$ from $(y,0)$ to $(x,t)$, for $y \in \R$, the relative length \eqref{eqn:relativelength}. This leads to
	the variational formula \eqref{eqn:KPZfixedpt}.

        We will use the following elementary fact about geodesics from $h_0$ throughout.
        \begin{lem}[Geodesic swapping] \label{lem:geoswap}
        Let $g$ be a geodesic from $h_0$ to $p = (x,t)$. Let $q = (g(s),s)$ be any point on $g$. Suppose $g'$ is a geodesic from $h_0$ to $q$.
        The path $\gamma$ that equals $g'$ on $[0,s]$ and then $g$ on $[s,t]$ is also a geodesic from $h_0$ to $p$. In particular,
        if $g$ is the unique geodesic from $h_0$ to $p$ then it is also the unique geodesic from $h_0$ to $q$.
        \end{lem}

        \begin{proof}
        Let $y = g(0)$ and $y' = g'(0)$. Since $g$ and $g'$ are both geodesics from $h_0$ to $q$, $\La(h_0;q) = h_0(y) + \La(y,0;q) = h_0(y') + \La(y',0;q)$.
        Furthermore, $g$ is a geodesic of $\La$ from $(y,0)$ to $p$ and $g'$ is a geodesic of $\La$ from $(y',0)$ to $q$.
        In order to show that $\gamma$ is a geodesic from $h_0$ to $p$, we need that $\La(h_0;p) = h_0(y')+\La(y',0;p)$ and that $\gamma$ is a geodesic of $\La$ from $(y',0)$ to $p$.

        By definition, $\La(h_0;p) \geq h_0(y')+\La(y',0;p)$. By the reverse triangle inequality, $\La(y',0;p) \geq \La(y',0;q) + \La(q;p)$. Consequently,
        \begin{align*}
            h_0(y') + \La(y',0;p) & \geq h_0(y') + \La(y',0;q) + \La(q,p) \\
            & = h_0(y) + \La(y,0;q) + \La(q,p) \\
            & = h_0(y) + \La(y,0;p) = \La(h_0;p).
        \end{align*}
        We have used that $\La(y,0;q) + \La(q,p) = \La(y,0;p)$ since $(y,0), q,p$ lie on the geodesic $g$.

        In order to verify that $\gamma$ is a geodesic of $\La$ from $(y',0)$ to $p$ it is enough to show $\La(y',0;q) + \La(q;p) = \La(y',0;p)$, since $\gamma$ is the concatenation of a geodesic of $\La$ from $(y',0)$ to $q$ with one from $q$ to $p$.
        We know that $h_0(y') + \La(y',0;p) = h_0(y) + \La(y,0;p)$ and $h_0(y') + \La(y',0;q) = h_0(y) + \La(y,0;q)$.
        Subtracting the second equation from the first gives $\La(y',0;p) - \La(y',0;q) = \La(y,0;p) - \La(y,0;q)$.
        But the latter difference equals $\La(q;p)$, as required.
        \end{proof}
        
	\subsection{Basic properties of geodesics from an initial condition} \label{sec:geobasics}
	Certain properties of geodesics from $h_0$ come from properties of point-to-point geodesics of $\La$, holding
	whenever the properties of geodesics from $\S$\ref{sec:Lgeodesic} do. The assertions in
	the following three lemmas hold almost surely in the samples of $\La$, simultaneously for all initial conditions.
        Specifically, assume samples of $\La$ for which the aforementioned properties of point-to-point geodesics from $\S$\ref{sec:Lgeodesic} hold (existence, closure, compactness and ``nearby geodesics meet"), and also for which $\La$ is continuous and satisfies the bound \eqref{eqn:Abound}.
	
	\begin{lem}[Existence of geodesics] \label{lem:exist}
		For all initial conditions $h_0$, there is a geodesic from $h_0$ to every $(x,t) \in \Hp$.
		There are also leftmost and rightmost geodesics from $h_0$ to every point.
	\end{lem}
	\begin{proof}
		Consider the function $y \mapsto \La(y,0;x,t) + h_0(y)$. It decays rapidly to $-\infty$ because the $\La$-term
		decays to $-\infty$ parabolically while $h_0$ is bounded above by a linear term. More concretely, fix an $a \in \R$
		for which $h_0(a)$ is finite. From the estimate \eqref{eqn:Abound} for $\La$ and the bound $h_0(y) \leq B(1 + |y|)$
		it follows that there is a compact interval $J$, depending only on $a, B$, the finite random variable $C$ and the point $(x,t)$,
		such that the supremum of $\La(y,0;x,t) + h_0(y)$ over $y \in \R$ cannot be attained outside $J$. Inside of $J$ the supremum is
		achieved because $h_0$ is upper semicontinuous and bounded above (and $\La$ is continuous). Thus, all maximizers
		of the function lie in the compact interval $J$, and being upper semicontinuous and bounded above, the set of
		maximizers form a non-empty compact subset $K \subset J $.
		
		A geodesic from  $h_0$ to $(x,t)$ is obtained by picking any point $y \in K$ and choosing an $\La$-geodesic from $(y,0)$ to $(x,t)$.
		The rightmost geodesic is going to be the rightmost $\La$-geodesic from the rightmost point in $K$ to $(x,t)$;
		likewise for the leftmost geodesic.
	\end{proof}

        \begin{lem}[Continuity of length] \label{lem:lencont}
            For all initial conditions $h_0$, the mapping $(x,t) \mapsto \La(h_0;x,t)$ is continuous over $(x,t) \in \Hp$.
        \end{lem}
        \begin{proof}
            We must show that if $(x_n,t_n) \to (x,t)$ then $\La(h_0;x_n,t_n) \to \La(h_0; x,t)$.
            It suffices to show that $\liminf_n \La(h_0;x_n,t_n) \geq \La(h_0;x,t)$ and $\limsup_n \La(h_0;x_n,t_n) \leq \La(h_0;x,t)$.
            By Lemma \ref{lem:exist}, there is a geodesic $g$ from $h_0$ to $(x,t)$ as well as a geodesic $g_n$ from $h_0$ to $(x_n,t_n)$ for every $n$. Suppose $g$ begins at $(y,0)$ and $g_n$ begins at $(y_n,0)$.
            Then $\La(h_0; x,t) = h_0(y) + \La(y,0;x,t)$ and $\La(h_0;x_n,t_n) = h_0(y_n) + \La(y_n,0; x_n,t_n)$.

            Now $h_0(y_n)+\La(y_n,0;x_n,t_n) = \sup_z \{h_0(z) + \La(z,0;x_n,t_n)\} \geq h_0(y) + \La(y,0;x_n,t_n)$.
            Since $\La(y,0;x_n,t_n) \to \La(y,0;x,t)$ by continuity, it follows that $\liminf_n \La(h_0;x_n,t_n) \geq h_0(y) + \La(y,0;x,t) = \La(h_0;x,t)$.
            
            On the other hand, as $(x_n,t_n) \to (x,t)$, there is a compact set $K \subset \Hp$ that contains every $(x_n,t_n)$ and $(x,t)$.
            There are also constants $a$ and $B$ such that $h_0(a)$ is finite and $h_0(y) \leq B(1+|y|)$.
            Arguing as in the proof of Lemma \ref{lem:exist}, there is a compact interval $J$ depending only on $K$, $a$, $B$ and the random variable $C$ from \eqref{eqn:Abound}
            such that if $(x',t') \in K$ and $z'$ is any maximizer of $z \mapsto h_0(z) + \La(z,0;x',t')$, then $z' \in J$.
            In particular, every $y_n$ belongs to $J$. By passing to a subsequence we may assume that $y_n \to y'$.
            Due to upper semicontinuity of $h_0$, $\limsup_n h_0(y_n) \leq h_0(y')$.
            Since $\La(y_n,0;x_n,t_n) \to \La(y',0;x,t)$ by continuity, it follows that $\limsup_n h_0(y_n) + \La(y_n,0;x_n,t_n) \leq h_0(y') + \La(y',0;x,t)$. Since $h_0(y') + \La(y',0;x,t) \leq \La(h_0;x,t)$, we have $\limsup_n \La(h_0;x_n,t_n) \leq \La(h_0;x,t)$.
        \end{proof}
        
	\begin{lem}[Closure, Continuity and Compactness of geodesics] \label{lem:ccc}
		For all initial conditions $h_0$, the following properties hold.
		\begin{enumerate}
			\item The limit of a sequence of geodesics from $h_0$ in the Hausdorff topology is a geodesic from $h_0$.
			\item The relative length function $g \mapsto \ell(h_0;g)$ is continuous with respect to geodesics $g$ from $h_0$ in the Hausdorff topology.
			\item Geodesics from $h_0$ whose graphs lie in a compact set $K \subset \overline{\Hp}$ are compact in the Hausdorff topology.
			Consequently, if $(x_n,t_n) \to (x,t)$ and $g_n$ is a geodesic from $h_0$ to  $(x_n,t_n)$ for each $n$,
			then a subsequence of these $g_n$ converges to a geodesic $g$ from $h_0$  to $(x,t)$ .
		\end{enumerate}
	\end{lem}
	\begin{proof}
		The assertions are proven one by one by using the aforementioned properties of point-to-point geodesics of $\La$.
		\begin{enumerate}
			\item Let $g_n$ be a sequence of geodesics from $h_0$ that are convergent. Suppose $g_n$ starts at $(y_n,0)$ and ends at $(x_n,t_n)$.
			Since each $g_n$ is a geodesic of $\La$, their limit $g$ is too, say from $(y,0)$ to $(x,t)$ where $(x_n,t_n) \to (x,t)$ and $y_n \to y$.
			Upper semicontinuity of $h_0$ means $\limsup_n h_0(y_n) \leq h_0(y)$, and since the lengths $\ell(g_n) \to \ell(g)$,
			it follows that $\limsup_n \ell(h_0; g_n) \leq \ell(h_0; g)$. Now $\ell(h_0; g_n) = \La(h_0; x_n,t_n)$,
                which tends to $\La(h_0; x,t)$ by continuity from Lemma \ref{lem:lencont}.
			So $\ell(h_0; g) \geq \La(h_0;x,t)$. The reverse inequality $\ell(h_0;g) \leq \La(h_0;x,t)$ holds because $g$ is a path from $h_0$ to $(x,t)$.
			It follows that $g$ is a geodesic from $h_0$.
			
			\item The relative length $\ell(h_0;g)$ equals $\La(h_0;x,t)$ for a geodesic $g$ from $h_0$ to $(x,t)$.
			Its continuity is implied by the closure property in $(1)$ and the continuity of $\La(h_0;x,t)$.
			
			\item Compactness follows from the compactness of point-to-point geodesics and the closure property in $(1)$.
			The assertion in the consequence follows from the compactness property if there is a compact set that contains the graph of every $g_n$. In this regard note that there are $x_{-}, x_{+}$ and $T$ such that $x_{-} \leq x_n \leq x_{+}$ and $t_n \leq T$ for every $n$.
		The graphs of the $g_n$ are bounded by the leftmost and rightmost geodesics from $h_0$ to  $(x_{-}, T)$ and $(x_{+}, T)$, respectively.
		\end{enumerate}
	\end{proof}
	
	\subsection{Unique geodesics from an initial condition and the unique geodesic condition} \label{sec:UGC}
	We establish uniqueness of geodesics from an initial condition to points in $\Hp$ in a strong enough sense.
	Such a notion is needed for the discussion on competition interfaces and motivates the following definition.
	
	\begin{defn}
		Given a sample from $\La$ and an initial condition $h_0$, a point $(x,t) \in \Hp$ is a unique geodesic point,
		UGP for short, if there is a unique geodesic from $h_0$ to $(x,t)$.
		
		The initial condition $h_0$ satisfies the \bemph{unique geodesic condition} if the following holds almost surely in samples of $\La$:
		
		\emph{For every $t > 0$, the set of points $x \in \R$ such that $(x,t)$ is not a UGP for $h_0$ is countable.}
		
		The unique geodesic condition is abbreviated as \textbf{UGC}.
	\end{defn} 
	
	We are going to prove that every initial condition satisfies the UGC (Theorem \ref{thm:ugc}).
        For the remainder of this section $h_0$ is a fixed initial condition.
        The first step to establishing the UGC for $h_0$ is uniqueness of the geodesic starting point given an endpoint.
	\begin{lem} \label{lem:uniqueend}
		Let $(x,t) \in \Hp$ and $h_0$ be an initial condition. Then the following property holds almost surely in samples of $\La$. Every geodesic from $h_0$ to $(x,t)$ starts at a unique common point.
	\end{lem}
	
	\begin{proof}
		We have to show the process $y \mapsto \La(y,0;x,t) + h_0(y)$ has a unique maximizer almost surely.
		Lemma \ref{lem:exist} says there is a maximizer, so we must show there cannot be more than one.
		If there were, then one could find two disjoint intervals with rational endpoints that capture different maximizers.
		So it is enough to show that for any two disjoint compact intervals $[a,b]$ and $[c,d]$, with $b < c$, the maximums
		of the process over the intervals are different almost surely. We can assume $h_0$ takes some finite values over
		both $[a,b]$ and $[c,d]$, for otherwise there would be no maximizers there.
		
		The process $y \mapsto \La(y,0;x,t) - \La(a,0;x,t)$ is a re-centred and re-scaled Airy process.
		It is absolutely continuous with respect to a Brownian motion $B(y)$ on the interval $[a,d]$, started from
		$B(a) = 0$ and with diffusivity constant $\sqrt{2}$ \cite{CH, SV}.
		So it suffices to prove $B + h_0$ has different maximums over $[a,b]$ and $[c,d]$ almost surely.
		Observe that
		\begin{align*} & \max_{y \in [a,b]} \{B(y) + h_0(y)\} - \max_{y \in [c,d]} \{B(y) + h_0(y)\} = \\
			& \max_{y \in [a,b]} \{B(y)-B(b) + h_0(y)\} - \max_{y \in [c,d]} \{B(y)-B(c) + h_0(y)\} + B(b)-B(c).
		\end{align*}
		The three terms on the right hand side above are independent. The first two terms are finite random variables
		and the third has a Normal distribution. So the right hand side has an absolutely continuous law
		and avoids the value 0 almost surely.
	\end{proof}
	
	The second step is uniqueness of the geodesic from $h_0$ to a given point.
	\begin{lem} \label{lem:uniquegeo}
		Let $(x,t) \in \Hp$ and $h_0$ be an initial condition.
		Then the following properties hold almost surely in samples of $\La$.
		There is a unique geodesic from $h_0$ to $(x,t)$.
		The set of non UGPs from $h_0$ has empty interior.
	\end{lem}
	
	\begin{proof}
		Given a point $(x,t)$, let $g$ be a geodesic from $h_0$ to $(x,t)$.
		If the values of $g$ are almost surely unique for all rational times $s \in [0,t]$, then $g$ is unique almost surely by continuity.
		The claim about the rational times follows from a union bound if for any given $s \in [0,t)$, the value of $g(s)$ is unique almost surely.
		
		Let $\hh_s(y) = \sup_{z \in R} \{h_0(z) + \La(z,0; y,s) \}$ be the height function grown to time $s$ from $h_0$.
		Set $\hat{\La}(y,u;y',u') = \La(y, s+u; y', s+u')$ for $0 \leq u < u'$ and $y,y' \in \R$. So $\hat{\La}$  is $\La$ viewed from time $s$ onwards.
		The process $\hat{\La}$ has the same law as $\La$ and is independent of $\hh_s$ \cite{DOV}.
		The value $g(s)$ is the starting point of a geodesic from
		$\hh_s$ to $(x,t-s)$ according to $\hat{\La}$. It is then unique almost surely by Lemma \ref{lem:uniqueend}.
		
		Now that we know any given point has a unique geodesic from $h_0$ to it almost surely, it follows that almost surely
		every rational point has a unique geodesic from $h_0$ to it. The set of non UGPs from $h_0$ has an empty interior when it contains no rational point.
	\end{proof}
	
	The third observation is that geodesics from $h_0$ have a natural ordering.
	\begin{lem}[Geodesic ordering] \label{lem:geoorder}
		Let $h_0$ be an initial condition. Then the following holds almost surely in samples of $\La$.
		Suppose $g_1$ and $g_2$ are any two geodesics from $h_0$ to $(x,t)$ and $(y,t)$, respectively. If $x < y$ then $g_1(s) \leq g_2(s)$ for every $s \in [0,t]$, that is, $g_1$ stays to the left of $g_2$. If $g_1$ and $g_2$ meet then they coalesce downward from that point. As a result, no two geodesics from $h_0$ to any points of $\Hp$ can cross.
	\end{lem}
	
	\begin{proof}
		Condition on the almost sure event that the non UGPs from $h_0$ have an empty interior.
		
		If $g_1$ and $g_2$ do not meet,  then continuity and $x < y$ implies $g_1(s) < g_2(s)$ for every $s$.
		Suppose they do meet at the maximal time $s \in (0,t)$ but do not coalesce downward. Then $p = (g_1(s),s) = (g_2(s),s)$ is not a UGP.
		There is an open region of points between the graphs of $g_1$ and $g_2$ from time $s$ to $t$. We claim no point in this region can be a UGP either, thus contradicting the assumption that non UGPs from $h_0$ have empty interior.
		
		Indeed, let $q = (z,v)$ be a point in said region and choose a geodesic $\gamma$ from $h_0$ to $q$.
		By continuity, the geodesic $\gamma$ has to meet $g_1$ or $g_2$ as some time $u \geq s$. Suppose it meets $g_1$.
		Consider the following two distinct paths $\pi_1$ and $\pi_2$ from $h_0$ to $q$: $\pi_1$ equals $g_1$ on $[0,u]$ and $\gamma$ on $[u,v]$; $\pi_2$ equals $g_2$ on $[0,s]$, $g_1$ on $[s,u]$ and $\gamma$ on $[u,v]$.
		We claim $\pi_1$ and $\pi_2$ are geodesics from $h_0$ to $q$ -- so $q$ is not a UGP.
		
		The point $r = (\gamma(u),u) = (g_1(u),u)$ has three geodesics from $h_0$ to itself, namely $\gamma$
		from $h_0$ to $r$, $g_1$ from $h_0$ to $r$, and $g_2$ from $h_0$ to $p$ followed by $g_1$ from $p$ to $r$.
		The paths $\pi_1$ and $\pi_2$ are obtained by swapping the segment of $\gamma$ from $h_0$ to $r$
		by the latter two segments, respectively. Swapping geodesic segments does not change the geodesic property (Lemma \ref{lem:geoswap}),
		and thus $\pi_1$ and $\pi_2$ remain geodesics.
	\end{proof}
	
	The three previous lemmas culminate in a proof of the unique geodesic condition.
	\begin{thm}[Unique geodesic condition] \label{thm:ugc}
		The unique geodesic condition holds for a given initial condition $h_0$.
	\end{thm}
	
	\begin{proof}
		Condition again on the event that the non UGPs from the initial condition $h_0$ have an empty interior,
		which implies the geodesic ordering property from Lemma \ref{lem:geoorder}.
		Condition also on the almost sure event about existence of geodesics from $h_0$ asserted by Lemma \ref{lem:exist}.
		
		For a point $p = (x,t)$, let $A_p$ be the interior of the region between the leftmost and rightmost geodesics from $h_0$ to $(x,t)$.
		The set $A_p$ is empty if and only if $p$ is a UGP. For distinct points $p = (x,t)$ and $q = (y,t)$ with the same time coordinate,
		$A_p$ and $A_q$ are disjoint by the geodesic ordering property. The sets $A_p$ as $p$ varies over all non UGPs with
		time coordinate $t$ are then disjoint, non-empty open sets. There can only be countably many of them
		(each contains a different rational point). As $t$ is arbitrary, the UGC is established.
	\end{proof}
	
	\subsection{Good samples assumption for an initial condition} \label{sec:goodoutcomes}
	The properties of geodesics discussed so far hold over various almost sure events associated to $\La$.
	Many of them are interrelated, and it is convenient to have a single event over which they are all satisfied.
	This will be useful in the coming sections. Such an event -- let us call it the good samples -- is described below.
	
	The good samples, relative to an initial condition $h_0$, are samples of $\La$ for which all the following properties hold, which they do almost surely.
	\begin{itemize}
		\item The process $\La$ is continuous and satisfies the bound from \eqref{eqn:Abound}.
		\item The properties of point-to-point geodesics of $\La$ described in $\S$\ref{sec:Lgeodesic} hold.
            \item The geodesic ordering property from Lemma \ref{lem:geoorder} holds for geodesics from $h_0$.
		\item The unique geodesic condition for $h_0$ is satisfied.
	\end{itemize}
	
	\subsection{Where do geodesics emanate from?} \label{sec:geoendpoint}
	For $(x,t)\in \Hp$, consider the point from which the rightmost geodesic from $h_0$ to $(x,t)$ emanates: 
	\begin{equation} \label{eqn:geoendpoint}
		e(x,t) = \text{starting point of the rightmost geodesic from}\; (x,t) \;\text{to}\; h_0.
	\end{equation}
	
	It encapsulates many properties of geodesics and will also help to understand interfaces in the next section.
	Assume good samples of $\La$ relative to $h_0$.
	
	\begin{prop} \label{prop:endpointfunc}
		For every $t > 0$, the mapping $x \mapsto e(x,t)$ is a non-decreasing, right continuous step function with a discrete set of discontinuities. The left limit of $e(\cdot, t)$ at $x$ is the starting point of the leftmost geodesic to $(x,t)$. The image of $e$ is also countable.
	\end{prop}
	
	\begin{proof}
		Let $e_t(x) = e(x,t)$. Ordering of geodesics from $h_0$ in Lemma \ref{lem:geoorder} implies $e_t$ is non-decreasing. Right continuity comes from the ordering and compactness properties of geodesics from $h_0$ in Lemmas \ref{lem:ccc} and \ref{lem:geoorder}. The left limit is derived by the same reasoning.
		
		In order to show $e_t$ is a step function with a discrete set of discontinuities it is enough to show that for every $x$, there are $w$ and $y$ with
		$w < x < y$, such that $e_t$ is constant on the interval $[x,y]$ (with common value $e_t(x)$), and also on the interval $[w,x)$ (the common value being the left limit of $e_t(x)$). Now when $y > x$ is sufficiently close to $x$, the rightmost geodesic from $(y,t)$ has to meet the rightmost geodesic from $(x,t)$ by geodesic compactness and the property that nearby geodesics meet. The ordering property implies that these geodesics coalesce downward, so $e_t(x)=e_t(y)$. Similarly, for $w < x$ sufficiently close, the rightmost geodesic from $(w,t)$ has to meet and coalesce downward with the leftmost geodesic from $(x,t)$.
		
		The image of $e_t$ is countable because $e_t$ is a step function with a discrete set of discontinuities. Now observe that the image of $e_t$ is contained in the image of $e_s$ if $t > s$, since a rightmost geodesic from time $t$ contains a rightmost geodesic from time $s$. So the image of $e$ is a countable union, $\mathrm{Im}(e) = \cup_{n \geq 1} \mathrm{Im}(e_{1/n})$, of countable sets. 
	\end{proof}
	
	\subsection{Infinite geodesics of the directed landscape} \label{sec:longgeo}
	We turn to the construction and properties of infinite  geodesics of $\La$.
	A highlight of this section is Theorem \ref{thm:geotree}, which describes the family of infinite geodesics in a fixed direction.
	
	\begin{defn}\label{def:longgeo}
		A continuous function $g:[s,\infty)\mapsto \R$ is an \bemph{infinite geodesic} of $\La$ if it is a geodesic of $\La$ when restricted to every interval $[s,T]$. 
		Its starting point is $(g(s), s)$.
		
		The \bemph{direction} of an infinite geodesic $g$ is the limit of $g(t)/t$ as $t$ tends to infinity, if it exists.
	\end{defn}
	
	Infinite geodesics may be identified with their graphs lying in the spacetime plane.
	When convenient, we will call an infinite geodesic in direction $d$ simply a geodesic with direction $d$.

	\begin{thm}[Existence of infinite geodesics] \label{thm:geoexist}
		Given a starting point $p =(x,s) \in \R^2$ and direction $\alpha \in \R$, there is almost surely an infinite geodesic $g$ of $\La$ from $p$ with direction $\alpha$. Moreover, there is $a>1$ and a random variable $C$ with $Ea^{C^3}<\infty$ so that
		\begin{equation}\label{eqn:longgeo}
         |g(s+t) - x - \alpha t| \leq C \big( 1 \vee t^{2/3} (\log \log t)^{1/3}\big) \quad \text{for}\; t \geq 1.
		\end{equation}
	\end{thm}
	
	\begin{proof}
		
		Due to shear and translation invariance of $\La$, see $\S$\ref{sec:undrand},
		it is enough to consider the starting point $p = (0,0)$ and direction $\alpha = 0$.
		
		For each $n$, let $g_n(t)$ be a geodesic of $\La$ from $(0,0)$ to $(0,n)$.
		These geodesics are almost surely unique for every $n$ and the law of $g_n(t)$ is $n^{2/3}g_1(t/n)$ \cite[Theorem 1.7]{DOV}.
		The following bound on geodesics is derived from Lemma \ref{lem:geobound} below by scale invariance.
		There are constants $c$ and $d$ such that for every $t > 0$ and $n \geq t$,
		\begin{equation} \label{eqn:geobound}
			\mathbf{Pr}\left ( \frac{\sup_{s\in [0,t]}|g_n(s)|}{t^{2/3}} > \lambda \right ) \leq c e^{-d \lambda^3}.
		\end{equation}
        Indeed, $\sup_{s \in [0,t]} |g_n(s)|$ has the same law as $t^{2/3} \sup_{s \in [0,1]} |g_{n/t}(s)|$,
        and then the conclusion of Lemma \ref{lem:geobound} can be applied.
        		
        Now we argue as in the proof of the law of the iterated logarithm to establish that there are absolute constants $c'$ and $d'$ for which
		\begin{equation} \label{eqn:geobound1}
			\mathbf{Pr}\left ( \sup_{t\in [1,n]} \frac{|g_n(t)|}{1 \vee t^{2/3}(\log \log t)^{1/3}} > \lambda \right ) \leq c' e^{-d' \lambda^3}
		\end{equation}
 		for every $n$. First we note that if $\lambda$ is such that $(d/27) \lambda^3 \leq 2$ with $d$ is as in \eqref{eqn:geobound},
 		then the bound \eqref{eqn:geobound1} holds for all such $\lambda$ with any given $d'$ by choosing $c'$ sufficiently large.
 		So let us assume $\lambda$ is such that $(d/27) \lambda^3 \geq 2$, and find appropriate $d'$ and $c'$.
   
        For $0 \leq k < \log n$, let $t_k = e^k$. Define
        $$M_k = \sup_{t \in [t_k, t_{k+1}]} \frac{|g_n(t)|}{1 \vee t^{2/3}(\log\log t)^{1/3}}.$$
        By an union bound, the probability in \eqref{eqn:geobound1} is bounded above by
        $ \sum_k \mathbf{Pr}(M_k > \lambda)$. For $t \in [t_0,t_1]$, $1 \vee t^{2/3}(\log \log t)^{1/3} = 1$ and so by \eqref{eqn:geobound},
        $$ \mathbf{Pr}(M_0 > \lambda) \leq c e^{-d e^{-2}\lambda^3}.$$
        If $t \in [t_k, t_{k+1}]$ then $t^{2/3} (\log \log t)^{1/3} \geq t_k^{2/3} (\log \log t_k)^{1/3} \geq (1/3) t_{k+1}^{2/3} (\log \log t_{k+1})^{1/3}$.
        Consequently,
        $$M_k \leq 3 \; \frac{\sup_{t \in [0,t_{k+1}]} |g_n(t)|}{t_{k+1}^{2/3} (\log\log t_{k+1})^{1/3}},$$
        which by \eqref{eqn:geobound} implies that
        $$\mathbf{Pr}(M_k > \lambda) \leq c e^{-d (\lambda/3)^3\log(k+1)} = c (k+1)^{-(d/27)\lambda^{3}}.$$
        Therefore, the probability in \eqref{eqn:geobound1} is bounded above by
        $$ ce^{-d e^{-2}\lambda^3} + \sum_{k \geq 2} c k^{-(d/27)\lambda^3} \leq ce^{-d e^{-2}\lambda^3} + c'' 2^{- (d/27)\lambda^3 + 1} = c' e^{-d'\lambda^3}.$$
		
		Let $E(n,\lambda)$ be the event in \eqref{eqn:geobound1}.
		Fatou's lemma applied to the indicator of the events $E(n,\lambda)$ over $n$ implies
	    $$ \mathbf{Pr}\left ( E(n,\lambda) \; \text{occurs eventually for all }\; n\right ) \leq
	    c' e^{-d' \lambda^3}.$$
		Thus, by the Borel-Cantelli lemma applied over integer values of $\lambda$, there is a random $\Lambda$ with tail bounds as in \eqref{eqn:geobound1} such that
 		$E(n,\Lambda)^c$ occurs infinitely often along some random subsequence. 
 		By passing to such a subsequence we may assume that for every $n$
 		\begin{equation}\label{eqn:geobound2}
 			\sup_{t\in [0,n]} \frac{|g_n(t)|}{1 \vee t^{2/3}(\log \log t)^{1/3}} \leq \Lambda.
 		\end{equation}
		
		The next step is a sample-wise compactness and diagonalization argument.
		Assume good samples of $\La$ with respect to the narrow wedge initial condition, which is $h_0(0) = 0$ and $h_0(x) = -\infty$ for all $x \neq 0$.
		Thus, all geodesics from $(0,0)$, in particular every $g_n$, satisfy the good samples properties from $\S$\ref{sec:goodoutcomes}.
  
		Along the subsequence where \eqref{eqn:geobound2} holds, the graphs of the geodesics $g_n(t)$ for $t \in [0,1]$ have endpoints that remain inside a compact subset of $\overline{\Hp}$. By the compactness property of geodesics, there is a further subsequence on which $g_n$ restricted to $[0,1]$ converges to a geodesic $\gamma_1$ for times $t \in [0,1]$. After we pass to this subsequence, the same argument implies there is a further subsequence on which $g_n$  restricted to $[0,2]$ converges to a geodesic $\gamma_2$. Continuing like this for integer times and then moving to the diagonal subsequence gives the desired infinite geodesic $g$, and a subsequence along which $g_{n}$ tends to $g$ on every compact interval of times.

        Since $g$ is the limit of geodesics satisfying \eqref{eqn:geobound2},
        it satisfies the bound \eqref{eqn:longgeo}.
	\end{proof}
	
	\begin{lem} \label{lem:geobound}
    There are constants $c,d>0$ such that for every $T \geq 1$, the geodesic $g_T$ from $(0,0)$ to $(0,T)$ satisfies
    $$
    \mathbf{Pr}\left (\sup_{t\in[0,1]}|g_T(t)|>\lambda \right )\le c e^{-d\lambda^3}.
    $$
    \end{lem}

    \begin{proof}
    The proof of the lemma is based on the following two bounds. The first bound is derived
    from \cite[Proposition 12.3]{DOV} and the second is from \cite[Corollary 2.14]{DSV}.

    For every $b>0$, there exists $a>1$ and a random variable $C$ with $\mathbf{E}[a^{C^3}]<\infty$ such that for any geodesic $g$ from  a point in $[-b,b]\times\{0\}$ to a point in $[-b,b]\times\{1\}$,
    \begin{equation} \label{eqn:modulus}
    |g(s)-g(t)|<C |s-t|^{2/3}\log^{1/3} (2/|s-t|).
    \end{equation}
	
	There are also constants $c_1$ and $d_1$ such that for every $t \in (0,1]$,
	\begin{equation} \label{eqn:DSSbound}
		\mathbf{Pr}\left ( \frac{|g_1(t)|}{t^{2/3}} > \lambda \right ) \leq c_1 e^{-d_1 \lambda^3}.
	\end{equation}
	Since $g_T(t)$ has law $T^{2/3}g_1(t/T)$, it follows from \eqref{eqn:DSSbound} that for every $T$ and $t \in (0,T]$,
    \begin{equation} \label{eqn:bound}
	    \mathbf{Pr}\left ( \frac{|g_T(t)|}{t^{2/3}} > \lambda \right ) \leq c_1 e^{-d_1 \lambda^3}.
    \end{equation}
		
    Let $\gamma_x$ be a geodesic from $(0,0)$ to $(x,1)$. The bound
    $$ \mathbf{Pr}\left (\sup_{x\in[0,1],t\in [0,1]} |\gamma_x(t)|>\lambda \right ) \leq c_2e^{-d_2\lambda^3} $$
    follows from \eqref{eqn:modulus} with $s=0$. Here $c_2$ and $d_2$ are absolute constants derived from the fact that $\mathbf{E}[a^{C^3}] < \infty$.
    
    Now consider geodesics $\gamma_x$ for $x \in [k,k+1]$. By shear invariance \eqref{eqn:shear}, $\gamma_x$ has law $\gamma_{x-k}(t) + kt$.
    Therefore, using the bound above, 
    $$ \mathbf{Pr} \left(\sup_{x\in[k,k+1],t\in [0,1]} |\gamma_x(t)-tk|>\lambda \right )\le c_2e^{-d_2\lambda^3}.$$
    This bound implies that if $\lambda \geq |2k|$ then
    $$ \mathbf{Pr} \left (\sup_{x\in[k,k+1],t\in [0,1]} |\gamma_x(t)|>\lambda \right )\le c_2e^{-d_2(\lambda/2)^3}.$$
    Let $m$ be a positive integer. By a union bound for $k=-m,-m+1, \ldots, m$ in the above, it follows that for $|\lambda|>2m$,
    \begin{equation}\label{e:unionk}
    \mathbf{Pr} \left (\sup_{x\in[-m,m],t\in [0,1]} |\gamma_x(t)|>\lambda \right)
    \leq 2mc_2 e^{-(d_2/8)\lambda^3}\leq c_3e^{-d_3\lambda^3}.
    \end{equation}
    Now 
    $$\sup_{t\in [0,1]} |g_T(t)|>2m$$
    implies that either $|g_T(1)|>m$, or $g_T(1)\in [-m,m]$ and the segment of $g_T$ over $[0,1]$ violates the event in \eqref{e:unionk} with $\lambda=2m$. By \eqref{eqn:bound} and \eqref{e:unionk}, the probabilities of either of these events are bounded above as claimed. 
    \end{proof}
    
	Infinite geodesics exist from a given point simultaneously across all directions, and they are unique for typical directions.
	
	\begin{thm}[Geodesics in every direction from a point] \label{thm:longgeo}
		Let $p \in \R^2$. The following properties of infinite geodesics of $\La$ from $p$ hold almost surely.
            \begin{enumerate}
            \item Let $g$ be a geodesic in direction $d$ and $g'$ be a geodesic in direction $d'$, both from $p$. If $d < d'$, then $g(t) \leq g'(t)$ for every $t$.
            \item There is a geodesic from $p$ in direction $d$ for every $d \in \R$. There is a leftmost geodesic $g_d^-$ and a rightmost geodesic $g_d^+$ from $p$ in direction $d$ for every $d \in \R$.
            \item Let $D$ be the random set of directions such that $d \in D$ if and only if there are multiple geodesics from $p$ in direction $d$ ($g_d^- \neq g_d^+$). Almost surely, $D$ is at most countable. For any $d \in \R$, $\mathbf{Pr}(d \in D) = 0$ (there is almost surely a unique geodesic from $p$ in any fixed direction).
            \end{enumerate}
	\end{thm}
	
	\begin{proof}
		We may take $p = (0,0)$ as before. Assume good samples of $\La$ with respect to the narrow wedge initial condition at zero, which is $h_0(0)=0$ and $h_0(x)=-\infty$ for $x \neq 0$. So all geodesics from $(0,0)$ satisfy the good samples properties from $\S$\ref{sec:goodoutcomes}. Note that good samples form an almost sure event.
            Furthermore, by Theorem \ref{thm:geoexist} and a union bound, there are geodesics from $(0,0)$
		along every rational direction almost surely. So assume samples of $\La$ for which these exist as well.
            Parts (1) - (3) can now be shown sample-wise.
		
		  For Part (1), due to the good samples assumption, we have geodesic ordering from Lemma \ref{lem:geoorder}.
            If $d < d'$, then $g(t) < g'(t)$ for all large $t$ due to their respective directions.
		But then geodesic ordering implies $g(t) \leq g'(t)$ for every $t$.
  
		For Part (2), let $d$ be an arbitrary direction. Approximate $d$ by rational directions $d_{2n}$ that are decreasing
		to $d$ as well as rational directions $d_{2n+1}$ that are increasing to it. Let $g_n$ be a geodesic
		from $(0,0)$ along direction $d_n$. Geodesic ordering and the compactness property, used as in the proof of
		Theorem \ref{thm:geoexist}, imply that a subsequence $g_{n_k}$ converges to an infinite geodesic
		$g$ from $(0,0)$. To see that $g$ has direction $d$ note that for every $n$,
		by geodesic ordering, $g_{2n+1}(t) \leq g(t) \leq g_{2n}(t)$. Dividing by $t$ and taking limits imply that
		$$ d_{2n+1} \leq \liminf_{t \to \infty} \, \frac{g(t)}{t} \leq \limsup_{t \to \infty} \, \frac{g(t)}{t} \leq d_{2n}.$$
		Taking the limit over $n$ shows the direction of $g$ is indeed $d$.
		
		The leftmost and rightmost geodesics in direction $d$ are defined according to
		$$ g_d^{-}(t) = \inf_{g} g(t) \quad \text{and}\quad g_d^{+}(t) = \sup_g g(t),$$
		where $g$ ranges over all geodesics from $(0,0)$ in direction $d$. Geodesic ordering
		and compactness ensure $g_d^{\pm}$ are in fact geodesics with direction $d$.

            For Part (3), suppose there is more than one geodesic from $(0,0)$ along a direction $d$.
		Then there is a non-empty open region $A_d$ of points between the leftmost and rightmost geodesics from $(0,0)$ in direction $d$. (Let $A_d$ be empty if there is only one geodesic in direction $d$.) If $d \neq d'$ then the sets $A_d$ and $A_{d'}$ are disjoint due to geodesic ordering. So the sets $A_d$, as the direction varies, are a collection of open, pair-wise disjoint sets. Only a countable number of them can be non-empty. Thus the set $D$ of directions along which there are multiple infinite geodesics from $(0,0)$ is at most countable.

            Uniqueness of a geodesic from $(0,0)$ in a fixed direction comes for free from the above construction and
		invariances of $\La$.
		Shear invariance of $\La$  \eqref{eqn:shear} implies the set $D$ is translation invariant.
		A countable, translation invariant set has zero probability of containing any given number. Indeed,
		$$ \int_{-\infty}^{\infty} \mathbf{1}_{x \in D} \, dx = 0, $$
		and on taking expectation this gives $\int_{-\infty}^\infty\mathbf{Pr}(x \in D)\,dx =0$, so that $\mathbf{Pr}(x \in D)=0$
		for almost every $x$. But this probability does not depend on $x$.
		So any fixed direction has a unique infinite geodesic from $(0,0)$ almost surely.
		
		Note the set $D$ is also scale invariant by the KPZ scale invariance of $\La$. As such, although countable, it should  be dense over the real line.
	\end{proof}
	
	\begin{cor} \label{cor:coalesence}
		Let $p=(x,s)\in \R^2$, let $t_n\to \infty$ and let $x_n/t_n\to d \in \R$.
		Let $g_n$ be a geodesic from $p$ to $(x_n,t_n)$ and let $g$ be a geodesic from $p$ in direction $d$. Assume that $g$ is unique almost surely.  
		Then, almost surely, for every $r>s$, $g_n=g$ on $[s,r]$ for all large $n$. 
	\end{cor}	
	
	\begin{proof}
		Let $\eps > 0$ and let $g_-$ and $g_+$ be geodesics from $p$ in the directions $d\pm \eps$.
		Then for all large enough $n$, $g_-(t_n) < g_n(t_n) < g_+(t_n)$, and by geodesic ordering,
		$g_-(t)\leq g_n(t)\leq g_+(t)$ for $t \in [s, t_n]$.
		By the compactness property of geodesics, $g_n$ is precompact in the Hausdorff topology restricted to compact time intervals,
		and every  limit point $\gamma$ satisfies
		$$ d-\eps\leq \liminf_{t \to \infty} \, \frac{\gamma(t)}{t} \leq \limsup_{t \to \infty} \, \frac{\gamma(t)}{t} \leq d+\eps.$$
		Since this holds for all $\eps>0$, $\gamma$ is a geodesic in direction $d$, and by uniqueness, $\gamma=g$.
		So $g_n\to g$ on compacts.
		
		Next, observe that for every $t >s$, the segment of $g$ from $p = (g(s),s)$ to $(g(t),t)$ is the unique geodesic between these points. This is because if there were another geodesic $\pi$ from $p$ to $(g(t),t)$ then the path obtained by swapping the segment of $g$ on $[s,t]$ with $\pi$ would produce another infinite geodesic
		from $p$ in direction $d$, contradicting its uniqueness. Therefore, if $g_n(t) = g(t)$ for some time $t$ then $g_n  = g$ on $[s,t]$ due to uniqueness.
		
		Finally, observe that $g_n$ meets $g$ at arbitrarily large times. Indeed, $g_n$ tends to $g$ on any compact time interval $[t_1,t_2]$ with $t_1 \geq s$. So by the property
		that nearby geodesics meet, $g_n$ must meet $g$ at some time in $[t_1,t_2]$. As $t_1$ may be arbitrarily large, the claim follows.
  
		The two observations above imply that for any compact interval $[s,r]$, $g_n = g$ on $[s,r]$ for all large $n$.
	\end{proof}

	\subsection{The geodesic tree in a direction}
        Consider all infinite geodesics of $\La$ in a fixed direction. From typical starting points there is a unique such geodesic, but there can be exceptional points with multiple geodesics. Nevertheless, any two such geodesics will coalesce to make a single coalescing family -- the geodesic tree. We prove this result as Corollary \ref{cor:geotree} for direction zero, which then extends to any fixed direction by shear invariance of $\La$.
	
	By Theorem \ref{thm:longgeo}, for every $(x,d) \in \R^2$, there is almost surely a unique geodesic $g_{x,d}$ started from $(x,0)$ in direction d. So the event 
	\begin{equation} \label{eqn:eventE}
		E = \{\text{there is a unique geodesic}\; g_{x,d}\; \text{from}\; (x,0) \;\text{in direction}\; d\; \text{for every}\; (x,d) \in \mathbb{Q}^2\}    
	\end{equation} 
	has probability one. We will condition on the event $E$ repeatedly in this section. We note that on the event $E$ we may also assume that these geodesics are ordered, in the sense that if $x$ is rational and $d < d'$ are rational then $g_{x,d}(t) \leq g_{x,d'}(t)$ for every $t$. Indeed, geodesic ordering with respect to the direction parameter for any given starting location $(x,0)$ is established in Theorem \ref{thm:longgeo}.
	
	\begin{lem} \label{lem:geoordering2}
		The following ordering of geodesics holds on the event $E$, and thus almost surely.
		If $g$ and $g'$ are any two geodesics in direction 0 started from $(y,0)$ and $(y',0)$ respectively, then
		$g(t) \leq g'(t)$ for every $t \geq 0$ when $y < y'$.
	\end{lem}
		
		\begin{proof}
			Condition on the event $E$ in \eqref{eqn:eventE} and let $g_{x,d}$ be the unique geodesic in direction
			$d$ stated from $(x,0)$ for every $(x,d) \in \mathbb{Q}^2$.
   
                We first establish the following auxiliary fact about geodesics. Suppose $x < x'$ and $d < d'$ are real numbers. Let $\gamma$ be a geodesic from $(x,0)$ in direction $d$ and $\gamma'$ be one from $(x',0)$ in direction $d'$. Suppose one of them is unique.
			Then $\gamma(t) \leq \gamma'(t)$ for every $t \geq 0$.
			
			To see why this is so, suppose to the contrary that $\gamma$ and $\gamma'$ cross.
			Then there will be times $0 < s < t$ such that $\gamma(s) = \gamma'(s)$, $\gamma(t) = \gamma'(t)$
			and $\gamma(u) > \gamma'(u)$ on $(s,t)$. In other words, if the geodesics cross then they must cross back. This follows from continuity, the assumption that the geodesics cross, while their respective directions $d < d'$ forces them to cross back. Observe that the segments of $\gamma$ and $\gamma'$ on $[s,t]$
			are both (finite) geodesics between their common respective endpoints.
			
			Now suppose $\gamma$ is unique. Consider the path $\pi$ obtained by following $\gamma$ on $[0,s]$,
			$\gamma'$ on $[s,t]$ and $\gamma$ on $[t, \infty)$. The path is obtained by swapping the geodesic segment of
			$\gamma$ on $[s,t]$ with the segment from $\gamma'$. As such, it is also an infinite geodesic from $(x,0)$ in direction $d$, which contradicts the uniqueness of $\gamma$. Thus, $\gamma \leq \gamma'$ on $[0,\infty)$.
			
			Now consider the geodesics $g$ and $g'$ from the statement of the lemma. Pick any rational number $x \in (y,y')$. Let $g_{x,0}$ be the unique geodesic from $(x,0)$ in direction 0. We will infer that $g(t) \leq g_{x,0}(t) \leq g'(t)$ for every $t$ to deduce the lemma. Let us prove that $g(t) \leq g_{x,0}(t)$ as the other inequality is analogous.
			
			Consider the sequence of directions $d_n = 1/n$ and the unique geodesics $g_{x,d_n}$ from $(x,0)$ with direction $d_n$.
			The aforementioned fact implies $g(t) \leq g_{x,d_n}(t)$ for every $t$ and every $n$. Due to geodesic
			ordering and compactness, along a subsequence the geodesics $g_{x,d_n}$ converge to some geodesic from $(x,0)$ in direction 0. Since there is a unique such geodesic $g_{x,0}$, the limit must be it.
			Therefore, $g(t) \leq g_{x,0}(t)$ for every $t$.
		\end{proof}
		
	\begin{lem} \label{lem:twogeocoalesce}
	   The following coalescence property of geodesics hold on the event $E$, and thus almost surely.
	   Let $g$ and $g'$ be geodesics in direction 0 from $(y,0)$ and $(y',0)$ respectively with $y < y'$.
	   Suppose one of them is unique. If $g$ and $g'$ meet, they coalesce from that point onward.
	\end{lem}
	
	\begin{proof}
	    Condition on the event $E$ in \eqref{eqn:eventE}. Due to the geodesic ordering property from Lemma \ref{lem:geoordering2},
	    $g \leq g'$. Suppose $g$ and $g'$ meet and let $p = (w,t)$ be the earliest meeting point. We will show that $g=g'$ on $[t,\infty)$. Suppose $g$ is unique. Let $\pi$ be the path obtained by following $g$ on $[0,t]$ and $g'$ on $[t,\infty)$. We will show that $\pi$ is also an infinite geodesic. Since $g$ is the unique geodesic from $(y,0)$ in direction 0, it follows that $g = \pi$, and hence $g$ and $g'$ coalesce as stated.
	    
	    In order to show that $\pi$ is an infinite geodesic it suffices to show that for every $T > t$, $\pi$ restricted to $[0,T]$ is a geodesic. Let $z = g(T)$ and $z' = g'(T)$; note $z \leq z'$. Let $\gamma$ be the rightmost geodesic from $(y,0)$ to $(z',T)$.
	    We show below that $g \leq \gamma$ on $[0,T]$. Assuming this, let $s$ be the earliest meeting time of $\gamma$ and $g'$ on $[0,T]$. Then on $[0,s)$, $g \leq \gamma < g'$. Since $g(t) = g'(t)$, it follows that $s \leq t$. The restrictions of $\gamma$ and $g'$ to $[s,T]$ are both geodesics from $(g'(s),s)$ to $(z',T)$. Therefore, the concatenation $\gamma'$ of $\gamma$ on $[0,s]$ followed by $g'$ on $[s,T]$ is also a geodesic. Next, the restrictions of $g$ and $\gamma'$ to $[0,t]$ are both geodesics from $(y,0)$ to $(g(t),t) = (g'(t),t) = (\gamma'(t),t)$. Thus the concatenation $\gamma''$ of $g$ on $[0,t]$ followed by $\gamma'$ on $[t,T]$ is also a geodesic. But by construction, $\gamma'' = \pi$ on $[0,T]$.
	    
	    To complete the proof we show that $g \leq \gamma$ on $[0,T]$. If this is not the case then there are times $0 \leq u < v \leq T$ such that $g(u) = \gamma(u)$, $g(v) = \gamma(v)$ and $\gamma < g$ on $(u,v)$. This follows from continuity and the fact that $g(0) = \gamma(0) = y$ while $g(T) = z \leq z' = \gamma(T)$. The restrictions of $g$ and $\gamma$ to $[u,v]$ are two distinct geodesics from $(g(u),u)$ to $(g(v),v)$. By swapping the segment of $\gamma$ on $[u,v]$ with $g$, we find another geodesic from $(y,0)$ to $(z',T)$ that is more to the right than $\gamma$, contradicting the definition of $\gamma$.
	    
	\end{proof}
	
	\begin{lem} \label{lem:geocoalesce}
		The following coalescing of geodesics holds almost surely.
		If $g$ and $g'$ are any two geodesics in direction 0 started from the real axis,
		then they coalesce upward, meaning there is a time $t$ such that $g(s) = g'(s)$ for $s \geq t$.
	\end{lem}
	
	\begin{proof}
		Condition on the almost sure event $E$ in \eqref{eqn:eventE} and on the properties of $\La$-geodesics discussed in $\S$\ref{sec:Lgeodesic}. Recall the unique geodesics $g_{x,0}$ for $x \in \mathbb{Q}$ from the event $E$.
		
		Due to the ordering of geodesics property from Lemma \ref{lem:geoordering2},
		it is enough to show the geodesics $g_{x,0}$ and $g_{x',0}$ coalesce for every $(x,x') \in \mathbb{Q}^2$.
		By Lemma \ref{lem:twogeocoalesce}, it is enough to show that $g_{x,0}$ and $g_{x',0}$ meet for every $(x,x') \in \mathbb{Q}^2$.
		By a union bound, it suffices to show this for every specific pair of rationals $x < x'$.
		
		For rationals $x < y$ consider the event
		$$ C_{[x,y]} = \{ g_{x,0} \; \text{meets}\; g_{y,0}\}.$$
		Thus we need to show that $\mathbf{Pr}(C_{[x,y]}) = 1$ for every rational $x < y$. However,
		due to translation and scale invariance of $\La$, $C_{[x,y]}$ has the same probability for
		every interval $[x,y]$. So the lemma follows from having $\mathbf{Pr}(C_{[0,1]}) = 1$.
		
		Observe that $C_{[x,y]}$ is non-increasing in that $C_{[x,y]} \subset C_{[x',y']}$ if $[x',y'] \subset [x,y]$.
		This is due to geodesic ordering. Thus, $\mathbf{Pr}(C_{[0,1]}) = \mathbf{Pr}(C_{[0,1/n]})$ for every integer $n \geq 1$ and $C_{[0,1/n]} \subset C_{[0, 1/(n+1)]}$. As a result,
		$$ \mathbf{Pr}(C_{[0,1]}) = \lim_{n \to \infty} \mathbf{Pr}(C_{[0,1/n]}) = \mathbf{Pr}(\cup_n C_{[0,1/n]}).$$
		The event $C_{[0,1/n]}$ occurs for sufficiently large $n$, so the probability above is one. Indeed, consider $g_{0,0}$ and $g_{1/n,0}$ only up to time 1. By geodesic compactness and ordering, $g_{1/n,0}$ tends to $g_{0,0}$ for times $t \in [0,1]$ as $n$ goes to infinity. Then, due to the property that nearby geodesics meet, they must in fact meet by time 1 for all large $n$.
	\end{proof}
	
	\begin{thm} \label{thm:geotree}
		 Consider all infinite geodesics of $\La$ in direction zero started from points along the real axis. It has the following almost sure properties.
		\begin{enumerate}
			\item There is a geodesic in direction zero from every point $(x,0)$ for $x \in \R$. There are also
			leftmost and rightmost geodesics $g^{-}_x$ and $g^{+}_x$ in direction zero starting from every $(x,0)$.
			
			\item If $x < y$ then $g^{+}_x(t) \leq g^-_y(t)$ for every $t$.
			
			\item There are only countably many $x \in \R$ for which there is more than one geodesic from $(x,0)$
			in direction zero (that is, $g^{-}_x \neq g^{+}_x$).
			
			\item Every pair of geodesics started from the real axis along direction zero coalesce upward.
                More precisely, if $g$ and $g'$ are two such geodesics and $g(0) \neq g'(0)$, then there is $t > 0 $ such that $g-g'$ is non-zero on $[0,t)$ and zero on $[t,\infty)$. For the leftmost and rightmost geodesics from $(x,0)$, there is $t \geq 0$ such that $g^-_x < g^+_x$ on $(0,t)$ and $g^-_x = g^+_x$ on $[t,\infty)$.
		\end{enumerate}
	\end{thm}
	
	\begin{proof}
		Condition on the almost sure event $E$ in \eqref{eqn:eventE} and on the properties of $\La$-geodesics
		from $\S$\ref{sec:Lgeodesic}. These ensure the conclusion of Lemmas \ref{lem:geoordering2}, \ref{lem:twogeocoalesce} and \ref{lem:geocoalesce}. Recall for rational $x$, $g_{x,0}$ denotes the unique geodesic in direction zero started from $(x,0)$.
		
		Part (1) is proved in the same way as Theorem \ref{thm:longgeo}. The unique geodesics $g_{x,0}$ for rational $x$
		can be used to build geodesics from every starting point at time zero, sample-wise, by the ordering and compactness property of geodesics. The leftmost and rightmost geodesics from $(x,0)$, for $x \in \R$, are the pointwise infimum and supremum, respectively, of all geodesics from $(x,0)$ along direction zero.
		
		Part (2) follows from Lemma \ref{lem:geoordering2}. The proof of part (3) is the same as in Theorem \ref{thm:longgeo}.
  
            For the proof of Part (4), let $g(0) = x$ and $g'(0) = x'$, and suppose $x < x'$. Let $y \in (x,x')$ be rational and consider the unique geodesic $g_{y,0}$ from $(y,0)$ in direction zero. By Lemma \ref{lem:geocoalesce}, there is an earliest time $t > 0$ such that $g(t) = g'(t)$.
            By geodesic ordering, $g(t) = g_{y,0}(t) = g'(t)$. Then by Lemma \ref{lem:twogeocoalesce}, $g = g_{y,0} = g'$ on $[t,\infty)$. Thus $g < g'$ on $[0,t)$ and $g = g'$ on $[t,\infty)$. Now suppose $x = x'$ and consider the leftmost and rightmost geodesics $g^-_x$ and $g^+_x$. These two geodesics coalesce by Lemma \ref{lem:geocoalesce}. So it suffices to show that if $t > 0$ is any time such that $g^-_x(t) = g^+_x(t)$, then $g^-_x = g^+_x$ on $[t,\infty)$. Consider the restrictions of $g^-_x$ and $g^+_x$ to $[0,t]$. By geodesic ordering, compactness, and the fact that nearby geodesics meet, there is a rational $y < x$ such that the unique geodesic $g_{y,0}$ meets $g^-_x$ at some time $t' \leq t$. Similarly, there is a rational $z > x$ such that $g_{z,0}$ meets $g^+_x$ at some time $t'' \leq t$. By Lemma \ref{lem:twogeocoalesce}, $g_{y,0} = g^-_x$ on $[t',\infty)$ and $g_{z,0} = g^+_x$ on $[t'', \infty)$. But then $g_{y,0} = g_{z,0}$ on $[t,\infty)$ by Lemma \ref{lem:twogeocoalesce} since $g_{y,0}(t) = g_{z,0}(t)$. Consequently, by geodesic ordering, $g^-_x = g^+_x$ on $[t,\infty)$.
	\end{proof}
	
	\begin{cor}[The geodesic tree in a direction] \label{cor:geotree}
		The geodesic tree of $\La$ in direction zero consists of all infinite geodesics in direction zero started from every point in $\R^2$. It has the following properties.
		\begin{enumerate}
			\item Almost surely, for every $p \in \R^2$ there is a leftmost geodesic $g_p^{-}$ and rightmost geodesic $g_p^{+}$ from $p$ in direction zero.
			
			\item Given $p \in \R^2$, the geodesic from $p$ in direction zero is unique almost surely.
			
			\item The following coalescence property holds almost surely. If $g$ and $g'$ are any geodesics from $(x,t)$ and $(x',t')$ in direction zero with $t' \leq t$, then there is an $s \geq t$ such that $g = g'$ on $[s, \infty)$. Also, if $g(t) \neq g'(t)$ then there is an $s > t$ such that  $g-g'$ is non-zero on $[t,s)$ and zero on $[s,\infty)$.

                \item The following holds almost surely. Let $p = (x,t)$ and $q = (x',t')$ with $t' \leq t$. The rightmost geodesics $g_p^+$ and $g_q^+$ have the property that if $g_p^+(t) = g_q^+(t)$ then $g_p^+ = g_q^+$ on $[t,\infty)$, and if $g_p^+(t) \neq g_q^+(t)$ then there is an $s > t$ such that $g^+_p - g^+_q$ is non-zero on $[t,s)$ and zero on $[s,\infty)$. The analogous property holds for leftmost geodesics.
		\end{enumerate}
	\end{cor}
	
	\begin{proof}
		By Theorem \ref{thm:geotree} and translation invariance of $\La$, the geodesics $g^{\pm}_p$ exist almost surely for every $p = (x,s)$ with $x \in \R$ and $s$ rational. Condition on the almost sure event that ensures the conclusions from Theorem \ref{thm:geotree} at all rational times.
  
            For an arbitrary point $p$, define
		$$ g_p^{-}(t) = \lim_{\eps \to 0} \, \inf_{g: |g^0-p| < \eps} g(t)$$
		where $g^0$ is the starting point of $g$, and the infimum is over all geodesics $g$ along direction 0 starting at rational times. The limit exists and is a geodesic by geodesic compactness, and has direction zero due to geodesic ordering. The geodesic $g_p^+$ is defined by changing the infimum to a supremum.
		
		Part (2) is proved in Theorem \ref{thm:longgeo}. Part (3) follows from Theorem \ref{thm:geotree} by considering the restrictions of $g$ and $g'$ to intervals $[t'',\infty)$ for rational $t'' > t$ and using continuity.
  
        For Part (4), if $g^+_p(t) \neq g^+_q(t)$ then the assertion follows from Part (3). If $g^+_p(t) = g^+_q(t)$ then $g^+_p \geq g^+_q$ on $[t,\infty)$ because $g^+_p$ is the rightmost geodesic from $p = (x,t)$ in direction zero. By Part (3), there is an $s \geq t$ such that $g^+_p = g^+_q$ on $[s,\infty)$. Suppose, by way of contradiction, that there is an $s' \in (t,s)$ with $g_p^+(s') > g^+_q(s')$. Then swapping the segment of $g^+_q$ on $[t,s]$ with that of $g^+_p$ exhibits another geodesic from $q$ in direction zero that it more to the right. Therefore, it must be that $s=t$.
	\end{proof}
	
	\subsection{Busemann functions and the law of the infinite geodesic} \label{sec:Busemann}
	
	The geodesic tree allows us to define and study the Busemann function of the directed landscape in a fixed direction.
	In turn, Busemann functions help to identify the law of an infinite geodesic as stated in Corollary \ref{cor:geolaw}.
	Other notable facts are that the Busemann function in a fixed direction and at a fixed time has a Brownian law (Corollary \ref{c:Busemann-Brownian}),
	and that backwards in time it evolves as a KPZ fixed point (Theorem \ref{thm:Busemann-KPZ}).
	
	First, we define the Busemann function.
	For $p\in \R^2$ let $g^+_p$ be the rightmost geodesic from $p$ in direction zero,
	and let $\kappa(p,q)$ be the point where the geodesics $g^+_p$ and $g^+_{q}$ coalesce upward. 
	Define
	\begin{equation} \label{eqn:BusemannW}
		W(p;q)=\La(p;\kappa(p,q))-\La(q;\kappa(p,q)).
	\end{equation}

	Note that if $\kappa'$ is any point on the geodesics after they coalesce then $W(p,q) = \La(p,\kappa') - \La(q, \kappa')$ because,
	by the sharpness of the reverse triangle inequality along geodesics, $\La(p,\kappa') = \La(p, \kappa(p,q)) + \La(\kappa(p,q), \kappa')$ and likewise for $\La(q, \kappa')$.
	It follows that $W(p;q)=-W(q;p)$ and  $W(p;q)+W(q;r)=W(p;r)$. So all values can be expressed from
	$$W(p):=W(p;0,0),$$
	which is the {\bf Busemann function} (in direction 0 with basepoint $(0,0)$) of $\La$. The Busemann function $W_d$
	in direction $d$ may be defined by considering geodesics in direction $d$ instead.
	
	Clearly $(0,0)$ is just a reference point where $W$ is anchored, and for any  $q\in \mathbb R^2$,
	by translation invariance of $\La$, as functions on $\mathbb R^2$,
	\begin{equation}\label{e:W-shift}
		W(\cdot+q)-W(q)\;\;\stackrel{law}{=} W(\cdot).
	\end{equation}
	It also follows from the definition and time stationarity of $\La$ that $t \mapsto W(\cdot,t)-W(0,t)$ is a stationary process.
	
	The following proposition gives an alternate definition of $W$ without the use of geodesics. 
	
	\begin{prop} \label{prop:Busemann}
		Almost surely, the following holds. 
		Let $t_n\to\infty$, $x_n/t_n\to 0$ and set $q_n=(x_n,t_n)$. Then for every compact interval $I$ and $t \in \R$, 
		$$W(p)=\mathcal L(p;q_n)-\mathcal L(0,0;q_n)\qquad \mbox{ for all } p\in I\times \{t\}$$ for all sufficiently large $n$.
	\end{prop}
	\begin{proof}
		Let $[y_1,y_2]$ be an interval containing $I$ so that there are unique geodesics $g_i$ in direction zero from $(y_i, t)$.
		There is almost surely a unique geodesic from $(0,0)$ in direction zero.
		Let $(z,r)$ be where all three of these geodesics coalesce for the first time. 
		
		By Corollary \ref{cor:coalesence} there is  an $n_0$  so that for $n \ge n_0$, the geodesics $g_{n,i}$ from $(y_i,t)$ to $q_n$
		as well as the geodesic $g_n$ from $(0,0)$ to $q_n$ satisfy  $g_{n,i}(r)= g_n(r)= z$.
		By geodesic ordering, any geodesic $g$ from any $p \in [y_1,y_2]\times \{t\}$ to $q_n$ satisfies $g(r)=z$.
		This implies that for such $p$ and $n \geq n_0$,
		$$
		\mathcal L(p;q_n)-\mathcal L(0,0;q_n)=\mathcal L(p;z,r)-\mathcal L(0,0;z,r)=W(p).
		$$ 
	\end{proof}
	
	\begin{cor}\label{c:Busemann-Brownian}
		For every $t$, the process $x \mapsto W(x,t)-W(0,t)$ is a two-sided Brownian motion with diffusivity $\sqrt{2}$. 
	\end{cor}
	\begin{proof}
		By stationary, it suffices to show this for $t=0$. By Proposition \ref{prop:Busemann},
		$W(x,0)$ restricted to compacts is the $n\to \infty$ limit of 
		$$W_n(x)=\mathcal L(x,0;0,n^3)-\mathcal L(0,0;0,n^3).$$
		By the definition and scaling invariance of $\La$, the process $W_n(x)$ has the law of
		$x\mapsto n[\Ai(x/n^2) - \Ai(0)] - x^2/n^3$ where $\Ai$ is the stationary Airy-two  process. 
		
		By the locally Brownian property of the Airy-two process \cite[Corollary 4.2]{CH}, $x \mapsto n[\Ai(x/n^2) - \Ai(0)]$ converges in law
		to a two-sided Brownian motion with diffusivity $\sqrt{2}$. The topology of convergence is that of uniformly on compacts.
		The term $x^2/n^3$ vanishes uniformly on compacts as $n \to \infty$. Thus $x \mapsto W(x,0)$ is Brownian. 
	\end{proof}
	
	Backwards in time, the Busemann function evolves as the KPZ fixed point. 
	
	\begin{thm}[Busemann law]\label{thm:Busemann-KPZ}
		For any pair of times $s<t$ and $x \in \mathbb{R}$,
		\begin{equation}\label{e:Busemann-metric}
			W(x,-t) = \sup_{y \in \mathbb{R}}\,  \{\La(x,-t; y,-s)+W(y,-s)\}. 
		\end{equation}
		The maximum is attained exactly at those $y$ for which $(y,-s)$ is on a geodesic from $(x,-t)$ in direction 0. 
		
		Moreover, the process  $(W(x,-t);\ x\in \R, t \geq 0)$ is  a KPZ fixed point with initial condition $W(\cdot,0)$, a two-sided Brownian motion with diffusivity $\sqrt{2}$.
	\end{thm}

	\begin{proof}
		Let $x,y \in \R$ and $s,t \in \R$ with $s < t$.
            Set $p = (x,-t)$ and $q=(y,-s)$. Let $u$ be the maximum of the time coordinates of $\kappa(p,(0,0))$, $\kappa(q, (0,0))$
		and the time when the leftmost geodesic from $p$ in direction zero meets $g_{(0,0)}$ (the unique geodesic from $(0,0)$ in direction zero). 
		(The latter is the ``left'' version of the time of $\kappa(p, (0,0))$, needed here since the geodesic from $p$ may not be unique).
		Let $r=(g_{(0,0)}(u),u)$. Since $g^+_p, g^+_q$ and $g_{(0,0)}$ have coalesced by time $u$ at point $r$,
		$$W(p)=\mathcal L(p;r)-\mathcal L(0,0;r) \quad \text{and} \quad W(q)=\mathcal L(q;r)-\mathcal L(0,0;r).$$
		The reverse triangle inequality applied to $p,q,r$ implies
		$$
		W(p)\geq \La(p,q)+W(q),
		$$
		with equality if and only if $q$ is on a geodesic from $p$ to $r$. Applying this for all $y \in \R$ gives \eqref{e:Busemann-metric}.
		Now by geodesic ordering, all geodesics from $p$ meet $g_{(0,0)}$ by time $u$, so $q$ is on a $p$ to $r$ geodesic if and only if $q$ is on an infinite geodesic from $p$ in direction zero. 
		
		By definition, $W(\cdot,0)$ is contained in the $\sigma$-field $\mathcal F_{\ge 0}$ where
		\begin{equation}\label{e:latesigma}
		\mathcal F_{\ge r}=\sigma \left (\mathcal L(x,s;y,t)\;:\; x,y\in \mathbb R, s,t\ge r \right ).
		\end{equation}
		Since  the values of $\mathcal L(x,-t;y,0)$ for $t > 0$ are independent of $\mathcal F_{\ge 0}$, $W(\cdot,-t)$ for $t > 0$ evolves as a KPZ fixed point with initial condition $W(\cdot, 0)$ due to \eqref{e:Busemann-metric}. The law of $W(\cdot,0)$ is given by Corollary \ref{c:Busemann-Brownian}.
	\end{proof}
	
	It follows from Theorem \ref{thm:Busemann-KPZ}, Corollary \ref{c:Busemann-Brownian}, and the stationarity
	of $t \mapsto W(\cdot, t) - W(0,t)$ that a two-sided Brownian initial condition with diffusivity $\sqrt{2}$
	is, up to an additive constant, a stationary distribution for the KPZ fixed point.
	Furthermore, Theorem \ref{thm:Busemann-KPZ} and the continuity of the KPZ fixed point implies $W$ is almost surely
	continuous on the lower half plane.  Due to the shift-invariance \eqref{e:W-shift}, $W$ is then almost surely continuous everywhere.
	
	When $(x,t)=(0,0)$, the first part of Theorem \ref{thm:Busemann-KPZ} can be interpreted as follows.
	For any time $s \geq 0$, $g_{(0,0)}$ restricted to $[0,s]$ is just the geodesic from $(0,0)$ to a ``final condition'' $W(\cdot,s)$
	at time $s$. By reversing time, it implies the following. 
	
	\begin{cor}[Law of the infinite geodesic] \label{cor:geolaw}
		Let  $g_{(0,0)}$ be the infinite geodesic of $\La$ from $(0,0)$ in direction zero.
		For $s>0$, let $\gamma_s$ denote the geodesic from the stationary two-sided Brownian initial condition with diffusivity $\sqrt{2}$ to $(0,s)$.
		As random continuous functions from $[0,s]\to \mathbb R$,
		$$
		g_{(0,0)}(\cdot) \stackrel{law}{=} \gamma_s(s-\cdot).
		$$
	\end{cor}
	
	The distribution function of $g_{(0,0)}(1)$ is the well-known KPZ scaling function \cite{PS}.
	
	\begin{rem}[Busemann functions in other directions]\label{r:Busemann-direction}
		For any direction $d$, the Busemann function $W_d$ in direction $d$ is related in law to $W_0$ through the shear invariance of the directed landscape. The  two-sided Brownian law of $W_d(\cdot,0)$ picks up a drift of $2d$. In fact, the reverse quadrangle inequality (see Figure \ref{fig:compete}) and Proposition \ref{prop:Busemann} imply that for any $t$ and $x<y$, the increment $W_d(y,t)-W_d(x,t)$ is non-decreasing in $d$. This means for any fixed $t$, the function $(d,x)\mapsto W_{d}(x,t)$ is the cumulative distribution function of a random measure on the plane, namely the measure of the rectangle $(c,d]\times(x,y]$ equals $W_d(y,t) - W_d(x,t) - W_c(y,t) + W_c(x,t)$. (Two-dimensional CDFs can be irregular: $F(d,x)$ and $F(d,x)+g(d)+h(x)$ are CDFs describing the same measure, for any $g,h$, so the fact that the law of $W_d(\cdot,t)$ is Brownian motion is not an obstacle. Moreover, for infinite measures, there seems to be no canonical way to choose $g$, $h$.)
        We expect this to be the limiting version of the shock measure introduced in \cite{DV}. 
	\end{rem}

	\subsection{Absolute continuity of geodesics} \label{sec:geoabscont}
	There is almost surely a unique geodesic of $\La$ from a given point $p = (y,s)$ to another point $q = (x,t)$ when $s < t$ \cite{DOV}.
	In this section we will show that for two unique geodesics with a common endpoint, the law of the shorter one is absolutely continuous with
	respect to the longer geodesic. This technical fact will be utilized later in the proof of Theorem \ref{thm:intgeo}.
	
	We will need a fact about the overlap between geodesics. For two geodesics $\gamma$ and $\pi$, the overlap is the set of times $t$ where $\gamma(t) = \pi(t)$.
	For a geodesic $\gamma$ that is the unique geodesic between its endpoints, its overlap with any other geodesic $\pi$ is an interval. Furthermore, if a sequence of geodesics $\gamma_n$ converges to $\gamma$, and $\gamma$ is unique, then their overlaps tend to the full interval of times where $\gamma$ is defined. These facts, which will be used below,  are proved in \cite[Lemma 3.3]{DSV}.
	
	\begin{prop} \label{prop:fingeoabscont}
		Let $\gamma$ be the unique geodesic from $q=(0,0)$ to $r=(x_1,t_1)$ with $t_1 > 0$, and let $\pi$ be the unique  geodesic from $p=(x_0,t_0)$ to $r$ with  $t_0<0$. Let $s\in (0,t_1)$.  Then $\gamma$ restricted to $[s,t_1]$ is absolutely continuous with respect to $\pi$ restricted to $[s,t_1]$.
	\end{prop}

	\begin{proof}
		Let $\gamma'$ and $\pi'$ be the restrictions of $\gamma$ and $\pi$ to $[s,t_1]$, respectively.
		Let $\mathcal F$ be the sigma algebra generated by $\La$ between times $0$ and $t_1$.
		
		We will show that
		\begin{equation} \label{eqn:gammapi}
			\mathbf{Pr}(\gamma'=\pi' \,|\,\mathcal F)>0 \quad \text{almost surely}.
		\end{equation}
		This suffices for absolute continuity by the following argument.
		Let $\pi_1,\pi_2,\ldots$ be conditionally independent samples from the conditional distribution of $\pi$ given $\mathcal F$
		and let $\pi'_k$ be their restrictions to $[s,t_1]$.  There is a random, finite first time $T$ so that $\pi'_T=\gamma'$,  as
		there is a positive probability of this occurring for each sample. Then for any set $A$,
		$$ \mathbf{Pr}(\gamma'\in A)=\sum_{k=1}^\infty \mathbf{Pr}(\pi'_k\in A,T=k).$$
		This means that if $A$ is a nullset for $\pi'$, then the right hand side is zero, and $A$ is also a nullset for $\gamma'$ as required.
		
		Now, using the fact about overlaps between geodesics mentioned above, there is an $\eps > 0$ such that all geodesics
		from $(x,0)$ to $r$ with $|x| \leq \eps$ coincide with $\gamma$ on $[s,t_1]$. Indeed, as $\eps \to 0$ the geodesics
		from $(x,0)$ to $r$ tend to $\gamma$ by compactness, and since $\gamma$ is unique, their overlap must eventually occupy $[s, t_1]$.
		Therefore, to conclude \eqref{eqn:gammapi} it is enough that for any $\eps > 0$ one has
		$\mathbf{Pr}(\pi(0)\in [-\eps,\eps]\,|\,\mathcal F)>0$ almost surely.
		 
		 In fact $\pi(0)$ is the maximizer of the function
		 \begin{equation} \label{eqn:airysum}
		 	x \mapsto \mathcal L(p;x,0)+\mathcal L(x,0,r).
	 	\end{equation}
		 The two summands are independent, continuous, and have finite maxima almost surely. The second is $\mathcal F$-measurable;
		 we will think of it as a fixed function.
		 
		 The first summand is a re-scaled and re-centred Airy process, the top line of an Airy line ensemble. We condition on all values of the Airy line ensemble except for the top line in $[-\eps,\eps]$. The Brownian Gibbs property gives that the conditional distribution is a Brownian bridge $B$ conditioned to be greater than the second to top line on $[-\eps,\eps]$ \cite{CH}. Let $m$ be the maximum of \eqref{eqn:airysum} outside $[-\eps,\eps]$, and let $h$ be the minimum of the second summand on $[-\eps,\eps]$. If $\max_{[-\eps,\eps]} B> m-h$, then \eqref{eqn:airysum} is maximized in $[-\eps,\eps]$. The conditional probability of this is always positive.
	\end{proof}
	
	There is a parallel result about the absolute continuity of infinite geodesics in a common direction. Recall there is almost surely a unique infinite geodesic from a given point in a given direction.
	
	\begin{prop} \label{prop:infgeoabscont}
		Let $g_0$ be the unique infinite geodesic of $\La$ from $(0,0)$ in direction zero.
		Let $g_p$ be the unique infinite geodesic of $\La$ from $p = (x_0,t_0)$ with $t_0 < 0$ also in direction zero.
		For any $s > 0$, the law of the restriction of $g_0$ to $[s, \infty]$ is absolutely continuous with respect to the restriction of $g_p$ to $[s,\infty]$.
	\end{prop}
	
	\begin{proof}
		Let $\mathcal F = \mathcal F_{\geq 0}$ be the sigma algebra generated by the restriction of $\La$ to times after 0, see \eqref{e:latesigma}.
		Let $g_0'$ and $g_p'$ be the restrictions of $g_0$ and $g_p$ to times in $[s,\infty]$. It is enough to show that
		$\mathbf{Pr}(g_0' = g_p' \, |\, \mathcal F) > 0$ almost surely.
		
		Using the fact about geodesic overlaps, there is an $\eps >0$ such that any geodesic from $(x,0)$ in direction zero with $|x| \leq \eps$ coincides with $g_0$ on $[s,\infty]$. Indeed, due to uniqueness of $g_0$, the geodesic from $(x,0)$ coalesces with $g_0$ once they meet, and they must meet before time $s$ if $\eps$ is small enough by the overlap condition. We have to prove that for any $\eps$,  $\mathbf{Pr}(g_p(0)\in [-\eps,\eps]\,|\, \mathcal F)>0$ almost surely.
		
		The location $g_p(0)$ is the maximizer of
		\begin{equation} \label{eqn:airyWsum}
			x \mapsto \La(p; x,0) + W(x)
		\end{equation}
	where $W(x) = \lim_n \La(x,0;0,n) - \La(0,0;0,n)$. This limit exists by Proposition \ref{prop:Busemann}, is $\mathcal F$-measurable and has the law of a Brownian motion by Corollary \ref{c:Busemann-Brownian}. We think of $W$ as a fixed function. The two summands are independent, continuous and the sum has a finite maximum $M$ outside of $[-\eps,\eps]$ almost surely. 
	
	The first summand is a re-scaled and re-centred Airy process, the top line of an Airy line ensemble. We condition on all values of the Airy line ensemble except for the top line in $[-\eps,\eps]$. The Brownian Gibbs property gives that the conditional distribution is a Brownian bridge $B$ conditioned to be greater than the second to top line on $[-\eps,\eps]$. Let $H=\min_{[-\eps,\eps]} W$. If $\max_{[-\eps,\eps]} B>M-H$, then \eqref{eqn:airyWsum} is maximized in $[-\eps,\eps]$. The conditional probability of this is always positive. 
	\end{proof}

	\subsection{Questions}
	
	\paragraph*{\textbf{Law of the geodesic starting point}}
	Let $h_0$ be an initial condition. The geodesic $g$ from $h_0$ to $(0,1)$ is unique almost surely.
	Compute the law of its starting point  $g(0)$ (which is also $e(0,1)$ in the notation of $\S$\ref{sec:geoendpoint}).
	In particular, does its distribution function have a determinantal form?
	
	The  law of the starting points for the flat and stationary initial conditions have been computed in determinantal form \cite{MFQR, PS}.
	
	\paragraph*{\textbf{Law of the coalescence point}}
	Study the law of the random point $\kappa(p,q)$ where the geodesic from $p$ in direction 0 first coalesces
	with the geodesic from $q$ in direction zero. In would be interesting to compute the law of $\kappa(p,q)$ exactly.
	See \cite{Pim} for results on last passage percolation models.
	
	\section{Competition interfaces} \label{sec:Interfaces}
	
	This section studies competition interfaces of the directed landscape.
	Let $h_0$ be an initial condition as in $\S$\ref{sec:KPZfixedpt}. Throughout this discussion,
	assume good samples of $\La$ relative to $h_0$ as in $\S \ref{sec:goodoutcomes}$, over which one may work sample-wise.
	
	\subsection{The competition function} \label{sec:compfunc}
	A point $p \in \R$ is a {\bf reference point} for an initial condition $h_0$ if there are points
	$a \leq p \leq b$ such that $h_0(a)$ and $h_0(b)$ are finite. The point $p$ is an {\bf interior reference point} if this holds for some $a<p<b$. 
	Define the left and right height functions split at a reference point $p$ as
	\begin{equation} \label{eqn:split}
	h^{-}_p(x) = \begin{cases} h_0(x) & \text{if}\; x \leq p \\ -\infty & \text{if} \; x > p \end{cases} \quad \text{and}\quad
	h^{+}_p(x) = \begin{cases} h_0(x) & \text{if}\; x \geq p \\ -\infty & \text{if} \; x < p \end{cases}.
	\end{equation}
	Let $\La(h^{-}_p; x,t)$ and $\La(h^{+}_p; x,t)$ be the directed landscape height functions associated to initial conditions $h^{-}_p$ and $h^{+}_p$, respectively. In other words,
	\begin{equation} \label{eqn:hp}
		\La(h^{-}_p; x,t) = \sup_{y \leq p} \, \{ h_0(y) + \La(0,y; x,t)\} \quad \text{and} \quad \La(h^{+}_p; x,t) = \sup_{y \geq p} \, \{h_0(y) + \La(0,y; x,t)\} .
	\end{equation}
	The competition function from $p$ is
	\begin{equation} \label{eqn:dp}
		d_p(x,t) = \La(h^{+}_p; x,t) - \La(h^{-}_p; x,t) \quad \text{for}\; (x,t) \in \Hp.
	\end{equation}
	The competition function is continuous on $\Hp$ owing to the continuity of height functions.
	
	\begin{prop}[Monotonicity] \label{prop:mono}
		For every $t > 0$, the competition function is non-decreasing in the variable $x$.
	\end{prop}
	
	\begin{proof}
		Suppose $x < y$ and consider geodesics from  $h^{\pm}_p$ to $(x,t)$ and from $h^{\pm}_p$ to $(y,t)$. The values of $\La(h^{\pm}_p; x,t)$ and $\La(h^{\pm}_p; y,t)$ are the relative lengths of these geodesics. The geodesic from $h^{+}_p$ to $(x,t)$ has to meet the one from $h^{-}_p$ to $(y,t)$
		because the $(x,t)$-geodesic starts to the right of the $(y,t)$-geodesic and ends to its left. Suppose they meet at a point $a = (z,s)$.
		See Figure \ref{fig:compete} for the setup (and essentially the proof).
		
		\begin{figure}[ht!]
			\begin{center}
				\includegraphics[scale=0.6]{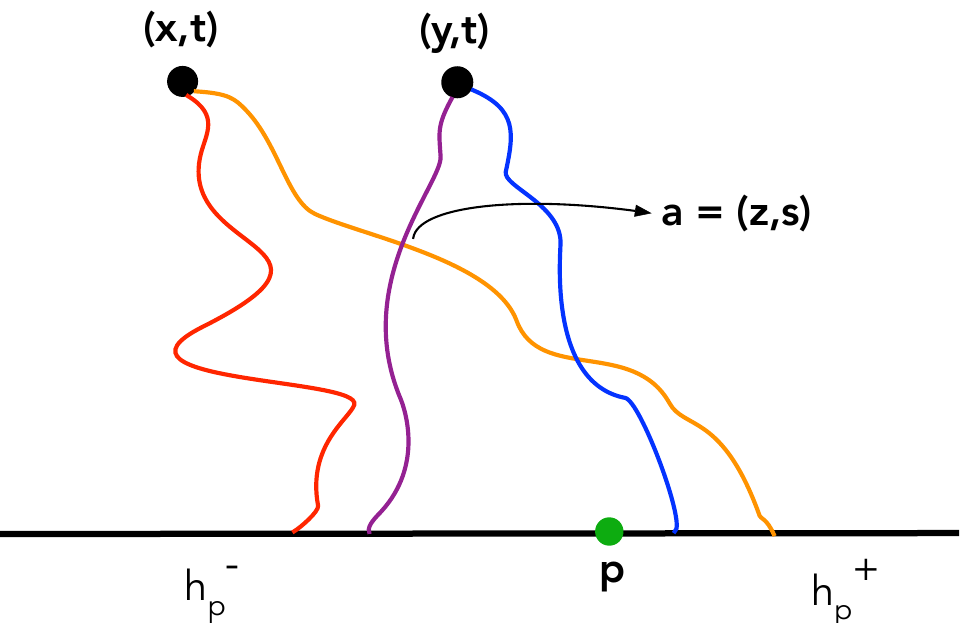}
				\caption{Reverse quadrangle inequality: red + blue $\geq$ orange + purple in length.}
				\label{fig:compete}
			\end{center}
		\end{figure}
		
		Let $g_1$ be the path from $h^{-}_p$ to $(x,t)$  obtained by going from $h^{-}_p$ to $a$ along the $ h^{-}_p\to (y,t)$ geodesic and then from $a$ to $(x,t)$ along the $ h^{+}_p\to (x,t)$ geodesic. Similarly, let $g_2$ be the path from $h^{+}_p$ to $(y,t)$ obtained by switching geodesics at $a$. 
		The sum of the lengths of $g_1$ and $g_2$ is $\La(h^{+}_p; x,t)+\La(h^{-}_p; y,t)$, as it is the same as the sum of the lengths of the $h^{+}_p\to(x,t) $ and $ h^{-}_p\to (y,t)$ geodesics.
		On the other hand, the length of $g_1$ is at most $\La(h^{-}_p; x,t)$, which is the geodesic length from $h^{-}_p$ to $(x,t)$. Similarly, $g_2$ has length at most $\La(h^{+}_p; y,t)$. So, $\La(h^{+}_p; x,t) + \La(h^{-}_p; y,t) \leq \La(h^{-}_p; x,t) + \La(h^{+}_p; y,t)$, which says that $d_p(x,t) \leq d_p(y,t)$.
	\end{proof}
	
	\subsection{Defining an interface} \label{sec:defofint}
	Given an initial condition $h_0$ and a reference point $p$ for it, the interface with reference point $p$
	is roughly the zero set of the competition function $d_p$. It represents points $(x,t) \in \Hp$ that have
	geodesics from $h_0$ starting both to the left and right of $p$. It is the boundary of this zero set that encodes information, leading to
	the following definition of interfaces.
	
	\begin{defn} \label{def:interface}
		The left and right interfaces from a reference point $p$ are as follows. For every $t > 0$,
		\begin{align*}
			I^{-}_p(t) &= \inf \, \{x \in \R: d_p(x,t) \geq 0 \}; \\
			I^{+}_p(t) &= \sup \, \{x \in \R: d_p(x,t) \leq 0 \}.
		\end{align*}
	\end{defn}
	
	It follows that $- \infty \leq I^{-}_p(t) \leq I^{+}_p(t) \leq + \infty$ since $d_p$ is continuous and non-decreasing in $x$.
	With the appropriate conventions, $d_p^{-1}(\cdot, t)\{0\} = [I^{-}_p(t), I^{+}_p(t)]$.
	
	Interfaces can be infinite, for instance $I^{-}_p(t)$ will be identically $-\infty$ if $h_0(x) = -\infty$ for every $x < p$.
	The right and left interfaces can also be separated by a region of points whose geodesics start at $p$.
	Such considerations lead to another definition.
	
	\begin{defn} \label{def:unique}
		The interface from reference point $p$ is uniquely defined if $I^{-}_p(t) = I^{+}_p(t) \in \R$ for every $t > 0$.
	\end{defn}
	
	To visualize the definitions, note the upper half place $\Hp$ is partitioned into the sets $\{d_p < 0\}$, $\{d_p=0\}$ and $\{d_p > 0\}$.
	These sets are connected and $\{d_p=0\}$ is sandwiched between $\{d_p < 0\}$ to the left and $\{d_p > 0\}$ to the right.
	For every $t > 0$, the intersections of $\R \times \{t\}$ with these sets are intervals (possibly empty or infinite).
	The left and right boundaries of $\{d_p = 0\}$ make the left and right interfaces, respectively.
	It is convenient to visualize and identify a finite interface $I_p$ with its graph $\{(I_p(t), t), \, t > 0\}$ lying in $\Hp$.
	
	\subsection{Basic properties of interfaces} \label{sec:propertiesofint}
	The following properties of interfaces hold for good samples of $\La$.
	
	\begin{prop}[Finiteness criterion] \label{prop:intfinite}
		Let $p$ be a reference point for $h_0$. The interface $I^{-}_p(t)$ is identically $-\infty$ if $h_0(x) \equiv -\infty$ for all $x < p$;
		otherwise it is finite. Similarly, $I^{+}_p(t)$ is identically $+ \infty$ if $h_0(x) \equiv - \infty $ for all $x > p$ and otherwise it is finite.
		So the left and right interfaces from an interior reference point are finite.
	\end{prop}
	
	\begin{proof}
		If $h_0(x)$ is identically $-\infty$ for all $x < p$ then $d_p(y,t) \geq 0$ for all $y \in \R$ because all geodesics start at or to the right of $p$.
		Thus, $I^{-}_p(t) \equiv -\infty$. Now suppose there is a point $a < p$ such that $h_0(a)$ is finite.
		We will prove that for every $t$, $d_p(x,t) < 0$ for sufficiently negative values of $x$.
		Then the set of points $x$ such that $d_p(x,t) \geq 0$ is bounded from below and non-empty
		(the latter because $h_0(x)$ is finite for some $x \geq p$).
		So $I^{-}_p(t)$ is finite. The same argument applies for the finiteness of $I^{+}_p(t)$.
		
		In order to see that $d_p(x,t)$ becomes negative, consider $\La(h^{\pm}_p; x,t)$ separately. For $x \leq a-1$,
		$$\La(h^{-}_p; x,t) \geq \Ai(a,0;x,t) - \frac{(x-a)^2}{t} + h_0(a).$$
		This is to be compared to $\La(y,0; x,t) + h_0(y)$ for $y \geq p$. Observe that
		$(x-a)^2 - (x-y)^2 = -(y-a)(y-x+a-x) \leq -(y-a)(y-x+1)$ since $x \leq a-1$ and $y \geq p > a$. So,
		$$ - \frac{(x-a)^2}{t} \geq - \frac{(x-y)^2}{t} + \frac{(y-a)(y-x+1)}{t}.$$
		One also has the bound
		$$|\Ai(a,0;x,t)| + |\Ai(y,0; x,t)| \leq C (t^{1/2} + |y|^{1/2} + |x|^{1/2}+1)$$
		for a random variable $C$ by \eqref{eqn:Abound}. As a result, since $h_0(y) \leq B(1+|y|)$,
		$$\La(a,0;x,t) + h_0(a) \geq \La(y,0;x,t) + h_0(y) + f(y)$$
		where
		$$f(y) = \frac{(y-a)(y-x+1)}{t} - C|y|^{1/2} -C |x|^{1/2} - B|y| -D_{t,a}$$
		for constants $C$, $B$ and $D_{t,a}$, some of which are random.
		
		The dominating term in $f$ is the quadratic $(y-a)(y-x+1)/t$, which is positive and tends to $+\infty$ as $x \to - \infty$
		faster than the other terms (note $y \geq p > a> x$.) If $x$ is sufficiently negative in terms of the parameters
		$a,t,C,B,D_{t,a}$, then $f(y)$ will be strictly positive over all $y \geq p$. For such $x$,
		$$\La(a,0;x,t) + h_0(a) > \sup_{y \geq p} \{\La(y,0;x,t) + h_0(y)\},$$
		and so $\La(h^{-}_p; x,t) > \La(h^{+}_p; x,t)$ as claimed.
	\end{proof}
	
	\begin{prop}[Continuity] \label{prop:intcont}
		When the interface $I^{+}_p(t)$ or $I^{-}_p(t)$ is finite, it is continuous.
	\end{prop}
	
	\begin{proof}
		Take for instance the right interface, and a time $t > 0$ at which we want continuity. Fix an $\epsilon > 0$.
		The point $(I^{+}_p(t)+\epsilon, t)$ satisfies $d_p(I^{+}_p(t)+\epsilon, t) > 0$. By continuity of $d_p$ there is a $\delta_1 > 0$
		such that $d_p(I^{+}_p(t)+\epsilon, s)>0$ for every $s \in (t-\delta_1, t+\delta_1)$. Consequently,
		$$I^{+}_p(s) \leq I^{+}_p(t) + \epsilon \quad \text{for}\; s \in (t-\delta_1, t+\delta_1).$$
		
		We need to show that there is a $\delta_2 > 0$ such that $I^{+}_p(t) - \eps \leq I^{+}_p(s)$ for all $s \in (t-\delta_2, t+\delta_2)$.
		Consider the point $(I^{+}_p(t)-\epsilon, t)$. If $d_p$ is negative at this point then arguing as above there is a $\delta_2 > 0$
		such that $I^{+}_p(t) -\epsilon \leq I^{+}_p(s)$ for every $s \in (t-\delta_2, t+\delta_2)$. So we are reduced to considering
		the case $d_p(I^{+}_p(t)-\epsilon, t) = 0$.
		
		Suppose that $d_p(I^{+}_p(t)-\epsilon, t) = 0$. Consider the point $a = (I^{+}_p(t)-\epsilon, t+\delta)$ for $\delta > 0$.
		If $d_p(a) \leq 0$ then $I^{+}_p(t+\delta) \geq I^{+}_p(t) - \epsilon$. We will argue below that there is a $\delta_3 > 0$
		such that $d_p(a) \leq 0$ for all $\delta \leq \delta_3$. Thus, $I^{+}_p(s) \geq I^{+}_p(t) - \epsilon$ for all $s \in [t, t+\delta_3]$.
		Similarly, for points of the form $b = (I^{+}_p(t)-\epsilon, t-\delta)$ there is a $\delta_4 > 0$ such that $d_p(b) \leq 0$ for all $\delta \leq \delta_4$,
		which implies $I^{+}_p(s) \geq I^{+}_p(t) - \eps$ for $s \in [t-\delta_4, t]$. This proves continuity at $t$.
		
		Suppose that $d_p(a) > 0$ with $a$ as above. Consider any geodesic $g$ from $h_0$ to $a$, which then has to start to the right of $p$ since $d_p(a) > 0$.
		At time $t$ this geodesic must be weakly to the right of $I^{+}_p(t)$. Otherwise, if $g(t) < I^{+}_p(t)$, then there is a geodesic to $(g(t),t)$ which starts at a point $\leq p$.
        Then following the latter geodesic until time $t$ and switching to $g$ after time $t$ yields a new geodesic to $a$, starting at a point $\leq p$.
		Since this cannot happen due to $d_p(a) > 0$, it must be that $g(t) \geq I^{+}_p(t)$. It follows that $|g(t+\delta) - g(t)| \geq \epsilon$.
		But by continuity and compactness of geodesics, this cannot happen for sufficiently small $\delta$.
		Thus, there is a $\delta_3 > 0$ such that if $s \in [t, t+\delta_3]$ then $d_p(I^{+}_p(t)-\epsilon, s) \leq 0$.
		
		A similar argument works for points of the form $b = (I^{+}_p(t)-\epsilon, t-\delta)$ to show that there is a $\delta_4 > 0$ such that $d_p(b) \leq 0$ for all $\delta \leq \delta_4$.
		Indeed, let $g$ be a geodesic from $h_0$ to $(I^{+}_p(t),t)$ that starts at a point $\leq p$.
		With $b$ as above, suppose $d_p(b) > 0$. Then all geodesics to $b$ start at points $> p$. As the geodesic $g$ starts at a point $\leq p$, geodesic ordering implies
		$g(t-\delta)$ must be to the left of $I^{+}_p(t)-\eps$. So $|g(t-\delta) - g(t)| \geq \eps$ if $d_p(b) > 0$. Geodesic continuity implies this won't happen for all small values of $\delta$,
		so indeed $d_p(b) \leq 0$ for all $\delta \leq \delta_4$.
	\end{proof}
	
	\begin{prop}[Starting location] \label{prop:intstart}
		Suppose for a reference point $p$ there are $a < p < b$ such that $h_0(a)$ and $h_0(b)$ are finite. Then $\liminf_{t \to 0} I^{-}_p(t) \geq a$
		and $\limsup_{t \to 0} I^{+}_p(t) \leq b$. As a result, if for every $\eps > 0$ there are $a$ and $b$ such that $p - \eps < a < p < b < p + \eps$
		and $h_0(a), h_0(b)$ are both finite, then $I^{\pm}_p(t) \to p$ as $t \to 0$.
	\end{prop}
	
	\begin{proof}
		We will show that the liminf of $I^{-}_p(t)$ is at least $a$ at $t \to 0$; the argument for the limsup being similar. Fix a point $x < a$
		and consider points $(x,t) \in \Hp$ as $t \to 0$. We will argue that $d_p(x,t) < 0$ for all sufficiently small $t$, almost surely in the underlying randomness.
		This implies that $\liminf_{t \to 0} I^{-}_p(t) \geq a$.
		
		Recall $\La(h^{\pm}_p; x,t)$ from \eqref{eqn:hp}. Clearly,
		$$\La(h^{-}_p; x,t) \geq \Ai(a,0;x,t) - \frac{(x-a)^2}{t} + h_0(a).$$
		
		Let $y_t \geq p$ be a point where the supremum in the definition of $\La(x,t; h^{+}_p)$ is obtained. So,
		$$\La(h^{+}_p; x,t) = \Ai(y_t,0;x,t) - \frac{(x-y_t)^2}{t} + h_0(y_t).$$
		Such a $y_t$ exists by the assumptions of $h_0$, namely that it is upper semicontinuous and bounded from above by a linear function.
		Moreover, almost surely in the underlying randomness, $y_t$ will be bounded as $t \to 0$. It follows that
		\begin{equation} \label{eqn:hhdiff1}
		  \La(h^{-}_p; x,t) - \La(h^{+}_p; x,t) \geq \Ai(a,0;x,t) - \Ai(y_t,0;x,t) + h_0(a) - h_0(y_t) +\frac{(x-y_t)^2-(x-a)^2}{t}.  
		\end{equation}
		
		There is a finite random variable $C$ such that $|\Ai(a,0;x,t)| + |\Ai(y,0;x,t)| \leq C (t^{1/2}+|x|^{1/2} + |y|^{1/2} + |a|^{1/2}+1)$ by \eqref{eqn:Abound}. This means there is a finite random variable $C'$ such that for $t \leq 1$, $|\Ai(a,0;x,t)| + |\Ai(y_t,0;x,t)| \leq C'$ because $y_t$ remains bounded, while $a$ and $x$ are now fixed. There is also a finite random variable $C''$ such that $h_0(a) - h_0(y_t) \geq C''$ since $h_0$ is bounded from above by a linear function, and $y_t$ is bounded. Consequently,
		\begin{equation} \label{eqn:hhdiff2}
		\Ai(a,0;x,t) - \Ai(y_t,0;x,t) + h_0(a) - h_0(y_t) \geq C'' - C'.
		\end{equation}
		
		Consider now the term $(x-y_t)^2 - (x-a)^2$, which equals $(y_t-a)(a+y_t-2x)$.
		Due to $y_t \geq p > a$, we find that $y_t-a \geq p-a$ and $a+y_t - 2x \geq 2(a-x)$.
		So,
		\begin{equation} \label{eqn:hhdiff3}
		   \frac{(x-y_t)^2}{t} - \frac{(x-a)^2}{t} \geq \frac{2(a-x)(p-a)}{t} = \frac{2 |a-x||p-a|}{t}. 
		\end{equation}
		Combining \eqref{eqn:hhdiff1} with \eqref{eqn:hhdiff2} and \eqref{eqn:hhdiff3} we infer that
		$$\La(h^{-}_p; x,t) - \La(h^{+}_p; x,t) \geq \frac{2 |a-x||p-a|}{t} + C'' - C'.$$
		It follows that $\La(h^{-}_p; x,t) > \La(h^{+}_p; x,t)$ for sufficiently small values of $t$, and so $d_p(x,t) < 0$ for such $t$.
	\end{proof}
	
	\subsection{Distributional invariances of interfaces} \label{sec:invarianceofint}
	Distributional invariances of the  directed landscape  \cite{DOV} give way to distributional invariances of interfaces.
	These are recorded here. In order to track dependence of interfaces on the initial condition, write $I^{\pm}_p(t; h_0(x))$
	for the left or right interface with reference point $p$ and initial condition $h_0$.
	
	The following invariances hold jointly in $p$ and $t$, and $h_0$.
		\begin{description}
		\item[Scaling.] For $\sigma> 0$,
		$$I^{\pm}_p(\sigma^3 t; h_0) \overset{law}{=} \sigma^2 I^{\pm}_{p/\sigma^2}(t; h_0(\sigma^2\cdot )/\sigma).$$
		This is due to the scaling invariance of the directed landscape height functions:
		$$\La(h_0(y); \sigma^2x,\sigma^3 t)/\sigma \overset{law}{=} \La(h_0(\sigma^2 y)/\sigma; x,t),$$
		which in turn implies
		$$d_p(\sigma^2x,\sigma^3 t; h_0(y))/\sigma \overset{law}{=} d_{p/\sigma^2}(x,t; h_0(\sigma^2 y )/\sigma).$$
		
		\item[Translation.] For $a \in \R$,
		$$ I^{\pm}_p(t; h_0(x)+a) \overset{law}{=} I^{\pm}_p(t; h_0(x)).$$
		Indeed, $\La(h_0(y)+a; x,t)$ has law $\La(h_0(y); x,t)+a$ and the competition function remains unchanged by such translations.
		Similarly the law of $\La(h_0(y+a); x,t)$ is $\La(h_0(y); x+a,t)$, which implies that
		$$I^{\pm}_p(t; h_0(x)) \overset{law}{=} I^{\pm}_{p+a}(t; h_0(x-a)) - a.$$
		
		Also, if $h_0$ satisfies $h_0(x-a) = h_0(x) - h_0(0) + h_0(-a)$ for every $x$, then $I^{\pm}_p(t; h_0(x))$
		has the same law as $I^{\pm}_{p+a}(t; h_0(x))-a$. As examples, if $h_0$ is flat or a  two-sided Brownian motion
		then $I^{\pm}_p(t; h_0)$ has the same law as $I^{\pm}_{p+a}(t; h_0) -a$ for every $a$.
		
		\item[Affine shift.] For $a \in \R$,
		$$I^{\pm}_p(t; h_0(x)+ax) \overset{law}{=} I^{\pm}_p(t; h_0(x)) - \frac{at}{2}.$$
		This comes from the invariance
		$$\La(x,t; h_0(y)+ay) \overset{law}{=} \La(x + (at)/2,t; h_0(y)),$$
		which implies that $d_p(x,t; h_0(y)+ay)$ has law $d_p(x+(at)/2, t; h_0(y))$.
	\end{description}
	
	\subsection{Example: bisector between two points} \label{sec:2wedges}
	Consider interfaces for the initial condition consisting of two narrow wedges:
	$$h_0(x) = \begin{cases} 0 & \text{if}\; x = \pm 1 \\ - \infty & \text{otherwise}\end{cases}.$$
	
	This initial condition and its competition function has been studied in \cite{BasGH}.
	The finite interfaces are $I^{\pm}_p(t)$ for $p \in (-1,1)$, and
	they are the same for every $p$.
	This interface is uniquely defined: $I^{-}_0(t) = I^{+}_0(t)$ for every $t$.
	Indeed, the zero set of the function $x \to d_0(x,t)$ is a non-empty interval by Propositions \ref{prop:mono} and \ref{prop:intfinite}.
	Any point in the zero set can not be a UGP as it must have geodesics ending at both $-1$ and $1$.
	Since the initial condition satisfies the UGC, the zero set is countable and so it must be a singleton.
	
	The unique interface $I(t) = I^{+}_0(t)$ consists of points in $\Hp$ that are $\La$-equidistant from
	$(-1,0)$ and $(1,0)$. It begins at the origin: $\lim_{t \to 0} I(t) = 0$. To see this note that
	$$d_0(x,t) = \Ai(1,0;x,t) - \Ai(-1,0; x,t) + \frac{4x}{t}.$$
	The bound \eqref{eqn:Abound} implies $|\Ai(1,0;x,t) - \Ai(-1,0; x,t)| \leq C (|x|^{1/2}+ t^{1/2}+1)$ for a random variable $C$.
	Suppose $t \leq 1$. If $x$ is such that $d_0(x,t) = 0$, then
	$$|x| \leq \frac{t}{4} C (|x|^{1/2}+2),$$
	from which it is easy to conclude that $|x| \leq C' t$ for a random variable $C'$. Consequently, $|I(t)| \leq C' t$ as $t \to 0$.
	
	The following proposition establishes the first half of Theorem \ref{thm:bisdir}.
	\begin{prop} \label{prop:bisector}
		There is a random variable $D$ having a normal distribution with mean 0 and variance $1/4$ such that, almost surely,
		$$ \lim_{t \to \infty} \frac{I(t)}{t} = D.$$
	\end{prop}
	
	\begin{proof}
		For $d \in \R$, almost surely there is an infinite geodesic $g$ from the initial condition $h_0$ with asymptotic direction $d$.
		Its existence can be proved as in Theorem \ref{thm:geoexist}. So $g(t)$ for $t \geq 0$ is a path that satisfies $g(t)/t \to d$ as $t \to \infty$,
		and for which the segment $g(s)$ for $0 \leq s \leq t$ is a geodesic to the initial condition for every $t$.
		
		An infinite geodesic cannot cross the bisector, that is, satisfy $g(t) < I(t)$ and $g(t') > I(t')$ for different times $t$ and $t'$.
		If it did cross the bisector then it would have to end simultaneously at $-1$ and $1$. So it lies either to the left or
		to the right of the bisector. If it is to the left ($g(t) \leq I(t)$) then $\liminf_{t \to \infty} I(t)/t \geq d$.
		While if it stays to the right then $\limsup_{t \to \infty} I(t)/t \leq d$.
		
		So almost surely, for every $d \in \mathbb{Q}$, either $\liminf_{t \to \infty} I(t)/t \geq d$ or $\limsup_{t \to \infty} I(t)/t \leq d$.
		This means there is a limit of $I(t)/t$ or else there would be a rational $d$ satisfying
		$\liminf_{t \to \infty} I(t)/t < d < \limsup_{t \to \infty} I(t)/t $.
		
		In order to derive the law of $D$, note $\mathbf{Pr}(D > a)$ is the limit as $t \to \infty$ of $\mathbf{Pr}(d_0(at,t) < 0)$.
		Now $d_0(at,t) = \La(1,0;at,t) - \La(-1,0;at,t)$, which by shear invariance of $\La$ has the law of $\La(1,0;0,t) - \La(-1,0;0,t) + 4a$.
		The difference $\La(1,0;0,t) - \La(-1,0;0,t)$ converges to $W(1,0) - W(-1,0)$ by Proposition \ref{prop:Busemann},
		where $W$ is the Busemann function. The law of $W(x,0)$ is Brownian with diffusivity $\sqrt{2}$ by Corollary \ref{c:Busemann-Brownian},
		so $W(1,0) - W(-1,0)$ has a Normal distribution with mean 0 and variance $4$. Therefore, $\mathbf{Pr}(D >a) = \mathbf{Pr}(2Z +4a < 0)$,
		where $Z$ is a standard Normal random variable. This shows $D$ is Normal with mean 0 and variance $1/4$.
	\end{proof}
	
	\subsection{Criterion for a uniquely defined interface} \label{sec:uniquedef}
	It is often necessary to know when an interface is uniquely defined, so we present a criterion to check this condition.
	
	Let $h_0(x)$ be an initial condition. Let $[a, b]$ be a compact interval and suppose $h_0$ is not identically $-\infty$ on $[a,b]$.
	\begin{defn} \label{def:polar}
		The polar measure of $h_0$ over $[a,b]$, denoted $\mathfrak{p}_{h_0, [a,b]}$, is the law of the maximizer
		of $h_0(x) + B(x)$ for $x \in [a, b]$, where $B$ is a Brownian motion started from $B(a) = 0$
		and with diffusivity constant $\sqrt{2}$. 
	\end{defn}
	The proof of Lemma \ref{lem:uniqueend} shows the maximizer is unique, so the measure is well defined.
	
	The polar points of $h_0$ are
	\begin{equation} \label{eqn:polarpoints}
		\mathfrak{pp}(h_0) = \{ x \in \R: \mathfrak{p}_{h_0, [a,b ]}(\{x\}) > 0 \;\; \text{for some}\;\; a < x < b\}.
	\end{equation}
	If $[a, b] \subset [c,d]$ then $\mathfrak{p}_{h_0, [a,b]}$ is absolutely continuous with respect to $\mathfrak{p}_{h_0, [c,d]}$.
	So the polar points of $h_0$ are the atoms of its polar measures over increasing compact intervals.
	There are at most a countable number of polar points of any $h_0$. Recall that $p$ is an interior point for $h_0$ if there are $a < p < b$ such that $h_0(a), h_0(b)$ are both finite.
	
	\begin{prop} \label{prop:intunique}
		Let $p$ be a non-polar interior reference point of $h_0$. Then the interface of $h_0$ from reference point $p$ is uniquely defined almost surely.
	\end{prop}
	
	\begin{proof}
		Since $p$ is an interior reference point, $I^{\pm}_p(t)$ are finite valued curves by Proposition \ref{prop:intfinite}.
		
		Recall $I^{\pm}_p$ are the boundaries of the set $\{d_p = 0\}$ and, by continuity of interfaces, they are equal if and only if this
		set has an empty interior. A point in the interior of $\{d_p=0\}$ has the property that all its geodesics from $h_0$ start at $p$; this
		is shown in Lemma \ref{lem:Gp} below. So to show $\{d_p=0\}$ has empty interior almost surely, it is enough to prove that any
		given $(x,t) \in \Hp$ has a geodesic from $h_0$ starting at $p$ with zero probability. For then the interior of $\{d_p=0\}$
		contains no rational points almost surely, and so it is empty.
		
		Now if $(x,t) \in \Hp$ has a geodesic from $h_0$ starting at $p$ then $y=p$ maximizes $y \mapsto \La(0,y; x,t) + h_0(y)$.
		Let $[c, d]$ be any compact interval of positive length containing $p$. Then $p$ is a maximizer of $y \mapsto \La(0,y;x,t) - \La(0,c;x,t) + h_0(y)$
		for $y \in [c, d]$. Let $B(y)$ be a Brownian motion started from $B(c) = 0$. The former process is absolutely continuous with
		respect to $B(y) + h_0(y)$ over $[c,d]$. As $p$ is not a polar point of $h_0$, it maximizes $B+h_0$ over $[c,d]$ with zero probability.
		The same therefore holds for the former process, as required.
	\end{proof}
	
	A sufficient condition for $p$ being a non polar point of $h_0$ is if the modulus of continuity of $h_0$ at $p$ has H\"{o}lder exponent more than $1/2$. Indeed, assuming without loss of generality that $p=0$ and $h_0(0)=0$, note that  by the law of iterated logarithm there exists arbitrary small $t>0$ so that $B(t)>\sqrt{t\log \log (1/t)}$, and for those $t$ one has $B(t)+h(t)>0$. 
	
    \begin{exm}\label{exm:BMpolar}
    Another example: given $p\in \R$, when $h_0$ is Brownian motion then $p$ is non-polar almost surely. This follows from the fact that sum of two independent Brownian motions is another Brownian motion. In a similar vein, any given $p$ is almost surely a non-polar point of the KPZ fixed point $\La(h_0; x,t)$ at any fixed positive time. This is because $\La(h_0; x,t)$ is absolutely continuous with respect to Brownian motion on any compact interval \cite[Theorem 1.2]{SV}.
    \end{exm}

	\subsection{Distribution of interfaces} \label{sec:interfacelaw}
	There is a geometric description of the law of the interface $I^{\pm}_p(t)$ associated to an initial condition $h_0$
	in terms of the concave majorant of a function. This also helps to derive the law of the asymptotic direction of interfaces.
	
	The {\bf concave majorant} of a function $f: \R \to \R \cup \{-\infty\}$ is the smallest concave function $c$ such that $c(x) \geq f(x)$
	for every $x$. We denote it by $\cm(f)$.
	
	Let $\partial_l c$ and $\partial_r c$ be the left and right derivatives, respectively, of a concave function $c$.
	Recall these are defined as $\partial_lc(p) = \lim_{\eps \downarrow 0} (c(p)-c(p-\eps))/\eps$
	and $\partial_rc(p) = \lim_{\eps \downarrow 0} (c(p+\eps)-c(p))/\eps$.
	The concave majorant has the property that all maximizers of $f$ are strictly to the right of $p$ if and only if $\partial_r \cm(f)(p) > 0$.
	Similarly, all maximizers of $f$ are strictly to the left of $p$ if and only if $\partial_l \cm(f)(p) < 0$. Indeed, $f$ and $\cm(f)$
	are equal at the maximizers of $f$ and the derivative of $\cm(f)$ is non-increasing. See \cite{Gr} for properties of concave majorants.
	
	Denote by $$\mathfrak{A}_t(x) = \La(0,0;x,t)$$ the (parabolic) Airy process at time $t$ as a process in $x$.
	
	\begin{prop}\label{p:interface-distr}
		Given an initial condition $h_0$, an interior reference point $p$ and a time $t > 0$,
		$$
		I^-_p(t)\;\;\stackrel{law}{=}\;\;-\frac{t}{2}\partial_{l}\cm(h_0+\mathfrak A_t)(p) \quad \text{and} \quad I^+_p(t)\;\;\stackrel{law}{=}\;\;-\frac{t}{2}\partial_{r}\cm(h_0+\mathfrak A_t)(p).
		$$
	\end{prop}
	
	\begin{proof}
		We prove the first claim as the second is symmetric.
		Note $I^-_p(t) > y$ if and only if $\La(\cdot,0;y,t)+h_0(\cdot)$
		takes its maximum strictly to the left of $p$. Since $h_0$ is an initial condition, hence
		does not grow faster than linearly, and $\La(\cdot,0;y,t)$ decays quadratically, the
		concave majorant of $\La(\cdot,0;y,t)+h_0(\cdot)$ exists, and the above is equivalent to
		$$
		\partial_l \cm \big (\La(\cdot, 0; y,t)+h_0(\cdot)\big )(p) < 0.
		$$
		By shear invariance, as functions of $x$,
		$$
		\La(x,0;y,t)\stackrel{law}{=} \mathcal L(x,0;0,t)+ \frac{y(2x-y)}{t}.
		$$
		$\La(x,0;0,t)$ has the same law as $\La(0,0;x,t)$ as a function of $x$ by the time reversal symmetry of $\La$ \cite{DV}.
		Therefore,
		$$
		\partial_l \cm \big (\mathcal L(\cdot,0;y,t)+h_0(\cdot) \big )(p)\stackrel{law}{=}  \partial_l \cm(\mathfrak A_t+h_0)(p)+\frac{2y}{t}.
		$$
		As a result, $\mathbf{Pr}(I^{-}_p(t) > y) = \mathbf{Pr}(\partial_l \cm(\mathfrak A_t+h_0)(p) < -2y/t)$ as claimed.
	\end{proof}
	
	\begin{cor}
		For any $t > 0$,
		$$ \mathbf{Pr}(I^-_p(t)=I^+_p(t))=1\;\; \Leftrightarrow \;\; \mathbf{Pr}(\cm(h_0+\mathfrak A_t)\mbox{ is differentiable at }p)=1.$$
	\end{cor}
	
	\begin{proof}
		Since $I^{-}_p(t) \leq I^{+}_p(t)$ and $\partial_r \cm(h_0 + \mathfrak A_t)(p) \leq \partial_l \cm(h_0 + \mathfrak A_t)(p)$,
		each of these pairs of random variables are equal with probability one if and only if they have the same law.
		Therefore, by Proposition \ref{p:interface-distr}, $I^-_p(t)$ and $I^+_p(t)$ have the same law if and only if
		$\cm(h_0 + \mathfrak A_t)$ is differentiable at $p$.
	\end{proof}
	
	\subsection{Asymptotic direction of interfaces} \label{sec:directionofint}
	Proposition \ref{prop:bisector} showed that the bisector between two points has an asymptotically Normal direction.
	What about other initial conditions? Intuition from Euclidean geometry suggests asymptotically flat initial conditions
	should have asymptotically vertical interfaces -- a fact confirmed by Theorem \ref{thm:direction}. More generally,
	in $\S$\ref{sec:subwedge} we compute the direction of interfaces for initial conditions that do not grow too fast.
	
	\begin{thm} \label{thm:direction}
		Suppose an initial condition $h_0(x)$ is asymptotically flat in that $|h_0(x)|/ |x| \to 0$ as $|x| \to \infty$.
		Then for every $p$, almost surely, $\lim_{t \to \infty} I^{\pm}_{p}(t)/t = 0$.
	\end{thm}
	
	Combining the theorem with the affine shift invariance of interfaces leads to
	\begin{cor}
		Suppose $h_0(x) = a x + h(x)$ where $h(x)$ is asymptotically flat and $a \in \R$.
		Then for every $p$, almost surely, $\lim_{t \to \infty} I^{\pm}_p(t)/ t = -a/2$.
	\end{cor}
	
	\begin{proof}
		We may assume that $p = 0$ by translation invariance of interfaces. Since $I^{-}_0(t)$ and $I^{+}_0(t)$
		are the first and last zeroes of the monotone function $d(\cdot, t)$, it is enough to show that for every $\epsilon > 0$
		it holds that $d(\eps t, t) > 0$ and $d(-\eps t, t) < 0$ for all large values of $t$. We will prove that
		$d(\eps t, t) > 0$, and the argument for the other inequality is entirely similar.
		
		Suppose $0 < \eps < 1/2$ and let $\delta \in (0, \eps^2/5)$. The bound \eqref{eqn:DOVAbound} for $\Ai(y,0;x,t)$
		implies that there is a random $t_0$ such that if $t \geq t_0$ then
		$$ |\Ai(y,0;x,t)| \leq \delta (t + |x| + |y|).$$
		Since $|h_0(x)|/|x|$ tends to zero, there is a $t_1$ such that for $t \geq t_1$,
		$$|h_0(x)| \leq \delta \left (|x| + \frac{t}{2} \right ).$$
		Suppose $t \geq \max \{t_0,t_1\}$ so that both the bounds above hold.
		We will show that $d(\eps t,t) > 0$, which is sufficient.
		
		Consider $\La(h_0^{-}; \eps t,t) = \sup_{y \leq 0} \, \{ h_0(y) + \La(y,0;\eps t, t)\}$. The bounds above imply
		$$ \La(h_0^{-}; \eps t,t) \leq \sup_{y \leq 0} \, \left \{ -\frac{(y-\eps t)^2}{t}-2\delta y  +2\delta t \right \}$$
		since $\eps < 1/2$. The optimization may be rearranged as
		$$ t(2\delta-\eps^2) + \sup_{y \leq 0} \, \left \{-\frac{y^2}{t}+2y(\eps-\delta) \right \} = t(2\delta-\eps^2). $$
		Since $\delta < \eps$, the maximum occurs at $y = 0$.
		
		A similar lower bound holds for $\La(h_0^{+}; \eps t,t;) = \sup_{y \geq 0} \, \{h_0(y) + \La(y,0;\eps t, t)\}$.
		\begin{align*}
			\La(h_0^{+}; \eps t,t) &\geq \sup_{y\geq 0} \left \{ -\frac{(y-\eps t)^2}{t}-2\delta y -2\delta t \right \}\\
			& = -t(2\delta+\eps^2) + \sup_{y \geq 0} \left \{-\frac{y^2}{t}+2y(\eps-\delta) \right \} \\
			& = -t(2\delta +\eps^2)+ t(\eps-\delta)^2 \\
			& \geq - 3 \delta t.
		\end{align*}
		In the final line we used that $\eps < 1/2$.
		
		Combining the upper bound on $\La(\eps t,t; h_0^{-})$ with the lower bound on $\La(\eps t,t; h_0^{+})$ shows 
		that $d(\eps t, t) \geq (\eps^2 - 5 \delta) t > 0$.
	\end{proof}
	
	\subsubsection{Directions for sub-wedge initial conditions} \label{sec:subwedge}
	For an initial condition $h_0$, its left and right upper linear growth rates are
	$$ a_{-} = \liminf_{x \to -\infty} \frac{h_0(x)}{x} \qquad a_{+} = \limsup_{x \to \infty} \frac{h_0(x)}{x}.$$
	Note $a_{\pm}$ take values in $\R \cup \{\pm \infty \}$. An initial condition is called {\bf sub-wedge}
	if $a_{+} < a_{-}$. For instance, $h_0(x) = - a|x| + b$ for $a > 0$ is sub-wedge, as is any $h_0$ that is $-\infty$ outside a
	compact interval or outside a half-line of the from $[p,\infty)$ or $(-\infty, p]$. We determine the asymptotic
	direction of interfaces to sub-wedge initial conditions in Theorem \ref{thm:randdirection} and Corollary \ref{cor:directionlaw}.
	
	If $B(x)$ is two-sided Brownian motion (of diffusivity $\sqrt{2}$) and $h_0$ is a sub-wedge initial condition, then
	\begin{equation}\label{e:supBy}
		\sup_{y \in \R}\, \{h_0(y)+B(y)+by\}
	\end{equation}
	is almost surely finite and achieved for $-a_-< b <-a_+$. For $b<-a_-$ or $b>-a_+$, it is almost surely infinite.
	Recall the Busemann function $W_d$ in direction $d$, for which $W_d(x,0)$ has law $B(x) + 2d x$.
	Consequently, if $d \in (- a_{-}/2, -a_{+}/2)$ then, almost surely,
	\begin{equation} \label{eqn:Wdh} 
		W_d(h_0) = \sup_{y \in \R} \, \{h_0(y) + W_d(y,0) \}
	\end{equation}
	is finite and achieved.
	
	A geodesic $\gamma$ from an initial condition $h_0$ in direction $d$ is a geodesic $\gamma: [0,\infty) \to \R$
	of $\La$ in direction $d$ such that $\gamma$ restricted to $[0,t]$ is a geodesic from $h_0$ for every $t$. Recall the notation
	$\La(h_0;p) = \sup_y \, \{h_0(y) + \La(y,0;p)\}$ from \eqref{eqn:Ldist}.
	
	\begin{lem}\label{p:h0-geod}
		Let $d \in \R$ and $h_0$ be an initial condition. Suppose $W_d(h_0)$ from \eqref{eqn:Wdh} is finite and achieved.
		Then there is a geodesic $\gamma$ from $h_0$ in direction $d$.
		Such a geodesic is obtained by choosing a maximizer $z$ of $W_d(h_0)$ and following the rightmost geodesic of $\La$ from $(z,0)$
		in direction $d$. With $z=\gamma(0)$, for any point $p$ on $\gamma$, 
		\begin{eqnarray}\label{e:h0p}
			\La(h_0;p)&=&\La(z,0;p)+h_0(z)
			\\
			\label{e:WzWp}
			&=&W_d(z,0) + h_0(z) -W_d(p)
			\\
			\label{e:WpWy}
			&=&
			W_d(h_0) -W_d(p).
		\end{eqnarray}
	\end{lem}

	\begin{proof}
		By definition of the Busemann function $W_d$,
		for any two points $p,q$ in increasing time order on a geodesic $\gamma$ of $\La$ in direction $d$,
		\begin{equation}\label{e:geod-pq}
			\mathcal L(p,q)=W_d(p)-W_d(q).
		\end{equation}
	
		Since $W_d(h_0)$ if finite and achieved, there is a $z$ such that for every $y \in \R$,
		\begin{equation}\label{e:ylex}
			W_d(y,0)+h_0(y)\leq W_d(z,0)+h_0(z) = W_d(h_0).
		\end{equation}
		Let $\gamma$ be the rightmost geodesic of $\La$ from $(z,0)$ in direction $d$. We will show $\gamma$ is a geodesic from $h_0$.
		
		Let $p$ be any point on $\gamma$. Let $y \in \R$ and let $\gamma'$ be the rightmost geodesic from $(y,0)$ in direction $d$.
		Let $q$ be the point where $\gamma$ and $\gamma'$ coalesce.
		If $q$ appears on $\gamma$ at time after $p$, the reverse triangle inequality implies
		$$\La(y,0;p) \leq \La(y,0;q) - \La(p;q).$$
		If $q$ appears at time before $p$, the sharpness of the reverse triangle equality along the geodesic $\gamma'$ implies
		$$\La(y,0;p) = \La(y,0;q) + \La(q;p).$$
		The two displays above together with \eqref{e:geod-pq} implies
		$$\La (y,0;p) \leq W_d(y,0)-W_d(p).$$
		Now \eqref{e:ylex} gives
		$$\mathcal L(y,0;p)+h_0(y) \le W_d(z,0)+h_0(z)-W_d(p)=\mathcal L(z,0;p)+h_0(z)$$
		by \eqref{e:geod-pq}, showing \eqref{e:WzWp}. Thus the left hand side is maximized at $z$, showing \eqref{e:h0p}.
		Since this holds for every $p$ on $\gamma$,  $\gamma$ is a geodesic from $h_0$.
		The definition of $z$ implies $\eqref{e:WpWy}$.
	\end{proof}
	
	So far we have considered the Busemann function and geodesic tree in a fixed direction.
	Our construction allows us to extend Busemann functions simultaneously in every rational direction.
	For the next result, we need to consider them as such. Lemma \ref{p:h0-geod} implies that if $h_0$ is sub-wedge then, almost surely,
	there are geodesics from $h_0$ in every rational direction $d$ for $d \in (-a_{-}/2, -a_{+}/2)$.
	
	Recall the splitting of an initial condition $h_0$ at a reference point $p$, which defines split functions $h^{\pm}_p$ according to \eqref{eqn:split}.
	The function $d \mapsto W_d(h^{+}_p) - W_d(h^{-}_p)$ is non-decreasing (restricted to rational $d$ as we do).
	This may be proved in the same way as Proposition \ref{prop:mono}, using the reverse quadrangle inequality, by
	interpreting $W_d(h_0)$ as a geodesic length from $h_0$ in direction $d$. In fact, the difference is the limit as $t \to \infty$
	of the function $d_p(dt,t)$ where $d_p(x,t)$ is the competition function from \eqref{eqn:dp}.
	This motivates the following theorem.
	
	\begin{thm}[Asymptotic direction]\label{thm:randdirection}
		Let $h_0$ be a sub-wedge initial condition with upper linear growth rates $a_{\pm}$. Let $p$ be an interior reference point of $h_0$.
		Define
		\begin{align*}
		D^{-}_p &= \inf \, \{d\in \mathbb Q :W_d(h_p^+) \geq W_d(h_p^-)\}; \\
		D^{+}_p &= \sup \, \{d\in \mathbb Q :W_d(h_p^+) \leq W_d(h_p^-)\}.
		\end{align*}
		These quantities lie in the interval  $[- a_{-}/2,- a_{+}/2]$.
		The left and right interfaces $I^{\pm}_p(t)$ have almost sure directions: $\lim_{t \to \infty} I^{-}_p(t)/t = D_p^{-}$
		and $\lim_{t \to \infty} I^{+}_p(t)/t = D^{+}_p$. 
	\end{thm}
	
	\begin{proof}
		We will prove the proposition for the left interface since the proof for the right interface is symmetric.
		
		If $d>-a_{+}/2$ then $W_d(h_p^+)=\infty$ and $W_d(h_p^-)<\infty$; so $D^-_p\leq -a_{+}/2$. Similarly, $D^-_p\geq -a_{-}/2$.
		
		Let $d < D^{-}_p$ be an arbitrary rational direction. Then $W_d(h_p^-) > W_d(h_p^+)$ and $W_d(h_p^+)$ is finite and achieved (since $d < -a_{+}/2$).
		Let $\gamma$ be the geodesic from $h_p^{+}$ in direction $d$ according to Lemma \ref{p:h0-geod} .
		For all times $t$,
		\begin{equation}\label{e:fromgeod}
			\La(h_p^+;(\gamma(t),t))=W(h_p^+)-W_d((\gamma(t),t)).
		\end{equation}
		Since $W_d(h_p^-)>W_d(h_p^+)$, there is a $y$ such that $W_d(y,0)+h_p^-(y) > W_d(h_p^+)$.
		Let $\gamma'$ be the rightmost geodesic of $\La$ from $(y,0)$ in the direction $d$. Then
		\begin{align*}
			\La(h_p^-;(\gamma'(t),t)) & \geq h_p^-(y) +\La(y,0;(\gamma'(t),t)) \\
			& = h_p^-(y) + W_d(y,0) -W_d((\gamma'(t),t)) \\
			& > W_d(h_p^+) -W_d((\gamma'(t),t)).
		\end{align*}
		
		Comparing this with \eqref{e:fromgeod}, it follows that for all times $t$ after $\gamma$ and $\gamma'$ meet,
		$$\La (h_p^+;(\gamma(t),t)) < \La (h_p^-; (\gamma(t),t)), $$
		which implies $ \gamma(t) \leq I^-_p(t)$. Since $\gamma$ has a direction $d$, this means $d \leq \liminf_t I^-_p(t)/t$.
		As this holds for all rational $d < D^{-}_p$, it means $D^-_p \leq \liminf_t I^-_p(t)/t$.
		
		A symmetric argument for the other side gives that for all rational $d>D_p^-$,
		for the geodesic $\gamma$ from  $h_p^-$ in direction $d$ given by Lemma \ref{p:h0-geod},
		$$
		\La(h_p^+;(\gamma(t),t))\geq  \La(h_p^-;(\gamma(t),t)) \qquad \mbox{ for all large enough } t,
		$$
		which implies $I^-_p(t)\leq  \gamma(t)$. It follows that $\limsup_t I^{-}_p(t)/t \leq D^{-}_p$.
	\end{proof}
	
	\begin{cor} \label{cor:directionlaw}
		Let $h_0$ be a sub-wedge initial condition with an interior reference point p. Then,
		\begin{align*}
			D_p^{-} &\stackrel{law}{=} -\frac{1}{2} \partial_l \cm(h_0 + B)(p)\\
			D_p^{+} & \stackrel{law}{=} -\frac{1}{2} \partial_r \cm(h_0+B)(p)
		\end{align*}
		where $B$ is a two-sided Brownian motion of diffusivity $\sqrt{2}$.
	\end{cor}
	
	\begin{proof}
		For rational values of $d$, it follows that if $ d < D_p^{-}$ then $W_d(h_p^{+}) < W_d(h_p^{-})$,
		which means all geodesics from $h_0$ start strictly from the left of $p$. So $h_0(y) + W_d(y,0)$
		has all maximizers to the left of $p$, which is equivalent to $\partial_l\cm (h_0 + W_d)(p) < 0$.
		Since $W_d$ equals $B(x) + 2dx$ in law, $\mathbf{Pr}(\partial_l \cm (h_0+W_d)(p) < 0)$
		equals $\mathbf{Pr}(\partial_l \cm(h_0 + B)(p) < -2d)$. So $\mathbf{Pr}(D_p^{-} > d) \leq \mathbf{Pr}(\partial_l \cm(h_0+B)(p) < -2d)$.
		Taking complements gives
		$$\mathbf{Pr}\left (-\frac{1}{2} \partial_l \cm(h_0+B)(p) \leq d \right ) \leq \mathbf{Pr}(D^{-}_p \leq d).$$
		
		On the other hand, $D^{-}_p < d$ implies $W_d(h_p^{+}) \geq W_d(h_p^{-})$. The latter is equivalent to
		$\partial_l \cm(h_0+W_d)(p) \geq 0$. The probability of this event is $\mathbf{Pr}(\partial_l \cm(h_0+B)(p) \geq -2d)$.
		Therefore, $\mathbf{Pr}(D^-_p < d) \leq \mathbf{Pr}(\partial_l \cm(h_0+B)(p) \geq -2d)$. Combined with the bound above, this gives
		$$ \mathbf{Pr}(D^-_p < d)  \leq \mathbf{Pr}\left (-\frac{1}{2} \partial_l \cm(h_0+B)(p) \leq d \right ) \leq \mathbf{Pr}(D^{-}_p \leq d).$$
		As this holds of all rational values of $d$, the two distribution functions are equal. The law of $D^+_p$ is derived in the same way.
	\end{proof}

	The asymptotic direction of interfaces can often be computed via Corollary \ref{cor:directionlaw}.
	This was the case for the two narrow wedges initial condition in $\S$\ref{sec:2wedges}.
	Three other examples are discussed below, the second of which proves the second half of Theorem \ref{thm:bisdir}.
	
	\paragraph{\textbf{One-sided Brownian initial condition}}
	Suppose $h_0(x)$ is a one-sided Brownian motion with diffusivity $\sqrt{2}\sigma$ for $x \geq 0$
	and equals $-\infty$ for all $x < 0$. Every reference point $p > 0$ has a uniquely defined interface with a
	random asymptotic direction $D_p$. The law of $D_p$ is $\sqrt{(\sigma^2+1)/2}\cm(B_+)'(p)$,
	where $B_+$ is a one-sided standard Brownian motion.
	
	It is shown in \cite[Corollary 2.2]{Gr} that the random variable $K = \cm(B_+)'(1)$ exists
	with probability one and has the probability density
	$$ \mathbf{Pr}(K \in dx) = 4 \mathbf{Pr}( N \in dx)- 4 x \mathbf{Pr}(N>x)dx, \qquad x \geq 0,$$
	where $N$ is a standard normal variable. Also, $K$ has the law of the product of a standard $\chi_5$
	random variable and an independent  beta$(1,3)$ random variable $\beta_1$; see the example below.
	By Brownian scaling, for $p > 0$,
	$$
	D_p \;\;\stackrel{law}{=}\;\; -\sqrt{\frac{\sigma^2+1}{2p}}\chi_5\, \beta_1.
	$$
	
	\paragraph{\textbf{Flat initial condition on an interval}}
	When $h_0$ is $0$ on $[0,1]$ and $-\infty$ elsewhere, and $p \in (0,1)$, the quantity in question is the slope at $p$
	of the concave majorant of Brownian motion on $[0,1]$. To compute the distribution, we use an argument suggested by Jim Pitman based on
	the results in \cite{OP}. First write Brownian motion on $[0,1]$ as $Nt+R(t)$, where $R(t)$ is a Brownian bridge and $N$ is an  independent
	standard normal.
	
	Next, Doob's transformation $f \mapsto g$, $g(u)=(1-u)f(u/(1-u))$, maps the line with slope $a$ and intercept $b$
	to the line with slope $a-b$ and intercept $b$. Also, it maps the standard Brownian motion to a standard Brownian bridge.
	As a result, 
	$$
	-\sqrt{2}D_p \;\; \stackrel{law}{=}\;\;N+K_t-I_t \stackrel{law}{=}N+K_1/\sqrt{t}-I_1\sqrt{t}, \qquad t=p/(1-p),
	$$
	where $K_t, I_t$ are the slope and intercept of the tangent line at time $t$ in the concave majorant of standard Brownian motion.
	The law of $(K_1,I_1)$ is is given in Ouaki and Pitman \cite{OP} (see Proposition 1.2, density $f_3$).
	If the vector $(\beta_1,\beta_2,\beta_3)$ are chosen from Dirichlet$(1,1,2)$ distribution, then for an independent $\chi_5$-random variable, $(K_1,I_1)\stackrel{law}=(\beta_1,\beta_2)\chi_5$. So
	$$
	D_p \;\; \stackrel{law}{=}\;\;\frac{p\beta_2-(1-p)\beta_1}{\sqrt{2(1-p)p}}\chi_5 -\frac{N}{\sqrt{2}}.
	$$

        \paragraph{\textbf{Brownian minus a wedge initial condition}}
	Take $h_0(x) = \sigma B(x) - \mu |x|$ where $\sigma \geq 0, \mu > 0$ and $B(x)$ is a two-sided standard Brownian motion.
	Then the interface from $p=0$ has direction
	$$ D_0 \stackrel{law}{=} \mathrm{Uniform}[-\mu/2, \mu/2].$$
	This is based on the fact that if $V_{\sigma, \mu}(x) = \sigma B(x) - \mu |x|$, then for every $\sigma, \mu > 0$,
	\begin{equation} \label{eqn:wedgecm}
		\cm(V_{\sigma,\mu})'(0) \stackrel{law}{=} \mathrm{Uniform}[-\mu, \mu].
	\end{equation}
	
	To prove this, note that the law of the derivative does not depend on $\sigma$. This is because $\cm(f)'(0)$ is unchanged
	under the transformation $f \mapsto \lambda f(x/\lambda)$ for any $\lambda > 0$, and this takes the law of $V_{\sigma, \mu}$
	to the law of $V_{\sqrt{\lambda}\sigma, \mu}$. Thus we may assume $\sigma = 1$.
	
	For any function $V$ with a concave majorant, $\cm(V)'(0) = \lim_{a \to \infty} \cm(V \mid_{[-a, \infty)})'(0)$.
	Consider the transformation $V \mapsto \tilde{V}$ where
	$$\tilde{V}(x) = \frac{V((x-1)a) - V(-a)}{\sqrt a} + \mu\sqrt{a} x  \qquad \text{for}\;\;  x\geq 0.$$
	This takes the law of $V_{1,\mu}$ restricted to $[a,\infty)$ to the law of $B_{2\mu\sqrt{a}}$,
	where $B_{\nu}(x) = B_{0}(x) + \nu (\min \{x,1\})$ for $x \geq 0$, and where $B_0(x)$ is a one-sided standard Brownian motion.
	As $\cm(V)'(0) = \frac{1}{\sqrt{a}}\cm(\tilde V)'(1)-\mu$, it is enough to show that as $\nu \to \infty$,
	\begin{equation}\label{eqn:Bmu}
		\frac{1}{\nu}\cm(B_\nu)'(1)\stackrel{law}{\to}\mathrm{Uniform}[0,1].
	\end{equation}

	Let $K_\nu=\cm(B_\nu)(1)$. Then $K_\nu-B_\nu(1)\geq 0$, and it is  non-increasing in $\nu$. This is because $K_\nu-B_\nu(1)=\cm(B_\nu-B_0(1)-\nu)(1)$ and $B_\nu(\cdot)-B_0(1)-\nu$ is non-increasing in $\nu$. 
	Since $B_\nu(1)/\nu\to 1$ almost surely, it follows that $K_\nu/\nu\to 1$ as well.
	
	When $\nu=0$, the results of Ouaki and Pitman \cite{OP} (Proposition 1.2, density $f_3$) imply that conditionally on 
	$(K_\nu, B_\nu(1))$ the law of $\cm(B_\nu)'(1)$ is uniform on $[0,K_\nu]$. Since $B_\nu$ has the law of $B_0$
	biased by $e^{\nu B_0(1) - \nu^2/2}$, the same property holds for all  $\nu$, and \eqref{eqn:Bmu} follows from $K_\nu / \nu \to 1$.
 
	\subsection{Interfaces are locally absolutely continuous with respect to geodesics} \label{sec:intabscont}
	This section proves Theorem \ref{thm:abscont}.
	A theorem about exponential last passage percolation states that the infinite geodesic in point-to-point
	last passage percolation has the same law as the competition interface in stationary last
	passage percolation \cite{BCS,Sep2}. This suggests a connection between geodesics and interfaces in the limit, which we
	formalize in terms of absolute continuity.
	
	\begin{defn}
	Let $I = (I(t), t \geq s)$ be a continuous process. Let $g$ be the (almost surely unique) infinite geodesic of $\La$ from $(0,s)$ in direction zero. The process $I$ is \bemph{geodesic-like on compacts} if for every compact interval $[a,b] \subset (s, \infty)$, the process $I \mid_{[a,b]}$ is absolutely continuous (as a continuous process) with respect to $g\mid_{[a,b]}$.
	\end{defn}
	
	We re-state Theorem \ref{thm:abscont} using the terminology established.
	\begin{thm} \label{thm:intgeo}
		Let $h_0$ be an initial condition. Suppose $p$ is an interior reference point for $h_0$ which is non-polar,
		so there is a uniquely defined interface $I_p(t)$ for $t \geq 0$ from reference point $p$. 
		The process $I_p(t)-p$ is geodesic-like on compacts.
	\end{thm}
	
	The theorem will be proved in several steps through a sequence of lemmas. The first lemma below is proved independently in \cite[Proposition 2.8]{GZ}
	
	\begin{lem} \label{lem:intgeo1}
		Let $I_0(t)$ be the uniquely defined interface from reference point $0$ associated to the stationary Brownian initial condition.
		The process $I_0$ has the same law as the infinite geodesic of $\La$ from $(0,0)$ with direction zero.
	\end{lem}
	
	\begin{proof}
		In order to prove this result we shall consider exponential last passage percolation,
		whose setting and terminology are described in the Introduction and in $\S$\ref{sec:2ndandcomp} below.
		
		Consider point-to-point last passage percolation on $\Ze^2$ with i.i.d.~exponential weights of mean 1.
		Let $L(a,b)$ be the last passage time from point $a \in \Ze^2$ to $b \in \Ze^2$ as defined by \eqref{eqn:LPPdef}.
		Thus $L(a,b)$ is the maximal sum of weights over all directed paths from $a$ to $b$.
		A geodesic is a directed path that achieves the last passage time between its endpoints.
		
		In this setting there is an almost surely unique infinite geodesic $\Gamma(n)$, $n = 0,1,2,3,\ldots$, from $(0,0)$ in direction zero \cite[Theorem 2.2]{Sep2}.
		This means that $n \mapsto (n, \Gamma(n))$ is an infinite directed path that is a geodesic when restricted to any finite interval of times,
		$\Gamma(0) =0$, and $\Gamma(n)/n$ tends to $0$ as $n$ goes to infinity.
		
		The second process to consider in the competition interface in stationary last passage percolation. Stationary last passage percolation looks at last passage times to an initial condition $h_0(x)$ that is a two-sided simple symmetric random walk on $\Z$. See $\S$\ref{sec:2ndandcomp} for the setup in detail. (There is actually a one-parameter family of stationary initial conditions indexed by the drift of the random walk, but we consider the symmetric case.) The stationary competition interface is a process $\Gamma'(n)$, for $n=0,1,2,3,\ldots$, which is the process $X_n$ from \eqref{eqn:2ndnstep} for the stationary initial condition. We note that $\Gamma'(n)$ records the position of a 2nd class particle starting from the origin in stationary tasep after $n$ steps.
		
		A surprising fact is that the geodesic $\Gamma$ and the competition interface $\Gamma'$ have the same law as random processes \cite[Proposition 5.2]{Sep2}.
		
    We may now consider the KPZ scaling limit of these two processes. Under KPZ scaling, $\Gamma$ and $\Gamma'$ are re-scaled to $t \mapsto m^{-2/3} \Gamma(\lceil m t \rceil)$ and $t \mapsto m^{-2/3} \Gamma'(\lceil m t \rceil)$ for a large integer $m$ and $t \geq 0$.
  
    As $m \to \infty$, the KPZ re-scaled competition interface converges in law to the uniquely defined interface $I_0(t)$ associated to the two-sided Brownian initial condition with diffusivity constant $\sqrt{2}$. The convergence here is under the topology of uniform convergence on compact subsets of $t \in (0,\infty)$. This is proved in Proposition \ref{prop:statintlimit}. Note further that $I_0(t)$ has asymptotic direction zero because Brownian motion is almost surely an asymptotically flat initial condition, so Theorem \ref{thm:direction} applies.
		
    Due to the equality of law between $\Gamma$ and $\Gamma'$, we conclude that the KPZ re-scaled geodesic has a tight law and admits distributional limits. Any limit point $g(t)$ of the re-scaled geodesic must be an infinite geodesic of $\La$ because limits of geodesics in last passage percolation are geodesics of $\La$ \cite{DV}. The process $g(t)$ also has direction zero because $I_0(t)$ does so. Since $\La$ almost surely has a unique infinite geodesic from $(0,0)$ with direction zero (Theorem \ref{thm:longgeo}), it follows that $\Gamma$ converges in law to said geodesic in the KPZ scaling limit.
		
    From the equality of laws between $\Gamma$ and $\Gamma'$, one concludes that $I_0(t)$ has the same law as the infinite geodesic of $\La$ from $(0,0)$ with direction zero.
	\end{proof}
	
	Let $h_0$ be a random initial condition independent of $\La$. We say $h_0$ is Brownian on compacts if for every interval $[-n,n]$, the law of $h_0(x)-h_0(0)$ on $[-n,n]$ is absolutely continuous with respect to that of a two-sided Brownian motion with diffusivity $\sqrt{2}$. Observe in this case every deterministic $p \in \R$ is a non-polar point of $h_0$ and, thus, the interface $I_p(t)$ of $h_0$ is uniquely defined by Proposition \ref{prop:intunique}.
	
	\begin{lem} \label{lem:intgeo2}
		Let $h_0$ be a random initial condition independent of $\La$ that is Brownian on compacts.
		For any $p \in \R$, the process $I_p(t)-p$ is geodesic-like on compacts.
	\end{lem}
	
\begin{proof}
    We may assume without loss of generality that $p=0$. The proof is built up in three stages of increasing complexity, which are:
    \begin{enumerate}
	\item \label{bc1} When $h_0$ is a two-sided Brownian motion.
	\item \label{bc2} When $h_0$ is a two-sided Brownian motion over $[-n,n]$ for some $n > 0$ and $-\infty$ outside the interval.
	\item \label{bc3} When $h_0$ is simply Brownian on compacts.
    \end{enumerate}
		
    For \eqref{bc1} the claim follows from Lemma \ref{lem:intgeo1} with an equality of laws.
		
    For the proofs of \eqref{bc2} and \eqref{bc3}, we note that a probability measure $\mu$ is absolutely continuous with respect to $\nu$ if there are events $A_m$ such that $\mu(A_m) \to 1$ and $\mu\mid_{A_m}$ is absolutely continuous with respect to $\nu$.
		
    Also, for an initial condition $h_0$, denote by $I_0(t; h_0)$ the interface from reference point $0$ associated to $h$ (provided it is uniquely defined as it will be in the following).
		
    For the proof of \eqref{bc2}, let $f_m$ be the continuous function that is 0 on $[-n,n]$, equal to $-m$ on the complement of $[-n-1,n+1]$, and linear on $[-n-1,-n]$ and $[n,n+1]$. Let $f_{\infty}$ be the pointwise limit of $f_m$. Let $B$ denote a two-sided Brownian motion. Observe that $h_0 = B + f_{\infty}$. Let $h_0^m = B + f_m$.
		
    By the Cameron-Martin Theorem, $h^m_0$ is absolutely continuous with respect to $B$. So $I_0(t; h^m_0)$ is geodesic-like on compacts by \eqref{bc1}. Let $t > 0$. We claim the event
    $$ A_m = \{I_0(s; h^m_0) = I_0(s; h_0)\;\text{for all}\; s \leq t \}$$
    holds eventually for all large $m$. In particular, $\mathbf{Pr}(A_m) \to 1$. This implies that $I_0(\cdot, h_0)$ is geodesic-like on every interval $[0,t]$, and since $t$ is arbitrary, \eqref{bc2} follows.
		
    To see that the events $A_m$ eventually hold, it helps to consider the Brownian path $B$ fixed. 
    There exists  $b > \sup_{s \in [0,t]} |I_0(s,h_0)|$ so that $(-b,t)$ is a UGP for all initial conditions $(h_0^m)^{-}$ and $(h_0)^{-}$, and $(b,t)$ is a UGP for all initial conditions $(h_0^m)^{+}$ and $(h_0)^{+}$. Such a $b$ exists since there are only countably many non-UGPs at time $t$ for each initial condition, see Theorem \ref{thm:ugc}. (One way to measurably choose from an intersection of an interval $[a,a+a']$ and the complement of a random measure zero set is to use $a+a'U$ with a uniform$[0,1]$ random variable $U$ independent of everything.)
  
    Let $\gamma_m$ be any geodesic to $(-b,t)$ from $(h_0^m)^{-}$ and $\gamma_{\infty}$ a geodesic to $(-b,t)$ from $(h_0)^{-}$. By definition,
    $$ \gamma_m(0) = \mathrm{argmax}_{y \leq 0}\, \{h^m_0(y) + \La(y,0; -b,t) \}.$$
    Let $Y=B(-n)+\La(-n,0;-b,t)$. Since $-n$ is not a polar point for $B$ almost surely (Example \ref{exm:BMpolar}), the maximum on the interval $[-n,0]$ is strictly greater than $Y$. On the other hand, the maximum on the interval  $(-\infty,-n)$ converges to $Y$ as $m\to\infty$. Thus, for large enough $m$,  $\gamma_m(0)\in (-n,0)$, and so $\gamma_m(0)=\gamma_\infty(0)$. Once this happens, $\gamma_m=\gamma_{\infty}$ since $(-b,t)$ is a UGP. 
  
    If the analogous condition also  holds for the geodesics from $(h_0^m)^+$ and $(h_0)^{+}$ to $(b,t)$, then by definition of the interface the event $A_m$ also holds. So we have shown that $A_m$ holds eventually almost surely.  
		
    For the proof of \eqref{bc3}, let $t>0$ be arbitrary. Let $b>0$ and consider the event $E_b$ that $\sup_{s \in [0,t]} |I_0(s; h_0)| < b$. Now let $E_{b,n}$ be the intersection of $E_b$ with the event that all geodesics from $h_0,h_0^+,h_0^-$ to $(-b,t)$ and $(b,t)$ start from within $[-n,n]$. On $E_{b,n}$, 
    \begin{equation*}
	I_0(s;h_0)=I_0(s; h_0|_{[-n,n]})\mbox{ for all } s \le t.
    \end{equation*}
    Here $h_0\mid_{[-n,n]}$ equals $h_0$ over $[-n,n]$ and is $-\infty$ outside.
		
    Note that $I_0\big(\cdot;h_0|_{[-n,n]}\big)=I_0\big(\cdot; h_0|_{[-n,n]}-h_0(0)\big)$ since the interface does not change when the initial condition is shifted by a constant. Since $h_0|_{[-n,n]}-h_0(0)$ is absolutely continuous with respect to $B|_{[-n,n]}$, part \eqref{bc2} implies that on the event $E_{b,n}$ the process $I_0(\cdot; h_0)$ restricted to $[0,t]$ is geodesic-like. Taking $b$ and then $n$ large, we infer \eqref{bc3}.
\end{proof}
	
	The next part of the argument requires some new notation. For an initial condition $h_0$ and $s \geq 0$, denote by $h_s$ the function $x \mapsto \La(h_0; x,s)$, which is $h_0$ grown to time $s$.
	For $t > s$, let $t \mapsto I(s,h_s,p,t)$ denote the interface from reference point $p$ for the initial condition $h_s$. We imagine that the interface starts from time $s$,
	and $p$ is assumed to be such that the interface is uniquely defined.
	
	The Markovian evolution of $s \mapsto h_s$ implies a Markovian evolution of an interface together with its environment:
	for any $s \geq 0$ it holds for every $t > s$ that
	\begin{equation} \label{eqn:intevol}
		I(0,h_0, p, t) = I(s,h_s, p_s, t) \quad \text{with} \quad p_s = I(0,h_0,p,s).
	\end{equation}
	This follows from the observation that a geodesic to $(x,t)$ from $h_s$ is the restriction of a geodesic to $(x,t)$ from $h_0$ to times $u \in [s,t]$.
	
	\begin{lem} \label{lem:intgeo3}
		Let $p$ be a non-polar point of $h_0$ which is also an interior reference point. So then $I(0,h_0,p,t)$ is uniquely defined.
		The following condition holds almost surely. For every $t > 0$ there is an $s_0$ with $0 < s_0 < t$ such that for any $s \leq s_0$,
		$$ I(0,h_0,p,t') = I(s,h_s,p,t') \quad \text{for all}\;\; t' \geq t.$$
	\end{lem}
	
	\begin{proof}
		We may assume that $p=0$. By the evolution of interfaces, it suffices to show the identity for $t'=t$.
		For all $s \in (0,t)$,
		$$
		z:=I(0,h_0,0,t)=I(s,h_s,p_s,t), \quad \text{where}\quad p_s=I(0,h_0,0,s).
		$$
		The second equality above, by the definition of interface, means that
		\begin{equation} \label{eqn:intgeo3a}
			\La(h_s|_{(-\infty,p_s]},s;z,t)=\La(h_s|_{[p_s,\infty)},s;z,t)=\La(h_s|_{\R},s;z,t)
		\end{equation}
		where
		$$ \La(h\mid_{J},s; z,t) = \sup_{y \in J} \, \{h(y) + \La(y,s;z,t) \}.$$

            Let $a_s$ and $b_s$ for $s \in [0,t]$ parameterize the leftmost and rightmost geodesics from $h_0$ to $(z,t)$, respectively.
            Due to the metric composition property of $\La$, $a_s$ and $b_s$ are both maximizers of the supremum in \eqref{eqn:intgeo3a} for every $s$ (with maximal value $\La(h_0;z,t)$).
            Indeed, consider the geodesic $s \mapsto a_s$. By definition, $a_0$ is a maximizer of the supremum $\La(h_0;z,t)$. Next, for any $s$ and $x \in \R$,
            $$h_s(x) + \La(x,s;z,t) = \sup_{y \in \R}\, \{h_0(y) + \La(y,0;x,s) + \La(x,s;z,t)\} \leq \sup_{y \in \R}\, \{h_0(y) + \La(y,0;z,t)\}$$ and the right hand side equals $\La(h_0;z,t)$.
            Moreover, $\La(a_0,0;a_s,s) + \La(a_s,s;z,t) = \La(a_0,0;z,t)$ because the path $s \mapsto a_s$ is a geodesic from $a_0$ to $(z,t)$.
            Thus, $h_s(a_s) + \La(a_s,s;z,t) \geq h_0(a_0) + \La(a_0,0;a_s,s) + \La(a_s,s;z,t) =\La(h_0;z,t)$.
  
		Since 0 is not a polar point of $h_0$, $a_0 < 0 < b_0$ almost surely.
		Therefore, by continuity of geodesics, there is an $s_0 > 0$ such that $a_s \leq 0 \leq b_s$ for all $s \in [0,s_0]$.
		Consider \eqref{eqn:intgeo3a} for $s \in (0,s_0]$.
		
		Suppose $p_s \geq 0$. Then, since $a_s \leq 0$,
		$$
		\La(h_s|_{(-\infty,p_s]},s;z,t)= \La(h_s|_{(-\infty,0]},s;z,t).
		$$
		On the other hand, since $p_s \geq 0$,
		$$
		  \La(h_s|_{[p_s,\infty)},s;z,t)\leq \La(h_s|_{[0,\infty)},s;z,t)\leq \La(h_s|_{\R},s;z,t)=\La(h_s|_{[p_s,\infty)},s;z,t).
		$$
		So
		$$
		\mathcal L(h_s|_{(-\infty,0]},s;z,t)=\mathcal L(h_s|_{[0,\infty)},s;z,t)
		$$
		which, by definition, means that  $z=I(s,h_s,0,t)$ as required.
		A symmetric argument applies when $p_s\le 0$ because $b_s \ge 0$.
	\end{proof}
	
	We can now complete the proof of Theorem \ref{thm:intgeo}.
	
	\begin{proof}
		We may assume $p=0$ without loss of generality. Let $[a,b] \subset (0,\infty)$ be a compact interval. By Lemma \ref{lem:intgeo3},
		there is an $s \in (0,a)$ such that $I(0,h_0,0,t) = I(s,h_s,0,t)$ for $t \in [a,b]$. By Lemma \ref{lem:intgeo2}, the process $t \mapsto I(s,h_s,0,t)$ over $[a,b]$ is absolutely continuous with respect to the geodesic of $\La$ from $(0,s)$ in direction zero because $h_s$ is Brownian on compacts \cite[Theorem 1.2]{SV}. This geodesic has the law of $t \mapsto g(t-s)$ for $t \geq s$, where $g$ is the geodesic of $\La$ from $(0,0)$ in direction zero. Now $g(t-s)$ is absolutely continuous over $t \in [a,b]$ with respect to $g(t)$ over $[a,b]$ by the absolute continuity of geodesics shown in Proposition \ref{prop:infgeoabscont}. Consequently, the process $I(0,h_0,0,t)$ is absolutely continuous over $[a,b]$ with respect to $g$.
	\end{proof}

	\subsection{Questions}
	
	\paragraph*{\textbf{No polar points for Brownian motion}}
	Show that Brownian motion almost surely has an empty polar set.
	
	\paragraph*{\textbf{Law of an interface}}
	Derive a method to compute, exactly, the law of $I_0(1)$ for interfaces with various initial conditions.
	This is interesting because interfaces are also the scaling limit of 2nd class particles in tasep.
	
	\paragraph*{\textbf{Likeliness between interfaces and geodesics}}
	Suppose for an initial condition $h_0$ the interface $I_0(t)$ from reference point 0 is uniquely defined.
	Is there a distributional identity between $I_0(t)$ and the infinite geodesic of $\La$ from $(0,0)$ in direction zero?
	Such an identity should be stronger than the absolute continuity result in Theorem \ref{thm:intgeo}.
	
	\section{A portrait of interfaces} \label{sec:portrait}
	Let $h_0$ be a fixed initial condition.
        The interface portrait of $h_0$ consists of all points lying on the finite interfaces associated to $h_0$:
	$$ \mathcal{I} = \{ (x,t) \in \Hp: x = I^{-}_p(t)\; \text{or}\; I^{+}_p(t) \;\text{for some reference point}\; p\}.$$
	
	\begin{figure}[ht!]
		\begin{center}
			\includegraphics[scale=2.5,trim={3.3cm 0 3.3cm 0},clip]{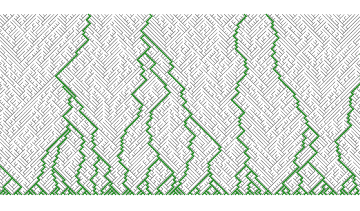}
			\caption{An interface portrait.}
			\label{fig:portrait}
		\end{center}
	\end{figure}
	
	This is the family of all finite interfaces, from all reference points, associated to $h_0$.
	It is an intriguing geometric object, see Figure \ref{fig:portrait}. Interfaces lie between geodesics,
	and the interface portrait is complementary to the web of geodesics emanating from $h_0$.
	The topics of this section are the topological and geometric properties of the portrait.
	
	Fix an initial condition $h_0$ and assume good samples of $\La$ relative to $h_0$ as in $\S \ref{sec:goodoutcomes}$.
	
	\subsection{Ordering of interfaces} \label{sec:orderofportrait}
	
	\begin{lem} \label{lem:order}
		If $p < q$ then $I^{+}_p(t) \leq I^{-}_q(t)$ for every $t$. So interfaces do not cross.
	\end{lem}
	\begin{proof}
		Any point $x$ that satisfies $I^{-}_q(t)\leq x \leq I^{+}_p(t)$ for some $t$ has a geodesic from $h_0$ to $(x,t)$ that starts at a point $\leq p$ and another that starts at a point $\geq q$. So $(x,t)$ is not a UGP. By the UGC property there can not be an interval of such points $x$, and so $I^{+}_p(t) \leq I^{-}_q(t)$.
	\end{proof}
	
	\begin{lem} \label{lem:leftright}
		Suppose the interface $I^{-}_p(t)$ is finite. Then there are reference points $p_n \to p$ from the left such that $I^{+}_{p_n}(t) \to I^{-}_p(t)$ for every $t$.
		Similarly, a finite right interface can be approximated from the right by left interfaces.
	\end{lem}
	
	\begin{proof}
		Since $I^{-}_p(t)$ is finite, there are points $a < p \leq b$ such that $h_0(a)$ and $h_0(b)$ are finite, according to the finiteness criterion in Proposition \ref{prop:intfinite}.
		For points $q \in (a,p)$ the right interface $I^{+}_q(t)$ is finite by the finiteness criterion.
		Moreover, by the previous lemma, $I^{+}_q(t)$ is non-decreasing in $q$ and bounded by $I^{-}_p(t)$; so it has a limit $\ell(t)$ as $q \to p$ from the left.
		If $\ell(t) < I^{-}_p(t)$ then consider a point $x$ between them. Any geodesic to $(x,t)$ must begin strictly to the left of $p$ and simultaneously to
		the right of every $q < p$. Such is impossible, so $\ell(t) = I^{-}_q(t)$ as required.
	\end{proof}
	
	\subsection{Topological view of the portrait} \label{sec:topoofportrait}
	For a reference for $p \in \R$, consider the set
	$$G_p = \{ (x,t) \in \Hp: \text{there is a geodesic from}\;h_0\; \text{to}\; (x,t)\; \text{starting at}\;(p,0)\}.$$
	The set $G_p$ lies inside the zero set $\{d_p=0\}$ of the competition function $d_p$.
	Like the zero set, $G_p$ is connected and, for every $t >0$, the intersection of $G_p$ with $\R \times \{t\}$ is an interval (possibly empty or infinite).
	
	\begin{lem} \label{lem:Gp}
		The set $G_p$ has the same interior as the zero set $\{d_p = 0\}$. All geodesics to interior points of $G_p$ start from $(p,0)$.
		The interfaces $I^{\pm}_p(t)$ are thus given by $\{d_p = 0\} \setminus \mathrm{Int}(G_p)$.
		Moreover, if $p \neq q$ then $\{d_p=0\}$ and $\{d_q=0\}$ (as well as $G_p$ and $G_q$) have disjoint interiors and can only intersect along their boundaries.
	\end{lem}
	
	\begin{proof}
		As $G_p$ lies inside the zero set $\{d_p=0\}$, $\mathrm{Int}(G_p) \subset \mathrm{Int}(\{d_p=0\})$.
		Suppose $(x,t)$ is an interior point of the zero set and consider any geodesic $g$ from $h_0$ to $(x,t)$, starting at $g(0)$.
		We will prove that $g(0)=p$, which shows that $\mathrm{Int}(\{d_p=0\}) \subset G_p$ and that all geodesics
		to $(x,t)$ start at $(p,0)$. Since $\mathrm{Int}(G_p)$ is the largest open set in $G_p$ and $\mathrm{Int}(\{d_p=0\})$ is open, it follows that $\mathrm{Int}(\{d_p=0\}) \subset \mathrm{Int}(G_p)$.
		
		Notice that a UGP in the zero set of $d_p$ must have its unique geodesic starting at $(p,0)$. By the UGC,
		since $(x,t)$ is an interior point of the zero set, there are UGPs $(x',t)$ and $(x'',t)$ that lie in the zero
            set and for which $x' < x < x''$. The unique geodesic to $(x',t)$ has to stay to the left of $g$, so that $p \leq g(0)$.
            Similarly, the geodesic to $(x'',t)$ stays to the right of $g$, so $g(0) \geq p$ as well.
		
		Finally, if $p \neq q$ and a point lies in both $\{d_p=0\}$ and $\{d_q=0\}$, then it cannot be a UGP. Due to the UGC, for every $t > 0$,
		the intervals $\{d_p=0\} \cap (\R \times \{t\})$ and $\{d_q=0\} \cap (\R \times \{t\})$ can intersect in at most one point.
            Therefore the sets $\{d_p=0\}$ and $\{d_q=0\}$ can only intersect along their boundaries and have disjoint interiors.
	\end{proof}
	
	\begin{prop} \label{prop:portraitclosed}
		The interface portrait is a closed subset of $\Hp$. It is also the closure in $\Hp$ of the set
		$$\{ (x,t) \in \Hp: x = I^{+}_p(t)\; \text{for some reference point}\;p\}$$
		of points on finite right interfaces or, alternatively, the closure of points on finite left interfaces.
	\end{prop}
	
	\begin{proof}
		Let $U$ be the union of the interior of the sets $\{d_p=0\}$ over all reference points $p$ of $h_0$.
		This is an open set and we claim that its complement in $\Hp$ is the interface portrait $\mathcal{I}$, which is then closed.
		
		If a point $(x,t)$ belongs to $\mathcal{I}$ then it is at the boundary of some zero set $\{d_p = 0\}$. So $(x,t)$ cannot lie in the interior of any zero set $\{d_q=0\}$
		by Lemma \ref{lem:Gp}. Therefore $\mathcal{I}$ lies inside the complement of $U$.
		
		Now suppose a point $(x,t)$ is not in $U$, which by Lemma \ref{lem:Gp} means that it does not lie in the interior of any $G_p$.
		The union of the sets $G_p$ over all reference points $p$ is the entirety of $\Hp$. So there is some $p$ such that $(x,t) \in G_p \setminus \mathrm{Int}(G_p)$.
		But then $(x,t)$ lies at the boundary of $\{d_p=0\}$ because $G_p \subset \{d_p=0\}$ and they have the same interior. So the complement of $U$ belongs to $\mathcal{I}$.
		
		The fact that $\mathcal{I}$ is the closure of all the finite right or left interfaces comes from Lemma \ref{lem:leftright}.
	\end{proof}
	
	Recall the function $e$ from \eqref{eqn:geoendpoint} that shows where geodesics emanate. It describes the interface portrait.
	
	\begin{prop} \label{prop:portraittopo}
		Call a point $(x,t)$ a discontinuity of the function $e$ if $x$ is a discontinuity of $e(\cdot, t)$.
		The interface portrait $\mathcal{I}$ is the set of discontinuities of $e$. 
		Thus, $\mathcal{I} \cap (\R \times \{t\})$ is a discrete set for every $t > 0$.
	\end{prop}
	
	\begin{proof}
		Let $e_t(x) = e(x,t)$ for a given $t$. Recall properties of $e$ from Proposition \ref{prop:endpointfunc}.
		
		If $x$ is not a discontinuity of $e_t$ then  $e_t$ is constant on $(x-\delta, x + \delta)$ for some $\delta > 0$ because $e_t$ is a step function with a discrete set of discontinuities. Let $p$ denote this constant. So the interval $(x-\delta, x+\delta) \times \{t\}$ lies in the zero set $\{d_p=0\}$. It is easy to see from this and continuity of the interfaces $I^{\pm}_p$ that the point $(x,t)$ then lies in the interior of $\{d_p=0\}$. So it cannot lie on any interface by Lemma \ref{lem:Gp}.
		
		Now suppose that $x$ is a discontinuity of $e_t$, meaning that the leftmost geodesic from $h_0$ to $(x,t)$ begins at some point $a$ and the rightmost one begins at $b > a$. Let $p$ be a point between $a$ and $b$. It is a reference point because $h_0(a)$ and $h_0(b)$ are finite. Clearly $d_p(x,t) = 0$. By Lemma \ref{lem:Gp}, the point $(x,t)$ does not belong to the interior of $\{d_p = 0\}$ because not all geodesics to it start at $p$. So it belongs to the boundary of $\{d_p = 0\}$ and lies on an interface.
		
		Discreteness of $\mathcal{I} \cap (\R \times \{t\})$  comes from discreteness of the discontinuities of $e_t$, see Proposition \ref{prop:endpointfunc}.
	\end{proof}
	
	\subsection{Geometric view of the portrait} \label{sec:geomofportarit}
	Geometrically, the interface portrait is a forest.
	
	\begin{thm}[Interface portrait is a forest] \label{thm:portraitgeom}
		The interface portrait $\mathcal{I}$ is a forest in $\Hp$ in the following sense: whenever two different interfaces meet at a positive time, they coalesce upward from then onwards. Moreover,
		geodesics do not cross interfaces. The only geodesics that meet $\mathcal{I}$ are those from points in $\mathcal{I}$.
	\end{thm}
	The theorem follows from a combination of the two lemmas below.
	
	\begin{lem} \label{lem:crossing}
		If $(x,t) \notin \mathcal{I}$ then every geodesic to it from $h_0$ lies outside $\mathcal{I}$. If $(x,t) \in \mathcal{I}$ then a geodesic from $h_0$ to $(x,t)$ can, looking backwards in time, follow an interface on which it belongs for a while and, if the geodesic ventures off, it doesn't meet $\mathcal{I}$ again. Every $(x,t) \in \mathcal{I}$ also has geodesics to it that lie (weakly) to the left and to the right of any interface on which $(x,t)$ belongs.
	\end{lem}
	
	\begin{proof}
		No UGP can lie on $\mathcal{I}$ because every point on an interface has a different leftmost and rightmost geodesic by Proposition \ref{prop:portraittopo}. So the geodesic from $h_0$ to a UGP won't meet any interface either. Now suppose $(x,t) \notin \mathcal{I}$ and let $g$ be a geodesic from $h_0$ to $(x,t)$. Suppose by way of contradiction that $g$ meets $\mathcal{I}$ at the interface $I$. Since $(x,t) \notin \mathcal{I}$ and the complement of $\mathcal{I}$ is open by Proposition \ref{prop:portraitclosed}, continuity of geodesics and interfaces imply that there is an $s < t$ such that $g(s) = I(s)$ while either $I < g$ or $I > g$ on $(s,t]$. Suppose $I < g$ on $(s,t]$. Due to the UGC, there is a UGP $(x',t)$ with $I(t) < x' < x = g(t)$. But due to geodesic ordering and continuity, the geodesic from $h_0$ to $(x',t)$ has to meet $I$ on $[s,t]$, which cannot be. So $g$ indeed does not meet $\mathcal{I}$.
		
		If $(x,t)$ lies on an interface $I$ then any geodesic to it can stay on $I$ from time $t$ downward, and if it ever leaves $I$ then it won't meet $\mathcal{I}$ again due to the previous assertion. By approaching $(x,t)$ with UGPs from the left and the right, and using geodesic compactness, it follows that there are geodesics to $(x,t)$ that stay to the left and to the right of $I$.
	\end{proof}
	
	\begin{lem} \label{lem:coalesce}
		If two different interfaces meet at a positive time then they coalesce upwards from then onwards.
	\end{lem}
	
	\begin{proof}
		Let $I$ and $J$ be finite interfaces (left or right). Suppose they meet at some point $(x,t) \in \Hp$, that is $I(t) = J(t) = x$.
		Now if $I$ and $J$ do not coalesce from time $t$ onward then there is a UGP not in $\mathcal{I}$ lying between them after time $t$, by the UGC and Proposition \ref{prop:portraittopo}. Its unique geodesic cannot meet $I$ or $J$ by Lemma \ref{lem:crossing},
		but since $I$ and $J$ meet at time $t$ and the geodesic ends between them after time $t$,
		it must encounter a point on $I$ or $J$ as it goes down to the initial condition. The contradiction means $I$ and $J$ do coalesce.
	\end{proof}
	
	\subsection{Questions}
	
	\paragraph*{\textbf{Geometry of interface portraits}}
	Find a criterion to determine when an interface portrait is a tree.
	For instance, prove that the flat (when $h_0 \equiv 0$) and stationary (when $h_0$ is Brownian motion) interface portraits are trees.
	
	\paragraph*{\textbf{Duality between interfaces and geodesics}}
	The interface portrait of the stationary initial condition (Brownian motion with diffusivity $\sqrt{2}$) is scale and translation invariant. (Here one considers the law of the interface portrait jointly over the Brownian motion and an independent directed landscape.) Study the distributional and geometric properties of the stationary portrait in detail. In particular, show that it has the same law as the geodesic tree from Theorem \ref{thm:geotree}.
	
	\section{The 2nd class particle} \label{sec:2ndclassparticle}
	
	This section derives the scaling limit of the 2nd class particle in tasep -- Theorem \ref{thm:scalinglimit}.
	We prove a more general result, Theorem \ref{thm:2ndclass},
	by relaxing the mode of convergence in \eqref{eqn:diffscaling}.
	There are several steps to the proof. First, we model the 2nd class particle in terms of a hole--particle pair and map its trajectory to
	a competition interface in last passage percolation. This mapping is well known and has been utilized for instance in \cite{CP, FGN, FP, FMP}.
	Next, we establish the scaling limit of the competition interface by using the fact that exponential last passage percolation
	scales to the directed landscape. Together, this leads to the scaling limit of the 2nd class particle. 

        We start with some background on the Burgers' equation that motivates our scaling. 

  \subsection{Tasep, the Burgers' equation and diffusive scaling}
  \label{s:burgers}
	Tasep is a microscopic model for the Burgers' equation with the 2nd class particle representing a microscopic characteristic.
	Suppose the initial condition of tasep is such that there is a macroscopic particle density $u_0(x)$:
	$$\int_0^x u_0(y) \, dy = \lim_{\eps \to 0} \eps \cdot \# \{ \text{particles inside}\; [0, \eps^{-1}x]\;\text{at time 0}.\}\,.$$
	Then for further times there is an almost sure macroscopic density $u(x,t)$ \cite{Rez2}:
	$$ \int_0^x u(y,t)\, dy = \lim_{\eps \to 0} \eps \cdot \# \{ \text{particles inside}\; [0, \eps^{-1}x]\; \text{at time}\; \eps^{-1}t \}.$$
	The density $u(x,t)$ satisfies the Burgers' equation (a scalar conservation law) in the form
	$$ \partial_t u + \partial_x (u (1-u)) = 0; \quad u(x,0) = u_0(x).$$
	
	There are 3 classical initial conditions for which the interplay between tasep and Burgers' equation has been studied.
	These are the rarefaction, shock and flat initial conditions. They are characterized by two parameters $\rho_{-}, \rho_{+} \in [0,1]$ for which
	$$u_0(x) = \rho_{-}\mathbf{1}_{x < 0} + \rho_{+} \mathbf{1}_{x \geq 0}.$$
	The rarefaction case is when $\rho_{-} > \rho_{+}$, shock being $\rho_{-} < \rho_{+}$, and flat being $\rho_{-} = \rho_{+}$.
	
	In all three cases the Burgers' equation is solved by finding characteristics, which are curves $x(t)$ along which $u(x(t),t)$ is constant.
	Characteristics satisfy the equation $x'(t) = 1 - 2u(x(t), t)$, starting from some $x(0) = x_0$. The characteristics are thus straight lines.
	
	In the rarefaction case there is a fan of characteristics emanating from the origin, and the solution $u(x,t)$ equals $u(x/t,1)$ where
	$$u(x,1) = \begin{cases}
		\rho_{-} & \text{if}\; x \leq 1- 2\rho_{-} \\
		\rho_{+} & \text{if}\; x \geq 1- 2\rho_{+} \\
		\text{linear in between} & \text{if}\; x \in [1-2\rho_{-}, 1- 2\rho_{+}]
	\end{cases}.
	$$
	A classical result is that in the rarefaction setting, a 2nd class particle in tasep started from the origin
	follows one of the characteristics inside the fan at random. If $X(t)$ is the position of the 2nd class particle at time $t$,
	then $X(t)/t$ converges in probability to a random variable uniformly distributed over the interval $[1-2\rho_{-}, 1- 2\rho_{+}]$ \cite{FK}.
	This convergence is later shown to hold almost surely \cite{MG}.
	
	In the shock case characteristics collide and Burgers' equation is not well-posed in the usual sense.
	There is the entropic solution, which, based on physical grounds, is taken to mean the correct solution in this setting, see \cite{Fer1, Rez1}.
	The entropic solution is simply the travelling front $u(x,t) = u_0(x - (1-\rho_{-}-\rho_{+})t)$.
	The line $x = (1-\rho_{-}-\rho_{+})t$ carries the discontinuity of the initial condition.
	It is also the macroscopic trajectory of a 2nd class particle in that a 2nd class particle
	from the origin has an asymptotic speed, the limit of $X(t)/t$ in probability, which is $1- \rho_{-}-\rho_{+}$ \cite{Lig}.
	
	Fluctuations of the 2nd class particle around its deterministic speed has been investigated in the shock setting.
	If the shock density profile is modelled by a Bernoulli random initial condition, $X(t)$ has Gaussian fluctuations on the scale of $t^{1/2}$ \cite{Lig}.
	When the initial condition is deterministic, say particles are placed periodically in large blocks to achieve the
	shock density profile, then $X(t)$ has Tracy-Widom fluctuations on the scale of $t^{1/3}$ \cite{FGN, QR}.
	
	Finally, consider the flat case $\rho_{-} = \rho_{+} = \rho$. The characteristics are the lines $x = x_0 + (1-2\rho)t$
	and $u(x,t) \equiv \rho$. A 2nd class particle from the origin has asymptotic speed $1-2\rho$ \cite{Fer1, Rez1, Sep1}. It is
	therefore interesting to investigate the fluctuations of its position $X(t)$ around $(1-2\rho)t$, which had remained open. This is the
	content of Theorem \ref{thm:scalinglimit}. We take $\rho = 1/2$. The initial conditions $X^{\eps}_0$ have a macroscopically
	flat density as $\eps \to 0$, but can differ from the identically flat density on scales of $\eps^{-1}$ along the $x$-axis
	and $\eps^{1/2}$ along the $u$-axis. This makes for the scaling \eqref{eqn:diffscaling}.

	\subsection{Statement of the limit theorem} \label{sec:limittheorem}
	
	Recalling back to the Introduction, consider tasep with a 2nd class particle at the origin and an initial condition $X_0(\cdot)$
	of regular particles. Encode the initial condition as a height function $h_0: \Z \to \Z$ according to
	\begin{equation} \label{eqn:initheight}
		h_0(0) = 0 \;\;\text{and}\;\; h_0(x+1) - h_0(x) = \begin{cases}
			1 & \text{if initially there is a hole at site}\; x\\
			-1 & \text{if initially there is a particle at site}\; x\\
		\end{cases}
	\end{equation}
	The 2nd class particle is treated as a particle in the above.
	
	Let $h^n_0(x)$ be a sequence of initial height functions. Extend them to $x \in \R$ by linear interpolation.
	The sequence converges if for a sequence $\eps_n \to 0$ one has that
	\begin{equation} \label{eqn:initialconv}
		- \eps_n^{1/2} h^n_0(2\eps_n^{-1}x) \to \mathbf{h}_0(x) \quad \text{as}\; n \to \infty.
	\end{equation}
	The sense in which this limit should hold is as follows, called hypograph convergence.
	Assume $\mathbf{h_0}(x)$ is an initial condition for the directed landscape, namely an
	upper semicontinuous function with values in $\R \cup \{-\infty\}$ and finite somewhere.
	The convergence should be in the natural sense for upper semicontinuous functions: 
	for every $x \in \R$ and sequence $x^n \to x$,
	$$\limsup_{n \to \infty} \; - \eps_n^{1/2}h^{n}_0(2\eps_n^{-1}x^n) \leq \mathbf{h}_0(x),$$
	and there exists some sequence $x^n \to x$ such that
	$$\liminf_{n \to \infty} \; - \eps_n^{1/2}h^{n}_0(2\eps_n^{-1}x^n) \geq \mathbf{h}_0(x).$$
	Assume also that there is a constant $c$ such that for every $n$ and $x$,
	$$-\eps_n^{1/2}h_0^{n}(2\eps_n^{-1}x) \leq c (1 + |x|).$$
	
	\begin{thm}[Limit of the 2nd class particle] \label{thm:2ndclass}
		Consider tasep with a sequence of initial conditions $X^{n}_0$ containing a 2nd class particle at the origin.
		Let $X^{n}(t)$ be the position of the 2nd class particle at time $t$ for the initial condition $X^n_0$.
		Let $h^{n}_0$ denote the corresponding initial height functions given by \eqref{eqn:initheight}.
		
		Suppose $h^{n}_0$ converges to $\mathbf{h}_0$ in the hypograph topology as in \eqref{eqn:initialconv}.
		Furthermore, suppose the interface $I_0(t)$ of the directed landscape from reference point $p=0$ for the
		initial condition $\mathbf{h}_0$ is uniquely defined (see Definition \ref{def:unique} and $\S$\ref{sec:uniquedef} for such a criterion).
		
		Then the re-scaled trajectory of the 2nd class particle
		\begin{equation} \label{eqn:2ndscaled}
			\mathbf{X}^{n}(t) = (\eps_n/2) X^{n}(\eps_n^{-3/2}t)
		\end{equation}
		converges to $I_0(t/2)$ in law, uniformly in $t$ over compact subsets of $(0,\infty)$.
	\end{thm}
	
	The convergence of $\mathbf{X}^{n}(t)$ to $I_0(t/2)$ may not extend to $t = 0$. The 2nd class particle can move
	a positive KPZ-distance in zero KPZ-time. For example, consider the initial condition $\mathbf{h}_0(x)$
	with narrow wedges at $x = -1$ and $x= 1$ only. The interface from reference point $p = 1/2$, say, begins at $x = 0$ at time zero.
	In tasep terms this means that when the initial conditions $h^{n}_0$ approximate $\mathbf{h}_0$, a 2nd class particle
	started from site $x = \eps_n^{-1}/2$ moves to a site $x = o(\eps_n^{-1})$ in time scales of order $o(\eps_n^{-3/2})$.
	
	\subsection{The 2nd class particle as a hole--particle pair} \label{sec:holeparticle}
        In this section we explain how tasep with a single second class particle can be modelled in terms of tasep with a marked hole--particle pair.
	Consider tasep with some initial configuration of particles but with no 2nd class particle.
	Let there be a particle at site 0, labelled 0, and a hole at site $-1$ with label 0, too.
	Then label the particles right to left and the holes left to right, as in Figure \ref{fig:0}.
	Write $X_t(n)$ for the position of particle number $n$ at time $t$; for example, $X_0(1) = -3$ in Figure \ref{fig:0}.
	\begin{figure}[ht!]
		\begin{center}
			\includegraphics[scale=0.5]{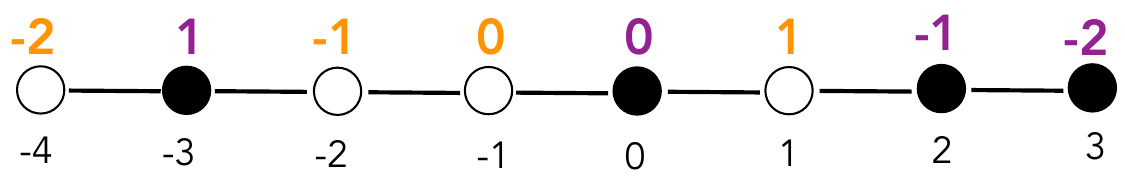}
			\caption{Snapshot of a tasep initial configuration.
				Particles are labelled right to left with particle 0 at site 0; holes are labelled left to right with hole 0 at site $-1$. Here, particle number 1 is at site $-3$ and hole number 1 is at site $1$.}
			\label{fig:0}
		\end{center}
	\end{figure}
	
	Consider the pair hole 0 -- particle 0 that are initially adjacent at sites $-1$ and $0$.
	Mark it as a pair like so: $\circ - \bullet$.
	This pair remains intact under the dynamics of tasep and moves left or right together as shown in Figure \ref{fig:1}.
	Note although the label of the particle or hole in the pair changes, the pair itself is always intact.
	
	\begin{figure}[ht]
		\begin{center}
			\includegraphics[scale=0.5]{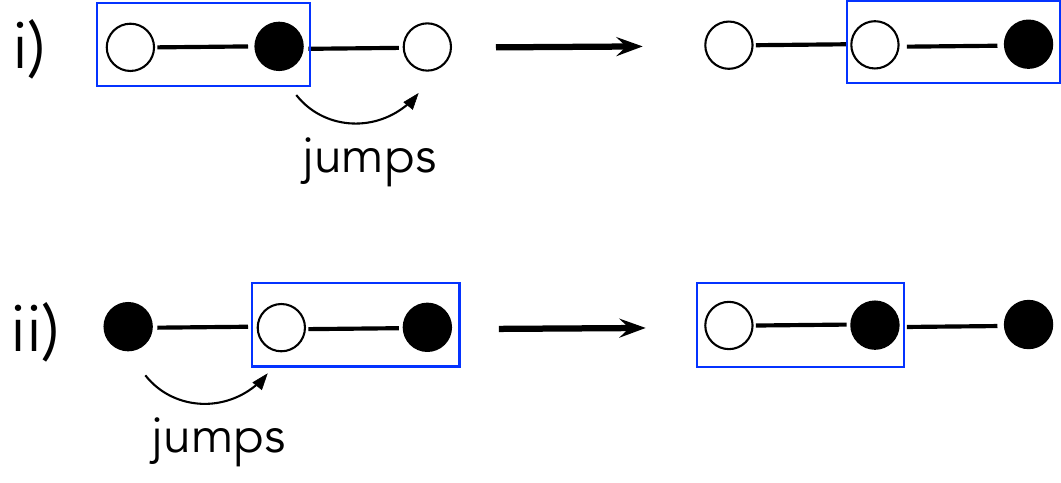}
			\caption{How the $\circ-\bullet$ pair moves in tasep.
				In (i) it moves right and in (ii) it moves left. It does not move otherwise.}
			\label{fig:1}
		\end{center}
	\end{figure}
	
	The $\circ-\bullet$ pair moves according to the same rules as a 2nd class particle!
	Let $X(t)$ be the position of the particle in the $\circ - \bullet$ pair at time $t$.
	Consider, at each time, the configuration obtained by removing the hole in the pair and then sliding all particles
	and holes to its left by one unit rightward. Mark the surviving particle from the $\circ-\bullet$ pair.
	The following observation is due to Ferrari, Martin and Pimentel; the lemma is from \cite[Lemma 6]{FP}.
	
	\begin{lem} \label{lem:HPpair}
		The new system (after removing, sliding and marking) evolves like tasep with a 2nd class particle,
		the 2nd class particle being the marked one. So $X(t)$ is the location of the 2nd class particle at time $t$.
	\end{lem}
	
	In order to analyse tasep with a 2nd class particle, it is therefore enough to follow the $\circ - \bullet$ pair in the original system,
	and this is what we will do. Let $X_n$ be the position of the particle in the $\circ-\bullet$ pair after it has taken $n$ steps:
	\begin{equation} \label{eqn:2ndnstep}
		X_n = \text{position of the particle in the}\; \circ-\bullet \; \text{pair after it has taken}\; n \; \text{steps}.
	\end{equation}
	Equivalently, by Lemma \ref{lem:HPpair}, $X_n$ is the position of the 2nd class particle after $n$ steps.
	Clearly, $X_0 = 0$ and $X_{n+1}-X_n \in \{\pm 1\}$. The process $X(t)$ is a time change of $X_n$.
	Let $\tau_n$ be the time when the $\circ-\bullet$ pair takes its $n$-th step. Then, with $\tau_0 = 0$,
	\begin{equation} \label{eqn:timechange}
		X(t) = X_n \quad \text{for}\; t \in [\tau_n, \tau_{n+1}).
	\end{equation}
	We will describe the evolution of $X_n$ in terms of a competition interface in last passage percolation and use \eqref{eqn:timechange}
	to get the scaling limit of the 2nd class particle.
	
	\subsubsection{Height functions}
	Consider the encoding \eqref{eqn:initheight} of an initial condition $X_0$ of tasep as a height function $h_0$.
	When $X_0$ contains a 2nd class particle at the origin, map it to an initial condition $\hat{X}_0$ with a $\circ-\bullet$ pair at $-1-0$ as follows.
	Remove the 2nd class particle, slide all particles and holes of $X_0$ to the left of the origin by one unit leftward, and insert a $\circ-\bullet$ pair at $-1-0$.
	Let $\hat{h}_0$ be the height function associated to $\hat{X}_0$. The relation between $h_0$ and $\hat{h}_0$ is
	$$ \hat{h}_0(x) = \begin{cases}
		h_0(x) & \text{if}\; x \geq 0 \\
		h_0(x+1) -1 & \text{if}\; x < 0
	\end{cases}
	$$
	
When a sequence of initial height functions $h^{n}_0$ of tasep converges to $\mathbf{h}_0$ according to \eqref{eqn:initialconv}, the corresponding functions $\hat{h}^{n}_0$ converge to $\mathbf{h}_0$ under the same scaling. Therefore, in order to prove Theorem \ref{thm:2ndclass}, it is enough to consider tasep with initial conditions that have a $\circ-\bullet$ pair instead of a 2nd class particle, assume the initial height functions converge according to \eqref{eqn:initialconv}, and prove that the trajectory of the $\circ-\bullet$ pair converges to the interface under the scaling \eqref{eqn:2ndscaled}.
	
We remark that the height function $h_0(x)$ is a simple random walk path. Such a path has a peak at site $x$ if $h_0(x \pm 1) = h_0(x) - 1$, and a valley at $x$ if $h_0(x \pm 1) = h_0(x) + 1$. Having a $\circ - \bullet$ pair at sites $-1 - 0$ corresponds to $h_0$ having a peak at $x=0$.
	
The evolution of tasep defines a growing family of height functions $h_t(x)$ for $t \geq 0$ that are simple random walk paths. Starting from $h_0(x)$, valleys simply turn to peaks at exponential rate 1, and $h_t(x)$ is the state of the function at time $t$. The trajectory of the $\circ-\bullet$ pair is that of the initial peak at $x=0$, so the 2nd class particle has the same evolution as a peak of $h_0$.
	
\subsection{Mapping the 2nd class particle to a competition interface} \label{sec:2ndandcomp}
The previous section explained how tasep with a 2nd class particle can be identified as tasep with a marked $\circ-\bullet$ pair. We now explain how the trajectory of the $\circ-\bullet$ pair can be mapped to a competition interface.
	
\subsubsection{Exponential last passage percolation}
From the evolution of a tasep, define the function
\begin{equation*}
    G(i,j) = \text{time when hole}\; i \;\text{interchanges position with particle}\; j.
\end{equation*}
Set $G(i,j) = 0$ if hole $i$ is initially left of particle $j$, with the conventions that $\mathrm{position}(\mathrm{hole} \; i) = -\infty$ if there is no hole $i$ and $\mathrm{position}(\mathrm{particle} \; j) = \infty$ if there is no particle $j$. Recall particles are labelled right to left and holes left to right, and particle 0 is initially at site 0 and hole 0 is initially at site $-1$.
	
Function $G$ is the model of last passage percolation. This mapping between tasep and last passage percolation goes back to Rost \cite{Ro}. Consider the set
$$\Lambda = \{(i,j) \in \Z^2 : G(i,j) = 0 \; \text{and}\; G(i+1,j+1) > 0 \}.$$
$\Lambda$ is a down/right path in $\Z^2$ which is determined by, and in one to one correspondence with, the initial configuration of tasep. Let $\omega'_{i,j}$ be independent exponential random variables with mean 1 for $i,j \in \Z$. Then $G(i,j)$ is determined by
\begin{equation} \label{eqn:LPP}
    G(i,j) = \max \{ \, G(i-1,j), G(i,j-1) \, \} + \omega'_{i,j}
\end{equation}
with boundary conditions $G(i,j) = 0$ on $\Lambda$. Indeed, in order for particle $j$ and hole $i$ to swap positions they must first be adjacent, which happens at time $\max \{ G(i-1,j), G(i,j-1)\}$. Then they wait an independent exponential unit of time before swapping, which accounts for adding $\omega'_{i,j}$.
	
Now change coordinates in last passage percolation from $(i,j) \in \Z^2$ to
$$ x = i-j,  \quad y = i+j, \quad \omega_{x,y}=\omega'_{i,j}$$
This takes $\Z^2$ to
$$ \Ze^2 = \{(x,y) \in \Z^2: x+y\; \text{is even}\}.$$
Define
\begin{equation} \label{eqn:LPPg}
    g(x,y) = G((x+y)/2, (y-x)/2) \;\; \text{for}\;\; (x,y) \in \Ze^2, 
\end{equation}
The boundary $\Lambda$ is mapped to the initial height function $h_0(x)$ defined by \eqref{eqn:initheight}. Unwrapping the recursion \eqref{eqn:LPP} shows that $g$ satisfies
\begin{equation} \label{eqn:LPPmax}
	g(x,y) = \max_{\pi} \, \sum_{\substack{k=0 \\ (x_k,y_k) \in \pi}}^{n-1} \omega_{x_k,y_k}
\end{equation}
where $\pi$ is any upward directed path in $\Ze^2$ from the graph of $h_0$ to $(x,y)$. Here directed means that at each step the path moves in the direction $(-1,1)$ or $(1,1)$.
	
As the $\omega_{i,j}$ have a continuous distribution we note that, almost surely, $g(x,y)$ takes distinct and positive values for every $(x,y)$ in $\Ze^2$ and the maximizing paths $\pi$ in \eqref{eqn:LPPmax} are unique. Also, any maximizing path $\pi$ always starts from a valley of $h_0$ because it can be continued down a slope of $h_0$ until it reaches a valley.
	
Now recall the process $X(t)$ which is the position of the particle in the $\circ-\bullet$ pair at time $t$, and the process $X_n$ from \eqref{eqn:2ndnstep} that records the steps of the pair. The following lemma is from \cite[Proposition 3]{FP}; it describes the evolution of $X(t)$.
\begin{lem} \label{lem:2ndLPP}
Starting with $X_0=0$, for every $n \geq 0$,
$$ X_{n+1} = \begin{cases}
	X_n + 1 & \text{if}\;\; g(X_n+1,n+1) < g(X_n-1,n+1) \\
	X_n - 1 & \text{if}\;\; g(X_n-1,n+1) < g(X_n+1,n+1)
\end{cases}.$$
In other words, $(X_{n+1}, n+1) = \mathrm{argmin}\, \{ g(X_n-1,n+1), g(X_n+1, n+1)\}$. Moreover, let
\begin{equation} \label{eqm:tau}
    \tau_n = g(X_n,n)\end{equation}
for $n \geq 1$ with $\tau_0=0$. Then $X(t) = X_n$ for $t \in [\tau_n, \tau_{n+1})$.
\end{lem}
	
\subsubsection{Competition interface}
Split the initial height function $h_0$ at the origin according to
$$ h_0^{-}(x) =
	\begin{cases} h_0(x) & \text{for}\;\; x \leq 0 \\ x & \text{for}\;\; x > 0 \end{cases} \quad \text{and}\quad
	h_0^{+}(x) = 
	\begin{cases} h_0(x) & \text{for}\;\; x > 0 \\ -x & \text{for}\;\; x \leq 0 \end{cases}.$$
Let
\begin{equation} \label{eqn:gpm}
g(x,y; h_0^{\pm})
\end{equation}
denote the last passage times associated to the initial height functions $h^{+}_0$ and $h^{-}_0$, respectively, see \eqref{eqn:LPPmax}. Note the same weights $\omega_{x,y}$ are being used for $g(\cdot; h_0^{\pm})$ as for $g$, and it is only the initial conditions that differ.
	
The variational principle \eqref{eqn:LPPmax} shows that
	$$g(x,y) = \max \{g(x,y; h_0^{-}), g(x,y; h_0^{+})\}.$$
The $\circ-\bullet$ pair induces a peak of $h_0$ at $x=0$. Therefore, the maximizing path $\pi$ from $h_0$ to $g(x,y)$ starts to the left or right of the origin and, since the weight are continuous, there are no ties between $g^{-}(x,y)$ and $g^{+}(x,y)$. The discrete competition function
    \begin{equation} \label{eqn:discomfun}
    d(x,y) = g(x,y; h^{+}_0) - g(x,y; h^{-}_0)
    \end{equation}
    is non-decreasing in $x$ and never vanishes. Define the sets Yellow (Y) and Blue (B) as
	\begin{align*}
		Y & = \{d < 0\} = \{(x,y) \in \Ze^2: \text{maximizing path from}\; h_0\; \text{to}\; (x,y) \; \text{starts at}\; h^{-}_0\} \\
		B & = \{d > 0\} = \{(x,y) \in \Ze^2: \text{maximizing path from}\; h_0\; \text{to}\; (x,y) \; \text{starts at}\; h^{+}_0\}
	\end{align*}
	
	If $(x,y) \in B$ then both $(x+1,y\pm 1) \in B$ as well.
	Indeed, the maximizing path to $(x+1, y \pm 1)$ from $h_0$ lies weakly to the right of the maximizing
	path to $(x,y)$. This observation implies if $(x,y) \in B$ then $(x+2,y) \in B$, by the chain of inclusions
	$(x,y)  \to (x+1,y+1) \to (x+2,y)$. This shows $B$ is connected and filled in to the right. Likewise,
	$Y$ is connected and filled in to the left because if $(x,y) \in Y$ then $(x-1,y\pm1) \in Y$ too.
	
	We infer that for every $n \geq 1$ there is an $I_n \in \Z$ such that
	$$ Y \cap (\Z \times \{n\}) = \{\cdots, I_n-4, I_n-2, I_n \} \; \text{and}\; B \cap (\Z \times \{n\}) = \{I_n+2, I_n+4, I_n+6, \cdots \}.$$
	The discrete competition interface $(I_n, n \geq 0)$ is defined by $I_0 = 0$ and
	\begin{equation} \label{eqn:compint}
		I_n = \text{rightmost point of}\; Y \cap ( \Z \times \{n\}).
	\end{equation}
	
	\begin{lem} \label{lem:compint}
		For every $n \geq 0$, $X_n \in \{I_n, I_n+2\}$ and $(X_n-1,n+1) \in Y$ while $(X_n+1,n+1) \in B$.
	\end{lem}
	
	\begin{proof}
		Clearly, $X_0 = I_0 = 0$. Observe that $(-1,1)$ is yellow and $(1,1)$ is blue because $h_0$ has a peak at the origin. So the claim holds
		for $n=0$, and we proceed by induction. Assuming the claim holds for $n$, consider the colour of site $(X_n, n+2)$.
		We will prove the claims at time $n+1$ when $(X_n,n+2)$ is yellow. The proof when $(X_n,n+2)$ is blue is analogous.
		
	Suppose $(X_n,n+2)$ is yellow. Since $(X_n-1,n+1) \in Y$ and $(X_n+1,n+1) \in B$ by hypothesis, the maximizing path to $(X_n,n+2)$ enters from $(X_n-1,n+1)$ instead of $(X_n+1,n+1)$. This means $g(X_n+1, n+1) < g(X_n-1,n+1)$, so $X_{n+1} = X_n + 1$ by Lemma \ref{lem:2ndLPP}.
		Now $(X_{n+1}+1,n+2)$ is blue since its lower-left site $(X_{n+1},n+1) = (X_n+1,n+1)$ is blue, and $(X_{n+1}-1,n+2) = (X_n,n+2)$
		is yellow by assumption. Also, $I_{n+1} = X_n-1$ and so $X_{n+1} = X_n+1 = I_{n+1}+2$. Both claims at time $n+1$ are established when $(X_n,n+2)$ is yellow.
	\end{proof}

\subsection{Scaling limit of exponential last passage percolation} \label{sec:lpplimit}
We first discuss the scaling limit of exponential last passage percolation, from which the limit of competition interfaces follows.

 \begin{thm} \label{thm:lppconvergence}
 Consider a sequence $\eps_n \to 0$. There exists a coupling of mean 1 exponential random variables $\omega_{x,y,n}, x,y\in  \Ze^2, n\ge 1$, and a directed landscape $\La$ such that for each $n$ the variables $\omega_{x,y,n}, x,y\in  \Ze^2$ are independent and the following hold. For $p,q\in \Ze^2$ define the exponential last passage values 
 $$L_n(p; q) = \max_{\pi} \sum_{(x_k,y_k) \in \pi} \omega_{x_k,y_k,n}$$
 where $\pi$ is any upward directed path in $\Ze^2$ from $p$ to $q$. For an initial condition $h_0$, define 
	\begin{equation} \label{eqn:LPPmax3}
		g_n(x,y;h_0) = \max_{z\in \mathbb Z}L_n(z,h_0(z);x,y)
	\end{equation}
 This is the same as  definition \eqref{eqn:LPPmax}.

Consider any collection of initial conditions $h_0^n$ so that with $\eps=\eps_n$,
their re-scaling $$\mathbf h_0^{n}(x)=-\eps^{1/2} h_0^n(2\eps^{-1}x)$$ satisfies
	\begin{enumerate}
		\item $\mathbf h_0^{n} \to \mathbf h_0$ in hypograph topology (see the discussion under \eqref{eqn:initialconv}), for some $\mathbf h_0: \R \to \R \cup \{-\infty\}$ with $\mathbf h_0 \not \equiv -\infty$, and 
		\item there exists $c \in \R$ such that for all $x$, $\mathbf h_0^{n}(x) < c(1+ |x|)$ for every $n$.
	\end{enumerate}
 
Let $h^n_t(x;h_0)$ denote the evolution of the tasep height function at time $t$ from initial condition $h_0$ as defined in \eqref{eqn:LPP} using the random variables $\omega_{x,y,n}$.

The following convergences hold almost surely as $n \to \infty$.
\begin{enumerate}[(i)]
    \item $\frac{\eps^{1/2}}{2}L_n(2\eps^{-1}y,\eps^{-3/2}s;2\eps^{-1}x,\eps^{-3/2}t)-\frac{t-s}{\eps} \to \mathcal L(y,s;x,t)$ compactly on $\{(y,s;x,t)\in \mathbb \R^4: s<t\}$.
    \item $\frac{t}{\eps} - \eps^{1/2}h^n_{2\eps^{-3/2}t}(2\eps^{-1}x;h_0^n) \to \mathcal L(\mathbf h_0; x,t)$ compactly over $(x,t) \in \Hp$.
    \item $\frac{\eps^{1/2}}{2} g_n(2\eps^{-1}x,\eps^{-3/2}t;h_0^n) - \frac{t}{\eps} \to \mathcal L(\mathbf h_0; x,t)$ compactly over $(x,t) \in \Hp$.
\end{enumerate} 
\end{thm}

\begin{proof}
We will follow the proof of Theorem 16.5 in \cite{DV}. The notation is slightly different there.  The word interface there does not refer to competition interfaces, but rather to the boundary of a directed metric neighborhood.
In \cite{DV} the tasep height function is defined so that it has the same sign as its limit (or the opposite sign as presented here). The correspondence between $n_{DV}$ there and $\eps_n$ here is $\eps_n =n_{DV}^{-2/3}$. So when we explain the translation between the two notations, we will just use the case $\eps_n=n^{-2/3}$ (so that we have $n=n_{DV}$) for simplicity. The coupling mentioned here comes from the coupling stated there.

\textbf{Proof of Claim  (i)} This is the convergence of exponential last passage percolation on the rotated lattice $\Ze^2$. This is proved in Corollary 16.1 (see also Corollary 13.2 for the stated coupling), with the translation 
$$
(d_X)_{DV}(x,-t;y,-s)=L_n(y,s;x,t);
$$
$$
(d_n)_{DV}(x,-t;y,-s)=\frac{\sqrt{\eps}}{2}L_n(\frac{2y}\eps,\frac{ s}{\eps^{3/2}};\frac{2x}\eps, \frac{t}{\eps^{3/2}})-\frac{t-s}\eps
$$
and by Corollary 16.1, we have $(d_n)_{DV}\to \mathcal L_{DV}$ compactly, with the correspondence $\La_{DV}(x,-t;y,-s)=\La(y,s;x,t)$. This proves (i).

\textbf{Proof of Claim (ii)} The claim is a  strengthening of \cite[Theorem 16.5]{DV} to compact convergence over $\Hp$.  In \cite{DV}, the proof of Theorem 16.5 uses Theorem 15.5, which relies on the tightness condition 15.4. The tightness condition can be strengthened to the following. 
 
\textbf{Strengthened tightness condition 15.4}. For every compact  $K\subset \Hp$, slope and height $s',h^*\in \mathbb R$ there exists $a\in \mathbb R$ and initial conditions $f_n:\R \to \R$ with $\mathbf{f}_n(x)\ge s'|x|-a$ for all $x\in \R$ and $\sup_{(y,t)\in K} \tilde h_n(y,t;f_n)\le h_*$ for all large enough $n$. The notations $\mathbf{f}_n(x) = - \eps^{1/2} f_n(2\eps^{-1}x)$ and $\tilde h_n(y,t;f_n) = \frac{t}{\eps} - \eps^{1/2}h^n_{2\eps^{-3/2}t}(2\eps^{-1}y;f_n)$ with $\eps = \eps_n$.

The proof of \cite[Theorem 15.5]{DV} implies Claim (ii) assuming the strengthened tightness condition 15.4. Lemma \ref{lem:tightnesscond} proves that strengthened tightness condition 15.4 holds in the setting of exponential last passage percolation. 

\textbf{Proof of Claim (iii)} Denote $\mathbf{g}_n(x,t; h_0^n) = \frac{\eps^{1/2}}{2} g_n(2\eps^{-1}x,\eps^{-2/3}t ;h_0^n) - \frac{t}{\eps}$.
Using \eqref{eqn:LPPmax3} we have that
\begin{equation} \label{eqn:gmax}
    \mathbf{g}_n(x,t; h_0^n) = \sup_{y \in \R}\, \{ d_n(x, -t; y, \eps \mathbf h_0^n(y)) + \mathbf h_0^{n}(y)\}
\end{equation}
where $d_n$ is $(d_n)_{DV}$ from above. By Claim (i), $d_n(x,-t;y,-s) \to \La(y,s;x,t)$ uniformly on compacts.

Since $h_0^n$ is a simple random walk path, $|h_0^n(x)| \leq |x|$, which implies that $|\eps \mathbf h_0^n(y)| \leq 2 \eps^{1/2} |y|$ and it tends to zero uniformly on compacts. Thus, $d_n(x, -t; y, \eps \mathbf h_0^n(y)) \to \La(y,0;x,t)$ uniformly on compacts.

Recall $\La(\mathbf h_0; x,t) = \sup_{y \in \R} \{\La(y,0;x,t) + \mathbf h_0(y)\}$. Suppose there is a compact subset $D \subset \R$ such that $h^n_0$ does not have any valleys outside of $(2/\eps_n)D$ for every $n$, that is, $h^n_0$ is decreasing outside the set $(2/\eps_n)D$. Then $\mathbf h_0(y) = - \infty$ for $y \notin D$. Moreover,
$$\mathbf{g}_n(x,t; h_0^n) = \sup_{y \in D}\, \{ d_n(x, -t; y, \eps \mathbf h_0^n(y)) + \mathbf h_0^{n}(y)\}; \quad \La(\mathbf h_0; x,t) = \sup_{y \in D} \{\La(y,0;x,t) + \mathbf h_0(y)\}.$$
In this case compact convergence of $d_n$ to $\La$ and hypograph convergence of $\mathbf h_0^n$ to $\mathbf h_0$ implies convergence of $\mathbf g_n(x,t; h_0^n)$ to $\La(\mathbf h_0; x,t)$ compactly over $(x,t) \in \Hp$.

In order to prove convergence in general, we use a tightness condition for last passage times, the proof of which is in Lemma \ref{lem:lpptightness}.

\textbf{Tightness condition for LPP}. For every compact  $K\subset \Hp$, slope and height $s',h^*\in \mathbb R$ there exists $a\in \mathbb R$ and initial conditions $f_n:\R \to \R$ with $\mathbf{f}_n(x)\ge s'|x|-a$ for all $x\in \mathbb R$ and $\sup_{(y,t)\in K} \mathbf{g}_n(y,t;f_n)\le h_*$ for all large enough $n$. Here $\mathbf{f}_n = -\sqrt{\eps} f_n(2x/\eps)$ with $\eps = \eps_n$.

Let $K \subset \Hp$ be compact.
Note that $\mathbf h_0 \leq c (1+ |x|)$ for some constant $c$. The bound \eqref{eqn:DOVAbound} on $\La$ implies that there is a compact set $D \subset \R$ such that
\begin{equation} \label{eqn:Drestriction}
\La(\mathbf h_0; y,t) = \La(\mathbf{h}_0|_{D}; y,t) \quad \text{for}\; (y,t) \in K; \;\; \mathbf{h}_0|_{D}(x) = \begin{cases} \mathbf h_0(x) & x \in D \\ -\infty & x \notin D \end{cases}.
\end{equation}

Let $x_0$ be such that $\mathbf h_0(x_0) > - \infty$ and let $x_n \to x_0$ such that $\mathbf h_0^n(x_n) \to \mathbf h_0(x_0)$. Then \eqref{eqn:gmax} and Claim (i) imply that
\begin{align*}
    \liminf_n \inf_{(y,t) \in K} \mathbf{g}_n(y,t;h^n_0) &\geq \lim_n \inf_{(y,t) \in K} d_n(y,-t; x_n, \eps_n \mathbf h_0^n(x_n)) + \mathbf h_0^{n}(x_n) \\
    & = \inf_{(y,t) \in K} \La(x_0,0;y,t) + \mathbf h_0(x_0) := b.
\end{align*}
With $b$ as defined above, for all large enough $n$, $\mathbf{g}_n(y,t;h^0_n) \geq b-1$ for every $(y,t) \in K$. Let $f_n$ be as in the tightness condition above with parameters $s' = c+1$ and $h^* = b-2$. Then for all large enough $n$, $\mathbf{g}_n(y,t;h^n_0) > \mathbf{g}_n(y,t;f_n)$ for every $(y,t) \in K$.
As $\mathbf{h}_0^n(x) \leq c(1+|x|)$, there is a compact set $E$ containing the set $D$ in its interior such that $\mathbf{h}_0^n < \mathbf{f}_n$ on $E^c$ (outside $E$). Looking at \eqref{eqn:gmax}, since $\mathbf{h}_0^n < \mathbf{f}_n$ on $E^c$ but $\mathbf{g}_n(y,t;h^n_0) > \mathbf{g}_n(y,t;f_n)$ for every $(y,t) \in K$, it follows that all maximizers achieving the value $\mathbf{g}_n(y,t;h^n_0)$ must belong to $E$. Therefore,
\begin{equation} \label{eqn:Erestriction}
\mathbf{g}_n(y,t; h_0^n) = \mathbf{g}_n(y,t; h_0^n|_{E}) \quad \text{for}\; (y,t) \in K.
\end{equation}
Here $h_0^n|_{E}$ is the initial condition that equals $h_0^n$ on the set $(2/\eps_n)E$ and is modified to decrease outside $(2/\eps_n)E$ (so that all valleys are inside $(2/\eps_n)E$).

Now $h_0^n|_{E}$ converges in the hypograph topology to $\mathbf{h}_0|_{E}$. By the first case that we proved,
\begin{equation} \label{eqn:Erestconv}
  \mathbf{g}_n(y,t;h_0^n|_{E}) \to \La(\mathbf{h}_0|_{E}; y,t)  
\end{equation}
uniformly over $(y,t) \in K$. We have that $\mathbf{h}_0|_{D} \leq \mathbf{h}_0|_{E} \leq \mathbf{h}_0$ because $D \subset E$.
Therefore, $\La(\mathbf{h}_0|_{D}; y,t) \leq \La(\mathbf{h}_0|_{E}; y,t) \leq \La(\mathbf{h}_0; y,t)$. Then, due to \eqref{eqn:Drestriction}, it follows that
\begin{equation} \label{eqn:allequal}
    \La(\mathbf{h}_0|_{D}; y,t) = \La(\mathbf{h}_0|_{E}; y,t) = \La(\mathbf{h}_0; y,t) \quad \text{for}\; (y,t) \in K.
\end{equation}
Combining \eqref{eqn:allequal} with \eqref{eqn:Erestconv} and \eqref{eqn:Erestriction} implies Claim (iii).
\end{proof}

For the next two lemmas we will use the notation from Section 15 of \cite{DV}. See the definitions of $a_n, b_n, \ell_n, d_n, d, I_d$ and $\tilde h_n$.
See also assumptions 15.1 and tightness condition 15.4. These notations deal with general planar directed metrics. In the setting of rate 1 exponential last passage percolation, we have $a_n = 0$, $b_n = n^{2/3} = \eps_n^{-1}$, $\ell_n(x,-t;y,-s) = a_n(y-x) + b_n(t-s) = n^{2/3}(t-s)$. We also have
\begin{equation} \label{eqn:d_n}
(d_n)(x,-t;y,-s)=\frac{\eps^{1/2}}{2}L_n(2\eps^{-1}y,\eps^{-3/2}s;2\eps^{-1}x, \eps^{-3/2}t)-\frac{t-s}{\eps}; \quad \eps = n^{-2/3}
\end{equation}
and
\begin{equation} \label{eqn:tildehn}
\tilde h_n(y,t; f_n) = \frac{t}{\eps}- \eps^{1/2} h^n_{2\eps^{-3/2}t}(2\eps^{-1}y;f_n); \quad \eps = n^{-2/3}. 
\end{equation}
We have that $$d(x,-t;y,-s)=\La(y,s;x,t)$$ and assumptions 15.1 and tightness condition 15.4 both hold by \cite[Lemma 16.2]{DV}.

\begin{lem} \label{lem:tightnesscond}
    Let $0 < t_1 < t_2$, $k>0$ and $|y_0|\le k$. Set $q = inf_{|y|\le 2k, t\in [0,t_1]} d(y_0,-t_2;y,-t)$ and assume that $|q| < \infty$.
    Suppose assumptions 15.1 hold. Then as $n\to\infty$, 
    $$ \sup_{|y|\le k, t\in [0,t_1]} \tilde h_n(y,t;f_n) \le \tilde h_n(y_0, t_2; f_n) +1 - q+o(1).$$

    In particular, for exponential last passage percolation, since tightness condition 15.4 holds and $|q| < \infty$ for every $t_1$ and $k$,
    the strengthened tightness condition 15.4 holds as well.
\end{lem}

\begin{proof}
Let $h'\in \mathbb R$. If for some $t\in [0,t_1]$ and $|y|\le k$ we have 
$\tilde h(t,y;f_n)>h',$
then for some $x \in \R$, with $r=(f_n(x)-a_nx)/b_n$ and  with $\ell_n(x,s;x+y,s+t)=a_n y+b_nt$, we have
$$
I_{d_n+\ell_n}(x,r;b_nt)(y) > s; \qquad s =-t-\frac{a_n}{b_n}y+\frac{h'}{b_n}.
$$
This means that there exists $y'$ arbitrarily close to $y$, in particular some $y'\in [-2k,2k]$, so that 
$$
(d_n+\ell_n)(y',s;x,r) > b_nt 
$$
Let $s'=-t_2-\frac{a_n}{b_n}(y_0+y-y')+\frac{h'+(q-1)}{b_n}$. Then 
\begin{align*}
(d_n+\ell_n)(y_0,s';x,r) &\ge (d_n+\ell_n)(y',s;x,r) + (d_n+\ell_n)(y_0,s';y',s) \\&>b_nt+ (d_n+ \ell_n)(y_0,s';y',s)
\\&\ge b_nt + q + o(1) + \ell_n(y_0,s';y',s)
\\&\ge b_nt + q + o(1) + b_n(-t+t_2+(1-q)/b_n)
\\&\ge b_nt_2 +1 +o(1)
\end{align*}
Here we have used that $d_n+ \ell_n$ satisfies the reverse triangle inequality (part of assumption 15.1) and that $d_n$ converges to $d$. Therefore, for all large enough $n$,
$$
I_{d_n+\ell_n}(x,r;b_nt_2)(y_0)\ge s'
$$
and so 
$
\tilde h_n(y_0,t_2;f_n) \ge  h'+q-1$, as required.
\end{proof}

\begin{lem} \label{lem:lpptightness}
    Let $0 < t_1 < t_2$, $k>0$ and $y_0 \in \R$. Set $q = inf_{|y|\le k, t\in [0,t_1]} d(y_0,-t_2;y,-t)$ and assume that $|q| < \infty$.
    Suppose assumptions 15.1 hold. Then as $n\to\infty$, 
    $$ \sup_{|y|\le k, t\in [0,t_1]} \mathbf{g}_n(y,t;f_n) \le \mathbf{g}_n(y_0, t_2; f_n) - q+o(1).$$
    Moreover, with $\tilde h_n$ as in \eqref{eqn:tildehn} and $\eps = \eps_n = n^{-2/3}$, we have that
    $$\mathbf{g}_n(y_0,t_2; f_n) \leq h^{*} \iff \tilde h_n(y_0,t_2 + 2\eps h^{*};f_n) \leq h^{*}.$$
    In particular, for exponential last passage percolation, since tightness condition 15.4 holds and $|q| < \infty$ for every $t_1$ and $k$,
    the tightness condition for LPP that is assumed in the proof of Theorem \ref{thm:lppconvergence} (Claim iii) holds as well .
\end{lem}

\begin{proof}
From \eqref{eqn:gmax} we have that
$$ \mathbf{g}_n(y_0, t_2; f_n) = \sup_{z \in \R}\, \{d_n(y_0,-t_2; z, \eps f_n(z)) + \mathbf{f}_n(z) \}; \; \eps = n^{-2/3},\;
\mathbf{f}_n = - \eps^{1/2} f_n(2\eps^{-1}x).$$
Since $d_n + \ell_n$ satisfies the reverse triangle inequality, we find that
$$(d_n+\ell_n) (y_0, -t_2; z,\eps \mathbf{f}_n(z)) \geq (d_n +\ell_n)(y_0,-t_2; y,-t) + (d_n+\ell_n)(y,-t; z, \eps \mathbf{f}_n(z)),$$
and, since $\ell_n$ is linear, the above simplifies to
$$d_n(y_0, -t_2; z,\eps \mathbf{f}_n(z)) \geq d_n(y_0,-t_2; y,-t) + d_n(y,-t; z, \eps \mathbf{f}_n(z)).$$
This inequality together with the variational representation of $\mathbf{g}_n$ imply that for every $|y| \leq k$ and $t \in [0,t_1]$,
\begin{align*}
    \mathbf{g}_n(y_0,t_2; f_n) &\geq d_n(y_0,-t_2; y,-t) + \mathbf{g}_n(y,t;f_n) \\
    & \geq \inf_{|y|\leq k, t\in [0,t_1]}\, \{d_n(y_0,-t_2; y,-t)\} + \mathbf{g}_n(y,t;f_n).
\end{align*}
Since $d_n$ converges to $d$, $\inf_{|y|\leq k, t\in [0,t_1]}\, d_n(y_0,-t_2; y,-t) = q + o(1)$ as $n \to \infty$. Therefore,
$$ \sup_{|y| \leq k, t \in [0,1]} \mathbf{g}_n(y,t;f_n) \leq \mathbf{g}_n(y_0,t_2;f_n) - q + o(1).$$

Now the relationship between $g_n(x,y;f_n)$ and $h^n_T(x;f_n)$ is that
$$ g_n(x,y; f_n) \leq T \iff h_T(x;f_n) \geq y.$$
In terms of the re-scaled functions this amounts to
$$ \mathbf{g}_n(x,t, f_n) \leq h^* \iff \tilde h_n(x, t + 2\eps h^*;f_n) \leq h^*$$
The tightness condition 15.4 from \cite{DOV} then implies the tightness condition for LPP.
\end{proof}

 \subsection{Scaling limit of the competition interface} \label{sec:compscale}
 In this section we assume to be in the context of Theorem \ref{thm:lppconvergence}; namely,
 we use the coupling of mean 1 exponential random variables $\omega_{x,y,n}$ and the directed landscape $\La$ therein and assume the convergence of last passage times in (iii). We also use the notation therein.
 
Let $(I^n_m)_{m \geq 0}$ be the competition interface for last passage percolation with initial condition $h^n_0$ as defined in \eqref{eqn:compint}. Re-scale it according to
    \begin{equation} \label{eqn:compscaled}
	\mathbf{I}^{n}(t) = (\eps_n/2) \, I^{n}_{\lceil \eps_n^{-3/2}t \rceil} \quad \text{for}\; t > 0.
    \end{equation}
Recall the discrete competition function $d(x,y)$ from \eqref{eqn:discomfun} and let
	$$d^{n}(x,t) = \mathbf{g}_n(x,t; h^{+,n}_0) - \mathbf{g}_n(x,t; h^{-,n}_0)$$
be the re-scaled competition function for the initial condition $h^{n}_0$.
Recall the notation $\mathbf{g}_n(x,t; h^n_0)$ from \eqref{eqn:gmax}. The function is non-decreasing in $x$,
	and the interface $\mathbf{I}^{n}$ satisfies
	\begin{equation} \label{eqn:Id}
		d^{n}(\mathbf{I}^{n}(t), t) < 0 \quad \text{while} \quad d^{n}(\mathbf{I}^{n}(t)+\eps_n,t) > 0.
	\end{equation}
	In the context of Theorem \ref{thm:lppconvergence}, $d^{n}(x,t)$ converges, almost surely, uniformly over compact subsets of $\Hp$ to the limiting competition function
	$$d_0(x,t) = \La(h_0^{+}; x,t) - \La(h_0^{-}; x,t)$$
	from \eqref{eqn:dp}.
	
	Recall from $\S$\ref{sec:defofint} the left and right interfaces $I^{\pm}_0(t)$ from the reference point $p=0$
	in the directed landscape  with initial condition $\mathbf{h}_0$, and that the interface is uniquely defined if
	$I^{\pm}_0(t)$ are finite and $I^{-}_0(t) = I^{+}_0(t)$ for every $t > 0$.
	
\begin{prop} \label{prop:intlimit}
Suppose $d^{n}(x,t)$ converges uniformly on compact subsets of $\Hp$ to $d_0(x,t)$. For every compact interval $[a,b] \subset (0,\infty)$ and $\delta > 0$, there is an $n_0 > 0$ such that if $n > n_0$ then
$$ I^{-}_0(t) - \delta \leq \mathbf{I}^{n}(t) \leq I^{+}_0(t) + \delta \quad \text{for}\;\; t \in [a,b].$$
		
As a consequence, if the interface from reference point 0 for the initial condition $\mathbf{h}_0$ is uniquely defined, then $\mathbf{I}^{\eps}(t)$ converges to it uniformly in $t$ over compact subsets of $(0,\infty)$.
\end{prop}
	
\begin{proof}
Consider the set $K^{+} = (I^{+}_0(t)+\delta, t)$ for $t \in [a,b]$. As $I^{+}_0(t)$ is continuous, this is a compact subset of $\Hp$. The competition function $d_0$ is positive on $K^{+}$ by definition of $I^{+}_0$. Being continuous, $d_0$ has a positive minimum value over $K^{+}$, say $\eta$.
		
The function $d^{n}$ converges to $d_0$ uniformly over $K^{+}$. So there is an $n_1$ such that for $n > n_1$, $d^{n}(x,t) \geq \eta/2$ on $K^{+}$. Consequently, by monotonicity and \eqref{eqn:Id}, $\mathbf{I}^{n}(t) \leq I^{+}_0(t) + \delta$ for every $t \in[a,b]$.

By a parallel argument applied to $K^{-} = ( I^{-}_0(t) - \delta/2, t)$ for $t \in [a,b]$, there is an $n_2$ such that $I^{-}_0(t) - (\delta/2) - \eps_n \leq \mathbf{I}^{\eps}(t)$ for every $t \in [a,b]$, if $n > n_2$.
		
Since $\eps_n \to 0$, there is an $n_3$ such that $\eps_n < \delta/2$ for all $n > n_3$. The proposition follows by taking $n_0 = \max \, \{ n_1, n_2, n_3 \}$.
\end{proof}
	
	The proposition and the aforementioned convergence of last passage times imply
	\begin{cor}[Interface limit] \label{cor:intlimit}
		Suppose the initial height functions $h^{n}_0(x)$ for tasep converge to $\mathbf{h}_0(x)$ under the re-scaling \eqref{eqn:initialconv} in the hypograph topology (along with the growth condition on $h^n_0$ assumed there).
		Suppose also the interface from reference point $p=0$ and initial condition $\mathbf{h}_0$ is uniquely defined in the directed landscape .
		Call this unique interface $I_0(t)$ for $t > 0$.
		The re-scaled competition interface $\mathbf{I}^{n}(t)$ in \eqref{eqn:compscaled} associated to $h^{n}_0$
		converges in law, uniformly in $t$ over compact subsets of $(0, \infty)$, to $I_0(t)$.
	\end{cor}
 
        \subsubsection{The stationary interface} \label{eqn:statint}
        The stationary initial condition $h_0$ for tasep that we consider is the two-sided simple symmetric random walk: $h_0(0)=0$ and $\eps_k=h_0(k+1)-h_0(k)$, $k \in \Z$, are i.i.d random variables taking values $\pm 1$ with probability $1/2$ and independent of the random weights $\omega_{x,y}$ from \eqref{eqn:LPPmax} that constitute the passage times.
        Through linear interpolation, $h_0$ is extended to $x \in \R$.
        (There is a one parameter family of stationary initial conditions according to the drift of the random walk, but here we consider the symmetric case.)

        The KPZ re-scaled competition interface for the stationary initial condition converges in law to the interface $I_0(t)$ of the directed landscape associated to the two-sided Brownian initial condition with diffusivity $\sqrt{2}$. The convergence here is jointly in the randomness over the initial condition and the passage times. We explain this in Proposition \ref{prop:statintlimit} below with the help of the following auxiliary tightness lemma concerning the initial condition.

        \begin{lem}\label{lem:stattightness}
            Let $h_0$ be the stationary initial condition above. For $0 < \eps \leq 1$, consider the random variables
            $$ \sup_{x \in \R} \; \left \{ \eps^{1/2} h_0(\eps^{-1}x) - |x| \right \}.$$
            The laws of these random variables are tight.
        \end{lem}

        \begin{proof}
            Let $S^{\eps}_-$ and $S^{\eps}_+$ be the supremum above restricted to $x \leq 0$ and $x \geq 0$, respectively.
            These two random variables are independent and identically distributed, and the supremum above over all $x \in \R$ equals $\max \{S^{\eps}_-, S^{\eps}_+\}$. It is therefore enough to show that the laws of $S^{\eps}_+$ are tight for $0 < \eps \leq 1$.

            Since $h_0(x)$ is the linear interpolation of its values at integer arguments $x$,
            $$S^{\eps}_+ = \sup_{x = 0, \eps, 2\eps, 3\eps, \ldots} \{ \eps^{1/2} h_0(\eps^{-1}x) - x\}.$$
            Clearly, $S^{\eps}_+ \geq 0$. It is enough to show that for every $0 < \eps \leq 1$, $\lambda > 0$ and $T > 0$,
            \begin{equation} \label{eqn:tightsup}
            \mathbf{Pr} \left ( \sup_{x = 0, \eps, 2\eps, 3\eps, \ldots;\, x \in [0,T]} \{\eps^{1/2} h_0(\eps^{-1}x) - x \} > \lambda \right) \leq e^{-\lambda}.
            \end{equation}

            For $x = 0, \eps, 2\eps, 3\eps, \ldots$, define
            $$ M_{\eps}(x) = e^{\eps^{1/2} h_0(\eps^{-1}x) - c(\eps)x}, \quad c(\eps) = \eps^{-1}\log\left( (e^{\sqrt{\eps}} + e^{-\sqrt{\eps}})/2\right).$$
            The process $M_{\eps}(x)$ is a positive martingale with $\mathbf{E}[M_{\eps}(x)] = 1$ for all $x$. By Doob's maximal inequality,
            $$\mathbf{Pr}\left ( \sup_{x \in [0,T]} M_{\eps}(x) > \lambda \right) \leq \lambda^{-1}.$$
            Upon taking logarithms and setting $\lambda$ to $e^{\lambda}$ gives
            $$\mathbf{Pr}\left ( \sup_{x = 0, \eps, 2\eps, 3\eps, \ldots;\, x \in [0,T]} \{\eps^{1/2} h_0(\eps^{-1}x) - c(\eps)x \} > \lambda \right) \leq e^{-\lambda}.$$

            The bound \eqref{eqn:tightsup} follows on showing that $c(\eps) \leq 1$ if $0 < \eps \leq 1$. Taylor expanding $(e^{\sqrt{\eps}} + e^{-\sqrt{\eps}})/2$ and using the fact that $\log(1+x) \leq x$, we indeed find that
            $$c(\eps) \leq \frac{1}{2!} + \frac{1}{4!} \eps + \frac{1}{6!} \eps^2 + \cdots \leq 1.$$

        \end{proof}

        \begin{prop} \label{prop:statintlimit}
            Let $h_0$ be the stationary initial condition and $(I_m)_{m \geq 0}$ be the competition interface for it defined via \eqref{eqn:compint}.
            For any sequence $\eps_n \to 0$, define $\mathbf{I}^{n}(t) = (\eps/2)I_{\lceil \eps^{-3/2} t \rceil}$ for $t \geq 0$ and $\eps = \eps_n$.
            Let $I_0(t)$ be the (uniquely defined) interface from reference point $p=0$ of the directed landscape associated to a two-sided Brownian initial condition with diffusivity constant $\sqrt{2}$. The process $\mathbf{I}^{n}(t)$ converges to $I_0(t)$ in law, uniformly over compact subsets of $t \in (0,\infty)$.
        \end{prop}

        \begin{proof}
            The functions $- \eps^{1/2}h_0(2 \eps^{-1}x)$ have the same law as $\eps^{1/2}h_0(2 \eps^{-1}x)$ for every $\eps = \eps_n$.
            By Lemma \ref{lem:stattightness}, there are coupled functions $h^{n}(x)$ and a finite random variable $C$
            such that $h^{n}(x)$ has the law of $-\eps_n^{1/2}h_0(2\eps_n^{-1}x)$ and $h^{n}(x) \leq C + 2|x|$.
            By Donsker's theorem and Skorokhod's representation theorem, we may also assume that there is a two-sided Brownian motion $B(x)$
            with diffusivity constant $\sqrt{2}$ coupled with the $h^{n}$ such that $h^{n}$ converges to $B$ uniformly on compact subsets of $x \in \R$, almost surely.

            Suppose we are in the context of Theorem \ref{thm:lppconvergence} with the coupling mentioned there (which is independent of the coupling of initial conditions assumed above). Consider the KPZ re-scaled last passage times $\mathbf{g}_n(x,t; h^n)$ associated to $h^n$ according to \eqref{eqn:gmax} 
            ($h^n$ is already re-scaled, so $\mathbf{h}^n_0$ in \eqref{eqn:gmax} is $h^n$ from here).
            Theorem \ref{thm:lppconvergence}(iii) implies that $\mathbf{g}_n(x,t; h^n)$ converges uniformly over compact subsets of $(x,t) \in \Hp$ to the limit $\La(B; x,t)$ as $n \to \infty$, almost surely. The limit $\La(B; x,t) = \sup_{y \in \R} \{B(y) + \La(y,0;x,t)\}$, which has the law of the directed landscape height function started from a two-sided Brownian initial condition of diffusivity $\sqrt{2}$. Similarly, the split last passage times $\mathbf{g}_n(x,t; h^{\pm,n})$ converge accordingly to $\La(B^{+}; x,t) = \sup_{y \geq 0} \{B(y) + \La(y,0;x,t)\}$ and
            $\La(B^{-}; x,t) = \sup_{y \leq 0} \{B(y) + \La(y,0;x,t)\}$. Thus, the re-scaled competition function
            $d^{n}(x,t)$ converges to $d_0(x,t) = \La(B^+; x,t) - \La(B^-;x,t)$.

            Proposition \ref{prop:intlimit} now implies $\mathbf{I}^{n}(t)$ converges to $I_0(t)$ in law, uniformly over compact subsets of $t \in (0,\infty)$.
        \end{proof}
        
\subsection{Scaling limit of the 2nd class particle} \label{sec:particlescale}
We can now prove Theorem \ref{thm:2ndclass}. Assume once again that we are in the context of Theorem \ref{thm:lppconvergence}. For every $n$, tasep with initial condition $h_0^n$ is defined by using the random variables $\omega_{x,y,n}$ and the bijective mapping between last passage percolation and tasep explained in Section \ref{sec:2ndandcomp}.
 
Recall the correspondence between the 2nd class particle and the $\circ-\bullet$ pair from Section \ref{sec:holeparticle}. We must first understand the random times when the $\circ-\bullet$ pair jumps. Recall from \eqref{eqn:2ndnstep} that $X_m$ is the position of the particle in the $\circ-\bullet$ pair after its $m$-th jump, and $\tau_m$ is the time of this jump. Lemma \ref{lem:2ndLPP} describes the evolution of $X_m$ and $\tau_m$. Let $m_t$ be the integer such that $t \in [\tau_{m_t}, \tau_{m_t+1})$, which means that the position $X(t)$ of the particle in the $\circ-\bullet$ pair satisfies
\begin{equation} \label{eqn:tandnt}
    X(t) = X_{m_t} \quad \text{for every}\; t.
\end{equation}
	
\begin{lem} \label{lem:jumptimes}
Assume the context of Theorem \ref{thm:lppconvergence} and the conditions of Theorem \ref{thm:2ndclass}: tasep initial conditions $h^{n}_0$ converge to $\mathbf{h}_0$ under the re-scaling \eqref{eqn:initialconv} in the hypograph topology and with the assumed growth condition. Assume there is a uniquely defined interface from reference point 0 for the initial condition $\mathbf{h}_0$.
		
Denote by $X^{n}_m$ , $\tau^{n}_m$ and $m^{n}_t$ the quantities $X_m$, $\tau_m$ and $m_t$ above for the tasep with initial condition $h^{n}_0$.  Given a compact interval $[a,b] \subset (0,\infty)$, the following holds uniformly for $s \in [a,b]$ for all large $n$. There is a finite random variable $C$ such that
$$ \sup_{s \in [a,b]}\; \left | \frac{m^{n}_{\eps^{-3/2}s}}{\eps^{-3/2}s} - \frac{1}{2} \right | \leq C \eps \quad \text{with}\; \eps = \eps_n.$$
\end{lem}
	
\begin{proof}
Fix $n$ (hence the initial condition $h^n_0$ and $\eps_n$) and to save notation suppress the $n$-dependence of particle locations, jump times, last passage times, etc. Recall from Lemma \ref{lem:2ndLPP} that $\tau_m = g(X_m,m)$ and
$$\tau_{m+1} = g(X_{m+1},m+1) = \min \{ g(X_m+1,m+1), g(X_m-1,n+1)\} \leq g(X_m, m+2).$$
Since $\tau_{m_t} \leq t \leq \tau_{m_t + 1}$, it follows that
\begin{equation} \label{eqn:tbound}
    g(X_{m_t}, m_t) \leq t \leq g(X_{m_t}, m_t+2).
\end{equation}
		
Now $X_m \in \{I_m, I_m + 2\}$ for every $m$ by Lemma \ref{lem:compint}. The processes $(I^n_m)_{m \geq 0}$ re-scaled according to \eqref{eqn:compscaled} converges to the uniquely defined interface $I_0(t)$ from reference point 0 for the initial condition $\mathbf{h}_0$, as implied by Proposition \ref{prop:intlimit}. So there is a random variable $L$ (using the coupling in the context of Theorem \ref{thm:lppconvergence}) such that $|I_m| \leq L m^{2/3}$ for every $m = m_t$ with $t = s\eps^{-3/2}$ and $s \in [a,b]$. Thus $X_m$ satisfies the same bound, under this coupling, for the same range of $m$.
		
From the convergence of the KPZ re-scaled last passage function $\mathbf{g}$ in \eqref{eqn:gmax} to the directed landscape height function $\La(\mathbf{h}_0;\cdot)$ as discussed in Theorem \ref{thm:lppconvergence} (the dependence on $n$ is suppressed here), it is easy to see that if $|x| \leq Lm^{2/3}$ then there is a random variable $C'$ such that
$$g(x, m) = 2m + C(x,m) m^{1/3} \quad \text{with}\;\; |C(x,m)| \leq C'.$$
Applying with to $x = X_{m_t}$ and then using \eqref{eqn:tbound} implies that
\begin{equation} \label{eqn:tbound2}
    |t - 2m_t| \leq C' m_t^{1/3}
\end{equation}
for every $t = s\eps^{-3/2}$ with $s \in [a,b]$. Now $t \leq (C'+2)m_t$, which shows that $m_t \geq c \eps^{-3/2}$ for every $t$ as before and a random $c$. As a result, one deduces from \eqref{eqn:tbound2} that
$$ \left | \frac{t}{m_t} - 2 \right | \leq C'' \eps$$
for a random variable $C''$. The estimate in the lemma follows after taking reciprocals.
\end{proof}
	
Now we complete the proof of Theorem \ref{thm:2ndclass}.
\begin{proof}
Consider $X^{n}_m$, the position of the particle in the $\circ-\bullet$ pair after $m$ steps for tasep with initial height function $h^{n}_0$. Since $|X^{n}_r - X^{n}_s| \leq |r-s|$,
$$|X^{n}_{m_t} - X^{n}_{\lceil t/2 \rceil}| \leq |m_t - \lceil t/2 \rceil|.$$
Plugging in $t = s \eps^{-3/2}$ for $s \in [a,b] \subset (0,\infty)$, with $[a,b]$ compact, and using Lemma \ref{lem:jumptimes},
$$ \left |X^{n}_{m_{s \eps^{-3/2}}} - X^{n}_{\lceil s \eps^{-3/2}/2 \rceil} \right | \leq C \eps^{-1/2}.$$
Here $C$ is the random variable from the aforementioned lemma.
		
Since $X^{n}(t) = X^{n}_{m_t}$, $\mathbf{X}^{n}(s) = (\eps/2) X^{\eps}_{m_{s \eps^{-3/2}}}$ with $\eps = \eps_n$.
Multiply the display above by $\eps/2$ and, writing $\overline{X}^{n}(s) = (\eps/2) X^{n}_{\lceil s \eps^{-3/2} \rceil}$, it becomes
$$ \left | \mathbf{X}^{n}(s) - \overline{X}^{n}(s/2)\right | \leq (C/2) \eps^{1/2}$$
for every $s \in [a,b]$ and $\eps = \eps_n$.
		
Finally, recall $\mathbf{I}^{n}(s) = (\eps_n/2)I^{n}_{\lceil s \eps_n^{-3/2} \rceil}$. Lemma \ref{lem:compint} implies
$|\overline{X}^{n}(s/2) - \mathbf{I}^{n}(s/2)| \leq \eps_n$. So,
$$| \mathbf{X}^{n}(s) - \mathbf{I}^{n}(s/2)| \leq \eps_n + (C/2) \eps_n^{1/2}.$$
The convergence of $\mathbf{X}^{n}(s)$ to $I_0(s/2)$ follows from the convergence of $\mathbf{I}^{n}(t)$ to $I_0(t)$ as stated in Corollary \ref{cor:intlimit}.
\end{proof}

	\subsection{Questions}
	
	\paragraph*{\textbf{Multiple 2nd class particles}}
	Consider tasep with more than one 2nd class particle. What is their joint scaling limit?
	It is no longer possible to model multiple 2nd class particles by using hole--particle pairs. Ideally, there should be a
	framework to study the joint scaling limits of 1st, 2nd, 3rd and so on class particles over suitable initial conditions.
	
	\paragraph*{\textbf{Universality of the 2nd class particle}}
	There are several particle systems in the KPZ universality class with a notion of a 2nd class particle,
	for instance, the general asymmetric simple exclusion process. Is the scaling limit of the 2nd class particle,
	namely the competition interface, universal for these models?
	
	An intriguing perspective would be to define the competition interface in terms of an SDE, mirroring the notion
	of characteristics in Burgers' equation. It may then be possible to define a 2nd class
	particle in models like the KPZ equation, and to show that it converges to the competition interface in the KPZ scaling limit.
	
	\paragraph*{\textbf{2nd class particles and characteristics}}
	Here is another question about where a 2nd class particle finds its characteristic, based on discussions with Jeremy Quastel.
	Suppose tasep has initial density $\rho_{-} = 1$ and $\rho_{+}=0$, so initially there are particles at all the negative sites.
	Consider a 2nd class particle started from $X(0) = x_t$. If $x_t$ is large enough then $X(t)$ will be outside the rarefaction fan.
	What is the smallest $x_t$, as $t$ tends to infinity, for which the 2nd class particle stays outside the fan? Formally,
	find the smallest positive sequence $x_t$ such that for every $\eps > 0$,
	$$ \lim_{t \to \infty} \mathbf{Pr} \left (X(t) \leq (1-\eps)t \mid X(0) = x_t \right) = 0.$$

	\subsection*{Acknowledgements}
	B.V. is partially supported by the NSERC Discovery accelerator grant. He thanks Duncan Dauvergne for several useful discussions. 
	M.R. thanks Jeremy Quastel for some early discussions which provided the impetus for this project.
	We thank Jim Pitman for explaining to us the concave majorant of Brownian motion.
        We also thank our three referees for extremely thoughtful and detailed reports (21 pages total) that helped us improve this paper greatly!

\end{document}